\documentclass[12pt,a4paper]{article}

\usepackage[british]{babel}

\usepackage[T1]{fontenc}
\usepackage{tikz-cd}

\usepackage[utf8]{inputenc}

\usepackage{csquotes}

\usepackage[backend=biber, maxbibnames=10, maxcitenames=4, maxalphanames=4, bibencoding=utf8, style=alphabetic, isbn=false, url=false, doi=true, giveninits=true]{biblatex}

\renewbibmacro{in:}{}
\renewbibmacro*{doi+eprint+url}{%
    \printfield{doi}%
    \newunit\newblock%
    \iftoggle{bbx:eprint}{%
        \usebibmacro{eprint}%
    }{}%
    \newunit\newblock%
    \iffieldundef{doi}{%
        \usebibmacro{url+urldate}}%
    {}%
}
\AtEveryBibitem{\clearfield{month}}
\AtEveryBibitem{\clearfield{note}}

\usepackage{setspace}

\usepackage[sc]{mathpazo}

\usepackage[indentafter]{titlesec}
\titleformat{name=\section}{}{\thesection.}{0.8em}{\centering\scshape}
\titleformat{name=\subsection}[runin]{}{\thesubsection.}{0.5em}{\bfseries}[.]
\titleformat{name=\subsubsection}[runin]{}{\thesubsubsection.}{0.5em}{\itshape}[.]
\titleformat{name=\paragraph,numberless}[runin]{}{}{0em}{}[.]
\titlespacing{\paragraph}{0em}{0em}{0.5em}
\titleformat{name=\subparagraph,numberless}[runin]{}{}{0em}{}[.]
\titlespacing{\subparagraph}{0em}{0em}{0.5em}

\usepackage{xurl}

\usepackage{mathtools}
\providecommand{\coloneqq}{\mathrel{\mathop:}=}
\providecommand{\eqqcolon}{=\mathrel{\mathop:}}

\DeclarePairedDelimiter\abs{\lvert}{\rvert}%
\DeclarePairedDelimiter\norm{\lVert}{\rVert}%

\makeatletter
\let\oldabs\abs
\def\abs{\@ifstar{\oldabs}{\oldabs*}}
\let\oldnorm\norm
\def\norm{\@ifstar{\oldnorm}{\oldnorm*}}
\makeatother

\usepackage{amssymb}

\usepackage{amsthm}

\usepackage[scr=boondoxo]{mathalpha}

\usepackage{dsfont}

\usepackage{enumitem}
\setlist[enumerate]{noitemsep, partopsep=0pt, topsep=0pt, parsep=0pt, itemsep=0pt}
\setlist[itemize]{noitemsep, partopsep=0pt, topsep=0pt, parsep=0pt, itemsep=0pt}

\usepackage{interval}

\intervalconfig{soft open fences}

\usepackage[normalem]{ulem}

\usepackage[stretch=10]{microtype}

\usepackage[affil-it, noblocks]{authblk}

\usepackage[dvipsnames]{xcolor}

\usepackage[draft=false]{hyperref}
\hypersetup{
    colorlinks = true, 
    urlcolor   = blue, 
    linkcolor  = blue, 
    citecolor  = ForestGreen 
}
\newsavebox{\linkunderlinebox}
\newcommand{\linkunderline}[1]{%
    \begingroup
    \sbox{\linkunderlinebox}{#1}%
    \leavevmode
    \rlap{\raisebox{-1.6pt}[0pt][0pt]{\rule{\wd\linkunderlinebox}{0.35pt}}}%
    \usebox{\linkunderlinebox}%
    \endgroup
}
\let\templatehref\href
\renewcommand{\href}[2]{\templatehref{#1}{\linkunderline{#2}}}
\ExplSyntaxOn
\NewDocumentCommand{\bibbreakablehref}{mm}{%
    \tl_map_inline:nn {#2} {%
        \templatehref{#1}{\linkunderline{##1}}%
        \penalty\UrlBreakPenalty
    }%
}
\ExplSyntaxOff
\DeclareFieldFormat{doi}{%
    \mkbibacro{DOI}\addcolon\space
    \bibbreakablehref{https://doi.org/#1}{#1}}
\DeclareFieldFormat{url}{%
    \mkbibacro{URL}\addcolon\space
    \bibbreakablehref{#1}{#1}}
\renewbibmacro*{cite}{%
    \printtext[bibhyperref]{%
        \linkunderline{%
            \textbf{%
                \printfield{labelprefix}%
                \printfield{labelalpha}%
                \printfield{extraalpha}%
                \ifbool{bbx:subentry}
                {\printfield{entrysetcount}}
                {}}}}}

\usepackage{aliascnt}
\usepackage[nameinlink,noabbrev,capitalize]{cleveref}
\crefname{equation}{}{}

\usepackage[plain]{fullpage}

\newlist{theoenum}{enumerate}{1} 
\setlist[theoenum]{label=\normalfont(\roman*), ref=\theproposition~\normalfont(\roman*), noitemsep, partopsep=0pt, topsep=0pt, parsep=0pt, itemsep=0pt}
\crefalias{theoenumi}{theorem}

\usepackage{tocloft}

\newcommand{\startappendix}{%
    \appendix
    \crefalias{section}{appsec}
    \titleformat{name=\section}{}{\appendixname~\thesection.}{0.8em}{\centering\scshape}
    \addtocontents{toc}{\protect\renewcommand{\protect\cftsecpresnum}{\appendixname~}}
    \addtocontents{toc}{\protect\renewcommand{\protect\cftsecaftersnum}{.\space}}
    \addtocontents{toc}{\protect\setlength{\protect\cftsecnumwidth}{6.7em}}
    \numberwithin{figure}{section}
}

\usepackage{titlefoot}

\usepackage{todonotes}

\usepackage{tikz}
\usepackage{amsmath}
\usepackage{graphicx}
\usetikzlibrary{arrows.meta, decorations.pathreplacing}
\usepackage{esint}

\usepackage{upgreek}
\usepackage{subcaption}
\usepackage{float}

\theoremstyle{plain}
\newtheorem{theorem}{Theorem}[section]
\newtheorem*{theorem*}{Theorem}
\crefname{theorem}{Theorem}{Theorems}
\Crefname{theorem}{Theorem}{Theorems}

\newaliascnt{proposition}{theorem}
\newtheorem{proposition}[proposition]{Proposition}
\aliascntresetthe{proposition}
\crefname{proposition}{Proposition}{Propositions}
\Crefname{proposition}{Proposition}{Propositions}

\newaliascnt{corollary}{theorem}

\aliascntresetthe{corollary}
\crefname{corollary}{Corollary}{Corollaries}
\Crefname{corollary}{Corollary}{Corollaries}

\newaliascnt{conjecture}{theorem}

\aliascntresetthe{conjecture}
\crefname{conjecture}{Conjecture}{Conjectures}
\Crefname{conjecture}{Conjecture}{Conjectures}

\newaliascnt{lemma}{theorem}

\aliascntresetthe{lemma}
\crefname{lemma}{Lemma}{Lemmas}
\Crefname{lemma}{Lemma}{Lemmas}
\newtheorem*{lemma*}{Lemma}

\theoremstyle{definition}
\newaliascnt{definition}{theorem}
\newtheorem{definition}[definition]{Definition}
\aliascntresetthe{definition}
\crefname{definition}{Definition}{Definitions}
\Crefname{definition}{Definition}{Definitions}

\theoremstyle{remark}
\newaliascnt{example}{theorem}
\newtheorem{example}[example]{Example}
\aliascntresetthe{example}
\crefname{example}{Example}{Examples}
\Crefname{example}{Example}{Examples}

\newaliascnt{remark}{theorem}
\newtheorem{remark}[remark]{Remark}
\aliascntresetthe{remark}
\crefname{remark}{Remark}{Remarks}
\Crefname{remark}{Remark}{Remarks}
\newtheorem*{remark*}{Remark}

\newaliascnt{appsec}{section}
\aliascntresetthe{appsec}
\crefname{appsec}{Appendix}{Appendices}
\Crefname{appsec}{Appendix}{Appendices}

\newcommand*\diff{\mathop{}\!\mathrm{d}}
\newcommand{\R}{\mathbb{R}}
\newcommand{\E}{\mathrm{E}}

\newcommand{\vol}{\mathrm{vol}}
\newcommand{\Ric}{\mathrm{Ric}}
\newcommand{\Leb}{\mathscr{L}}
\newcommand{\supp}{\mathrm{supp}}

\renewcommand{\exp}{\mathrm{exp}}
\renewcommand{\epsilon}{\varepsilon}

\DeclareMathOperator{\T}{\mathrm{T}}

\let\originalleft\left
\let\originalright\right
\renewcommand{\left}{\mathopen{}\mathclose\bgroup\originalleft}
\renewcommand{\right}{\aftergroup\egroup\originalright}

\title{%
  {\large\bfseries\MakeUppercase{Ollivier--Ricci curvature for Causal Sets}}%
}

\author{%
  Joe Barton%
  \protect\footnotemark[1]%
  \protect\footnotemark[2]%
  \and
  Samuël Borza%
  \protect\footnotemark[1]%
  \protect\footnotemark[3]%
  \and
  Jona Röhrig%
  \protect\footnotemark[1]%
  \protect\footnotemark[4]%
}

\date{}

\makeatletter
\def\@maketitle{%
  \newpage
  {\centering
    \@title\par
    \vskip 1.5em%
      {\large
        \lineskip .5em%
        \begin{tabular}[t]{c}%
          \@author
        \end{tabular}\par}%
    \ifx\@date\@empty\else
      \vskip 1em%
        {\large \@date}%
    \fi
    \par}%
  \vskip 1.5em%
}

\let\template@maketitle\maketitle
\renewcommand{\maketitle}{%
  \begingroup
  \renewcommand{\thefootnote}{\fnsymbol{footnote}}
  \long\def\@makefntext##1{%
    \parindent 0pt%
    \noindent\makebox[0pt][r]{\@makefnmark\,}##1%
  }%
  \template@maketitle
  \let\orig@makefnmark\@makefnmark
  \insert\footins{%
    \reset@font\footnotesize
    \interlinepenalty\interfootnotelinepenalty
    \splittopskip\footnotesep
    \splitmaxdepth\dp\strutbox
    \floatingpenalty\@MM
    \hsize\columnwidth
    \@parboxrestore
    \parindent 0pt%
    \noindent Date: \today.\par
  }%
  \let\@makefnmark\orig@makefnmark
  \footnotetext[1]{Faculty of Mathematics, University of Vienna, Oskar-Morgenstern-Platz 1, 1090 Vienna, Austria}%
  \def\@makefnmark{}%
  \footnotetext[2]{\textit{E-mails}: %
    \orig@makefnmark\href{mailto:joe.barton@univie.ac.at}{\nolinkurl{joe.barton@univie.ac.at}};\,
    {\let\@makefnmark\orig@makefnmark \footnotemark[3]}%
    \href{mailto:samuel.borza@univie.ac.at}{\nolinkurl{samuel.borza@univie.ac.at}};\,
    {\let\@makefnmark\orig@makefnmark \footnotemark[4]}%
    \href{mailto:jona.roehrig@univie.ac.at}{\nolinkurl{jona.roehrig@univie.ac.at}}}%
  \endgroup
}
\makeatother

\begin{document}

\maketitle              

\providecommand{\keywords}[1]
{
	\par\noindent\textbf{\textit{Keywords---}} #1\par
}

\providecommand{\msc}[1]
{
	\noindent\textbf{\textit{MSC (2020)---}} #1\par
}

\begin{abstract}
	We introduce a novel notion of Ollivier--Ricci curvature for causal sets using Lorentzian optimal transport. The construction is motivated by a new Lorentzian asymptotic formula of independent interest, which recovers timelike Ricci curvature, up to higher-order terms, from the transport distance between probability measures on nearby causal diamonds. Passing to the discrete setting, this leads to a mesoscopic notion of Ricci curvature defined along maximal chains and built from probability measures on causal diamonds. We study several variants, including idle and Lin--Lu--Yau type curvatures, prove local-to-global propagation results and timelike Bonnet--Myers theorems, and compute the curvature for a range of explicit causal sets. We design high-density Poisson sprinkling numerical experiments recovering the expected constant-curvature signatures of Minkowski, de Sitter, and anti-de Sitter space. These results provide evidence that the construction captures timelike Ricci curvature from order-theoretic data.
\end{abstract}


\msc{53C50, 83C27, 49Q22, 06A07}

{\renewcommand{\contentsname}{\large Contents}%
  \small
  \tableofcontents
}

\section{Introduction}

\label{Introduction}

Causal set theory is an approach to quantum gravity in which spacetime is assumed to be discrete. The universe is modelled as a set of spacetime events, equipped with a partial order that represents the causal relation between them, such that the number of events in any causal diamond is finite. The idea that causality is central to spacetime geometry has a long history and arguably appears already in the work of Weyl and Lorentz. While the first axiomatic framework for causal spaces was given by Kronheimer and Penrose \cite{KronheimerPenrose}, causal set theory, as understood nowadays in the literature and as followed in this work, was introduced by Bombelli, Lee, Meyer and Sorkin in their foundational paper \cite{BombelliSorkin}.

The programme of causal sets is often motivated by the Hawking--King--McCarthy--Malament theorem \cite{Malament1977,Zeeman1964,HawkingKing1976}, which states that if there is a bijective map between two past- and future-distinguishing spacetimes that preserves their causal structure, then the map is a conformal isomorphism. The essence of Lorentzian geometry can therefore be viewed as the conjunction of its causal structure and its volume element. Adapted to the discrete setting, this philosophy leads to the now famous ``slogan of causal set theory''
\begin{center}
  Order + Number $\cong$ Spacetime Geometry.
\end{center}
Here, ``order'' represents the partial order relation of a causal set, while ``number'' indicates that the volume of a region of a discrete spacetime is obtained simply by counting the number of elements it contains.

Over the past decades, there have been numerous results in causal set theory, and it is difficult to do justice to it in just a few lines. One striking achievement is Sorkin's order-of-magnitude prediction of the cosmological constant \cite{Sorkin1991}, in agreement with current experimental data. Another important theoretical development is the construction in \cite{sumati2007} of a topology on causal sets using thickened antichains, together with the result that one can recover the homology of a globally hyperbolic spacetime from a faithfully embedded causal set at sufficiently high sprinkling density. For a recent overview of causal set theory, including an extensive bibliography and historical notes, we refer the interested reader to \cite{sumati2025}.

Causal set theory is not without open problems. The \emph{Hauptvermutung} conjectures that one causal set cannot be a Poisson point process approximation of two macroscopically different spacetimes with high probability.
Furthermore, quantum mechanics will have to be reformulated \cite{sorkin1994}, and gravity adapted in terms of a discrete structure.

It is this last point that highlights the need for a discrete notion of curvature in order to formulate the causal set analogue of Einstein's equation. In \cite{Fay2010,Fay2013}, the authors introduce a family of retarded, Lorentz-invariant, nonlocal d'Alembertian operators on scalar fields over causal sets. For causal sets approximated by curved spacetime, these operators approximate $\square - \tfrac{1}{2}R$, where $R$ is the Ricci scalar. This provides a way to define a causal set analogue of scalar curvature.

The present work develops a new Ricci theory for causal sets, motivated by the development of curvature via optimal transport in both metric measure spaces and Lorentzian non-smooth structure, and by the incredible success that Ollivier's notion of Ricci curvature has had on combinatorial graphs.

Since the seminal works \cite{S-ActaI,S-ActaII,LV-Ricci}, optimal transport has become a powerful tool for studying curvature on non-smooth spaces of ``positive signature''. More recently, a growing body of work has sought to adapt these ideas to non-smooth Lorentzian geometry \cite{Kunzinger2018,minguzzi2019lorentzian,McCann4,MondinoSuhr2023,CavMondino2024,octet}. In this context, it is natural to attempt to define a new notion of curvature for causal sets using methods from optimal transport.

This approach is also supported by nearly two decades of active research on optimal-transport-based curvature on combinatorial graphs. In the foundational work \cite{Ollivier2009}, Ollivier introduced the notion of curvature that provides the inspiration for this work. Given two points on a graph, Ollivier curvature compares the original distance between the two points with the optimal transport distance between the uniform measures on small balls around them. Positive curvature means that the optimal transport distance between the balls is smaller than the distance between their centres; negative curvature means that it is larger. Ollivier--Ricci curvature has been successful in turning the classical idea of Ricci curvature into a flexible, transport-based notion that works naturally on discrete metric spaces \cite{Lin2011,Liu2014,Munch2019,Bourne2018,Cushing2020}. The original motivation behind Ollivier's definition of discrete curvature comes from a result in Riemannian geometry \cite[Example 7, Section 8]{Ollivier2009}, which we reinterpret in the Lorentzian setting by replacing metric balls with causal diamonds, see \Cref{fig:3DOll}.

\begin{figure}[ht]
  \centering
  \includegraphics[width=0.95\textwidth]{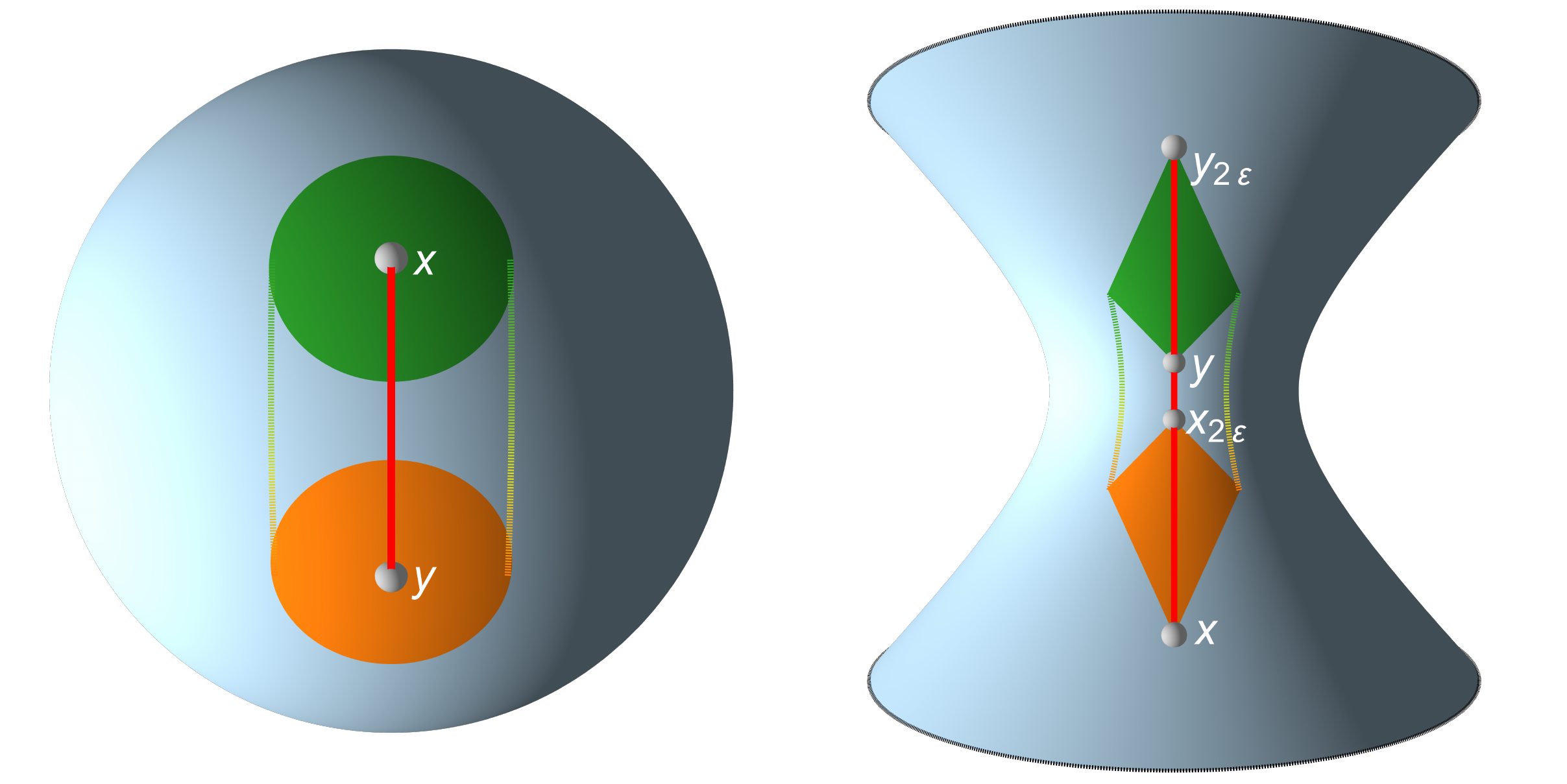}
  \caption{Riemannian and Lorentzian version of Ollivier Ricci--curvature}
  \label{fig:3DOll}
\end{figure}

This is the continuum mechanism behind the paper: a purely transport-theoretic comparison of nearby causal diamonds detects the timelike Ricci tensor. More precisely, let $(M,g)$ be an $n$-dimensional globally hyperbolic Lorentzian manifold with time-separation function $\uptau$, $x \in M$, $v \in \T_x(M)$ be a future-directed timelike vector, and for $0 < 2\epsilon < \delta$, set
\[
  y \coloneqq \exp_x(\delta v), \qquad x_{2 \epsilon} \coloneqq \exp_x(2\epsilon v), \quad \text{and} \quad y_{2 \epsilon} \coloneqq \exp_x\left((2\epsilon+\delta) v\right).
\]
We prove in \cref{section:smooththeorem} (see \cref{thm:smoothTheorem}) that if $\mu_x$ denotes the uniform measure on the causal diamond $J(x,x_{2\epsilon})$ and $\mu_y$ denotes the uniform measure on $J(y,y_{2\epsilon})$, then, as $\epsilon,\delta \to 0$, the $\ell_1$-Lorentzian optimal transport distance satisfies
\begin{equation}
  \label{eq:MainTheorem}
  \ell_1(\mu_x, \mu_y)
  =
  \delta \left(
  1 + \frac{\epsilon^2}{2}\frac{n}{(n+1)(n+2)}
  \mathrm{Ric}(v,v)
  + O(\epsilon^3+\epsilon^2\delta)
  \right).
\end{equation}
Thus the first correction to the transport distance between nearby causal diamonds is governed by the Ricci curvature in the timelike direction $v$, see \cref{fig:smooththeorem}.

\begin{figure}[ht]
  \centering
  \includegraphics[width=0.9\textwidth, trim=0.5cm 16.5cm 2cm 4.5cm, clip]{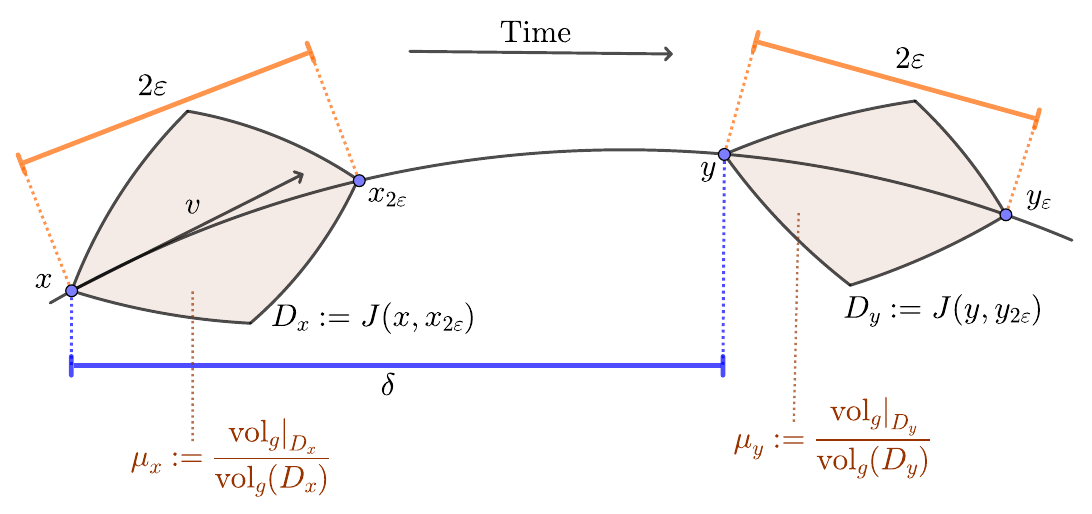}
  \caption{The two causal diamonds used in the asymptotic expansion of $\ell_1(\mu_x,\mu_y)$, where $\mu_x$ and $\mu_y$ are the normalised volume measures on $D_x$ and $D_y$.}
  \label{fig:smooththeorem}
\end{figure}

Passing to the causal set setting, let $\gamma=(x_0,\ldots,x_{n+r})$ be a maximal chain, that is to say a geodesic for the discrete time-separation function $\uptau$, and set $x=x_0$ and $y=x_n$. The parameter $r$ fixes the size of the discrete causal diamonds around these points: $\mu_x$ is supported on $J(x_0,x_r)$, while $\mu_y$ is supported on $J(x_n,x_{n+r})$. Isolating the Ricci term in \cref{eq:MainTheorem} and neglecting the higher-order terms then leads to
\begin{equation}
  \label{eq:ORCST}
  \kappa_r(\gamma)\coloneqq\frac{\ell_1(\mu_x,\mu_y)}{\uptau(x,y)}-1,
\end{equation}
which we define as the \emph{Ollivier--Ricci curvature} of the causal set. For the purposes of this work, we focus mainly on the case where the measures are uniform on causal diamonds, while also considering $\alpha$-idleness and the LLY construction as in \cite{Lin2011}.

The coarse Ricci curvature for causal sets \cref{eq:ORCST} differs from Ollivier's original formula in several important ways. First, Ollivier--Ricci curvature on combinatorial graphs is usually defined using optimal transport between balls of radius one. Such balls form the smallest non-trivial neighbourhoods, and comparing probability measures on them gives a genuinely local notion of curvature. By contrast, we show in \cref{prop:r1,prop:r2} that, in the causal set setting, one must choose causal diamonds of sufficiently large radius, say at least $r \geq 3$, in order to detect curvature.

Furthermore, on combinatorial graphs Ollivier--Ricci curvature is usually evaluated along an edge, so $x$ and $y$ are typically neighbouring vertices with $d(x,y)=1$, see however \cite{LongScale} for variants involving non-adjacent vertices. This is again unsuitable in the causal set setting, where the geodesic $\gamma$ joining $x$ and $y$ must be sufficiently long for the relevant intervals to contain enough elements. If $\gamma$ is too short, the transport distance is dominated by discreteness effects rather than by the large-scale geometry one aims to approximate. In some sense, the causal set curvature defined here is inherently \emph{mesoscopic}.

These properties will be discussed in detail in \cref{subsection:ORCST}: the definitions of the three variants, their local-to-global propagation properties, and the corresponding timelike Bonnet--Myers theorems. The Ollivier--Ricci curvature \cref{eq:ORCST} is then illustrated by examples in \cref{subsection:explicitexamples}, including complete layered causal sets, product causal sets, and Coxeter groups.

Furthermore, in \cref{subsection:numerical} we present a series of numerical experiments showing that the curvature notion introduced here behaves as expected on Poisson sprinklings of Lorentzian manifolds. For instance, for sufficiently large sprinkling density and sufficiently small $\epsilon>0$, the computed curvature approaches the expected values for the three constant-curvature geometries, namely Minkowski, de Sitter, and anti-de Sitter space; see \cref{fig:numerical1}. The code used to generate the numerical experiments and figures is available in the companion GitHub repository \url{https://github.com/samuelborza/Ollivier-Ricci-CST}.

\begin{figure}[!htbp]
  \centering
  \begin{subfigure}{0.48\textwidth}
    \centering
    \includegraphics[width=\linewidth]{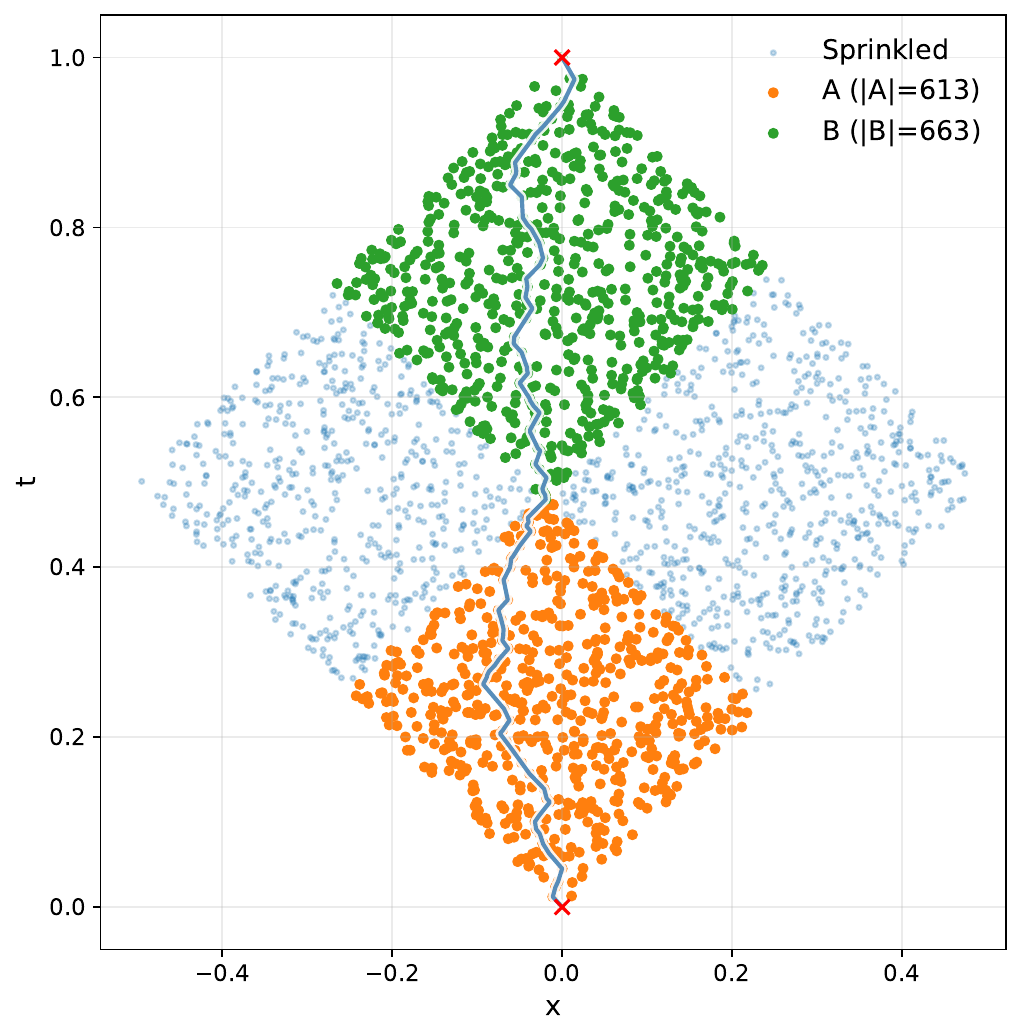}
    \caption{Minkowski space}
  \end{subfigure}
  \hfill
  \begin{subfigure}{0.48\textwidth}
    \centering
    \includegraphics[width=\linewidth]{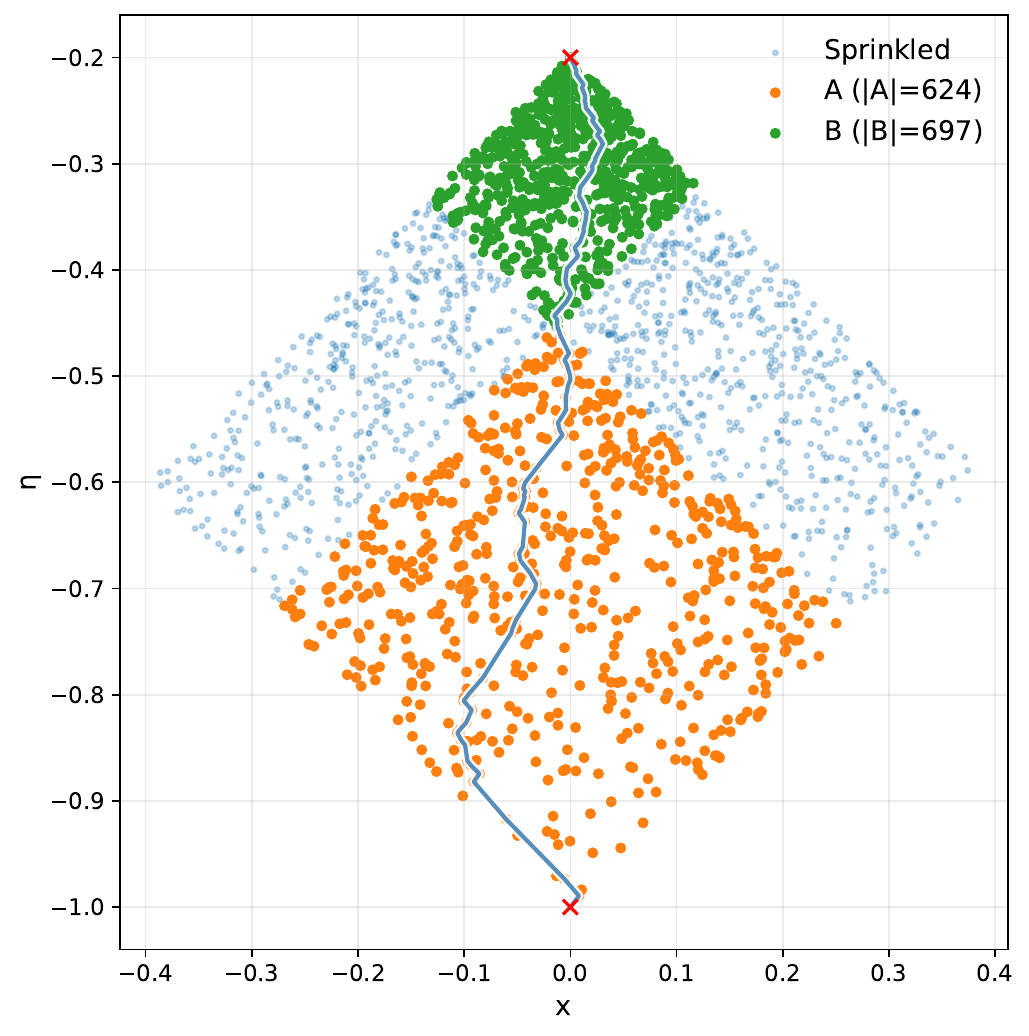}
    \caption{de Sitter space}
  \end{subfigure}

  \vspace{0.5em}

  \begin{subfigure}{0.48\textwidth}
    \centering
    \includegraphics[width=\linewidth]{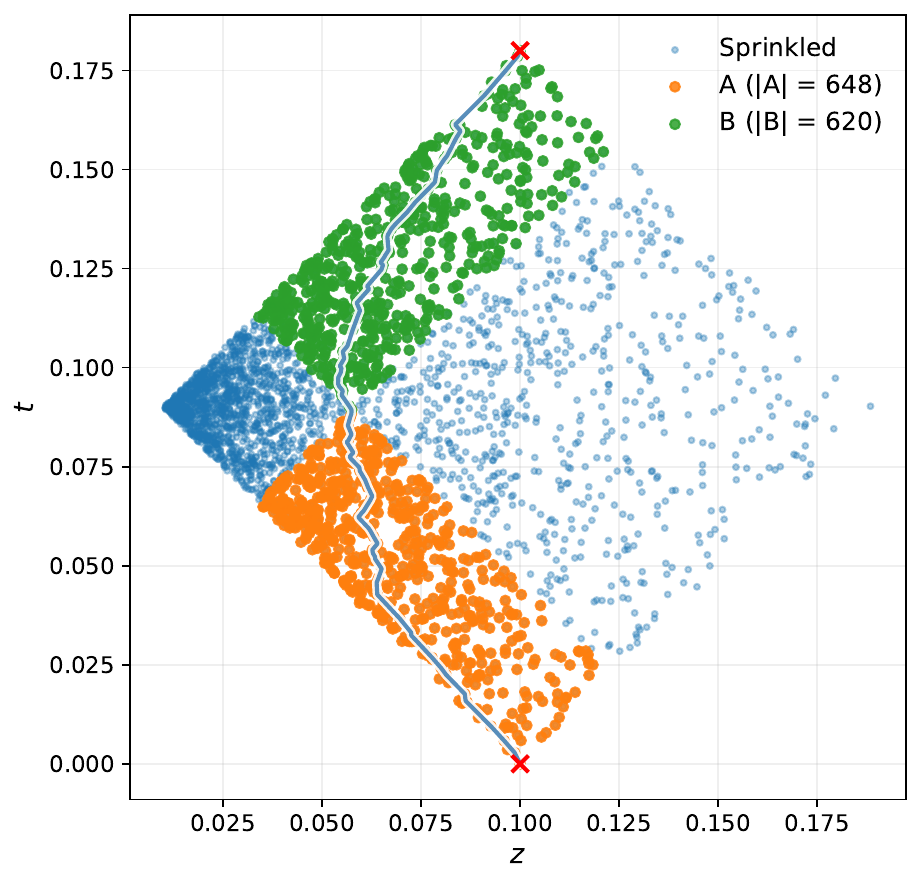}
    \caption{anti-de Sitter space}
  \end{subfigure}
  \caption{Numerical experiments for the Ollivier--Ricci curvature of causal sets generated from sprinklings of Minkowski, de Sitter, and anti-de Sitter space.}
  \label{fig:numerical1}
\end{figure}

\section*{Acknowledgements}

This research was funded by the Austrian Science Fund (FWF) [Grant DOI \href{https://www.fwf.ac.at/en/research-radar/10.55776/EFP6}{10.55776/EFP6}]. For open access purposes, the authors have applied a CC BY public copyright license to any author accepted manuscript version arising from this submission. S.B.'s research was also supported by SUBLOR, a research project funded by the European Union under the Horizon Europe programme's Marie Skłodowska-Curie Actions Postdoctoral Fellowships, grant agreement No.\@ \href{https://cordis.europa.eu/project/id/101282277}{101282277}.

The authors thank Norbert Peyerimhoff, Sumati Surya, Clemens Sämann, Shiping Liu, and Tobias Beran for stimulating and helpful discussions. We are also grateful to Professor David A. Meyer for suggesting, after the first version of this manuscript, that the bias observed in \cref{fig:PltCurvatureNDep} scales asymptotically as $N^{-1/3}$; see \cref{rem:asymptotic-bias}.

\section{Preliminaries}

\subsection{Lorentzian geometry}

\label{subsec:LorentzianGeometry}

Here is a review of some basic facts about Lorentzian geometry, and we refer the reader to the classic text \cite{ONeill1983} for further details.

We denote by $(M,g)$ a Lorentzian manifold of dimension $n$, with signature convention $(-+\cdots +)$, admitting a global time-orientation $T$. For any $x \in M$, we say that a tangent vector $v \in \T_x(M)$ is
\[
  \begin{cases}
     & \text{causal}    \\
     & \text{timelike}  \\
     & \text{null}      \\
     & \text{spacelike}
  \end{cases}
  \quad \text{ if } \quad g_x(v, v) \quad
  \begin{cases}
     & \leq 0 \text{ and } v \neq 0 \\
     & < 0                          \\
     & = 0 \text{ and } v \neq 0    \\
     & > 0 \text{ or } v = 0
  \end{cases}
\]
and is \emph{future-directed} if $g_x(v, T(x)) < 0$. A unit future-directed timelike vector satisfies $g_x(v,v)=-1$. We shall also use the notation $\langle \cdot, \cdot \rangle_x \coloneqq g_x(\cdot, \cdot)$, and may omit the subscript $x$ when it is clear from the context.

A smooth curve $\gamma$ is said to be {\em causal} (resp. {\em timelike}, {\em null}, {\em future-directed}) if its tangent vector $\dot{\gamma}(t)$ is {\em causal} (resp. {\em timelike}, {\em null}, {\em future-directed}) for all $t$. A point $x \in M$ {\em causally precedes} $y \in M$, a relation that we denote by $x \leq y$, if $x=y$ or if there exists a future-directed causal curve $\gamma$ joining $x$ to $y$. Analogously, we say that $x$ {\em chronologically precedes} $y$, denoted by $x \ll y$, if there exists a future-directed timelike curve joining $x$ to $y$. The {\em causal} and the {\em chronological future} of $A \subseteq M$ are defined, respectively, as
\begin{equation*}
  \label{eq:causalchronologicalfuturedefinition}
  J^+(A) \coloneqq \{ y \in M \mid \exists x \in A,\ x \leq y \}, \ \text{ and } \ I^+(A) \coloneqq \{ y \in M \mid \exists x \in A,\ x \ll y \}.
\end{equation*}
Similarly, one can define the causal and chronological past $J^-(A)$ and $I^-(A)$. The {\em causal} and {\em chronological diamonds} of two sets $A, B \subseteq M$ are given by
\begin{equation*}
  \label{eq:causalchronologicaldiamonddefinition}
  J(A, B) \coloneqq J^+(A) \cap J^-(B), \ \text{ and } \ I(A, B) \coloneqq I^+(A) \cap I^-(B).
\end{equation*}
The {\em Lorentzian length} of a future-directed causal curve $\gamma : I \to M$ is given by
\[
  L(\gamma) \coloneqq \int_I \sqrt{-g(\dot{\gamma}(t),\dot{\gamma}(t))}  \diff t.
\]
The {\em time-separation function} $\uptau$ between two points $x,y \in M$ is defined as
\[
  \uptau(x,y) \coloneqq \sup \{ L(\gamma) \ | \ \gamma \ \text{is future-oriented causal, joining $x$ to $y$}\} \quad \text{ if } x \leq y,
\]
and we set $\uptau(x,y) \coloneqq 0$ if $x$ does not causally precede $y$. The time-separation $\uptau$ satisfies the {\em reverse triangular inequality}, i.e.
\begin{equation}
  \label{eq:reversetriangle}
  \uptau(x, z) \geq \uptau(x, y) + \uptau(y, z) \qquad \text{ whenever } x \leq y \leq z.
\end{equation}
At times, we will also equip $\T_x(M)$ with the Minkowski metric $g_x$ and denote the corresponding mathematical objects in this causal structure with a tilde, for instance $\tilde J^+$, $\tilde J^-$, etc.

The {\em exponential map} at $x\in M$ is defined by $\exp_x(v)\coloneqq\gamma_v(1)$ where $\gamma_v$ is the unique geodesic with initial conditions $\gamma_v(0)=x$ and $\dot\gamma_v(0)=v \in \T_x(M)$. A geodesic $\gamma$ is a curve such that $\nabla_{\dot\gamma}\dot\gamma=0$, where $\nabla$ denotes the Levi-Civita connection of the Lorentzian metric $g$.

A {\em convex normal neighbourhood} of a point $x$ is an open neighbourhood $U$ of $x$ such that there exists a unique geodesic segment contained entirely in $U$ joining any two points in $U$, and the exponential map is a diffeomorphism from a star-shaped neighbourhood of $0\in \T_x(M)$ onto $U$. Every point $x \in M$ has a convex normal neighbourhood. The energy between $x$ and $y$ in a convex normal neighbourhood is defined as
\[
  \E(x,y)\coloneqq\frac12\int_0^1 \langle \dot\gamma(t),\dot\gamma(t)\rangle \mathrm{d}t,
\]
where $\gamma$ the unique geodesic from $x$ to $y$. Thus $\E(x,y)$ is negative, zero, or positive according as this geodesic is timelike, null, or spacelike. If $x\leq y$ and $x,y\in U$, then
\[
  \uptau(x,y)^2=-2\E(x,y).
\]

The {\em Riemann curvature tensor} is defined by
\[
  R(X,Y)Z
  \coloneqq
  \nabla_X\nabla_YZ-\nabla_Y\nabla_XZ-\nabla_{[X,Y]}Z,
\]
and satisfies
\[
  R(X,Y)=-R(Y,X),\qquad
  \langle R(X,Y)Z,W\rangle=-\langle R(X,Y)W,Z\rangle,
\]
\[
  \langle R(X,Y)Z,W\rangle=\langle R(Z,W)X,Y\rangle.
\]
The {\em Ricci tensor} is the contraction
\[
  \mathrm{Ric}(X,Y)
  \coloneqq
  \mathrm{tr}\bigl(Z \mapsto R(Z,X)Y\bigr) = \sum_{i=1}^{n-1}
  \langle R(v,e_i)e_i,v\rangle
\]
if $e_0=v,e_1,\dots,e_{n-1}$ is a Lorentz-orthonormal frame with $g(v,v)=-1$.

  {\em Normal coordinates} centred at $x$ are obtained by writing each $y\in U$ uniquely as
\[
  y=\exp_x\left(u^0 v+\sum_{i=1}^{n-1} u^i e_i\right).
\]
The tuple $(u^0,u^1,\dots,u^{n-1})$ gives the normal coordinates of $y$ centred at $x$. In these coordinates, one has that
\[
  g_{\alpha\beta}(x)=\eta_{\alpha\beta},
  \qquad \text{ and } \qquad \sqrt{|\det g_{ij}(w)|}
  =
  1-\frac16\sum_{i,j=0}^{n-1}\operatorname{Ric}_{ij}(x)w^i w^j+O(|w|^3),
\]
where $\eta=\operatorname{diag}(-1,1,\dots,1)$, and $\abs{\cdot}$ is any norm on $\T_x(M)$.

The {\em parallel transport} along a curve $c$ will be denoted by $\mathord{\parallel_s^t}(c) : \T_{c(s)}(M)\to \T_{c(t)}(M)$. We recall that it is an isometry and that it sends $v \in \T_{c(s)}(M)$ to $V(t)$, where $V$ is the unique vector field along $c$ satisfying
\[
  \nabla_{\dot c} V = 0,
  \qquad
  V(s)=v.
\]
If $\gamma$ is a geodesic with initial velocity $v \in \T_x(M)$, then $\dot\gamma(t) = \mathord{\parallel_0^t}(\gamma)[v]$.
We also denote by $D_t V \coloneqq (c^*\nabla)_{\partial_t}V$ the covariant derivative of $V$ along $c$, with $c^*\nabla$ being the pullback connection on $c^*\T(M)$.
If $c$ is a smooth curve with $c(0)=x$ and $u,v,w,z\in \T_{c(s)}(M)$, then, as $t\to s$, we have
\begin{equation}
  \label{eq:Riemannsmallo}
  \bigl\langle R_{c(s)}(u,v)w,z\bigr\rangle_{c(s)}
  =
  \bigl\langle
  R_{c(t)}
  \bigl(\mathord{\parallel_s^t}(c)u,\mathord{\parallel_s^t}(c)v\bigr)
  \mathord{\parallel_s^t}(c)w,
  \mathord{\parallel_s^t}(c)z
  \bigr\rangle_{c(t)}
  +
  O(|t-s|).
\end{equation}

Let $\Gamma:(-\varepsilon,\varepsilon)\times[0,1]\to M$ be a smooth two-parameter family of curves, and let $V$ be a smooth vector field along $\Gamma$. Set $S \coloneqq \partial \Gamma/\partial s$ and $T \coloneqq \partial \Gamma/\partial t$, then the covariant derivatives satisfy the commutator identity
\begin{equation}
  \label{eq:commutatoridentity}
  D_s D_t V - D_t D_s V = R(S,T)V.
\end{equation}
Suppose now that $\Gamma$ is a geodesic variation of $\gamma$, meaning that $\Gamma(0,t)=\gamma(t)$ and, for each fixed $s$, the curve $t\mapsto \Gamma(s,t)$ is a geodesic. If
\[
  J(t)\coloneqq S(0,t)=\left.\frac{\partial \Gamma}{\partial s}\right|_{s=0}(t)
\]
is the associated variational field, then $J$ satisfies the Jacobi equation
\begin{equation}
  \label{eq:JacobiEquationGeneral}
  D_t^2J+R(J,\dot\gamma)\dot\gamma=0.
\end{equation}

\subsection{Causal set theory}

\label{subsec:CausalSet}

We give a brief account of causal set theory, focusing on some mathematical features of the theory rather than its connection to physics. The textbook \cite{sumati2025} provides a good introduction to causal set theory. We start with a definition of causal set, which is basically a locally finite partially ordered set.

\begin{definition}
  \label{def:causalset}
  A \emph{causal set} is a pair $(C,\leq)$, where $C$ is a set and $\leq$ is a relation on $C$ satisfying:
  \begin{enumerate}[label=(\roman*)]
    \item \emph{Reflexivity}: for all $x \in C$, $x \leq x$;
    \item \emph{Antisymmetry}: for all $x,y \in C$, if $x \leq y$ and $y \leq x$, then $x=y$;
    \item \emph{Transitivity}: for all $x,y,z \in C$, if $x \leq y$ and $y \leq z$, then $x \leq z$;
    \item \emph{Local finiteness}: for all $x,z \in C$, the causal diamond
          \[
            J(x, y) \coloneqq \{ y \in C : x \leq y \leq z \}
          \]
          is finite.
  \end{enumerate}
\end{definition}

We call the relation $\leq$ the causal relation, while the chronological relation $\ll$ is defined by $x \ll y$ if $x \leq y$ and $x \neq y$. One could equivalently define a causal set as a pair $(C,\ll)$, with $\ll$ irreflexive and transitive, rather than as a pair $(C,\leq)$, with $\leq$ reflexive and antisymmetric. This convention is used in many publications on causal set theory. The causal relation of a globally hyperbolic Lorentzian manifold, as introduced in the previous section, satisfies the first three conditions. It is the local finiteness condition that introduces spacetime discreteness.

The (counting) time-separation distance is defined as
\[
  \uptau(x,y)
  \coloneqq
  \max \{m \mid x=x_0 \ll x_1 \ll \cdots \ll x_m=y\},
\]
whenever $x\leq y$, and $\uptau(x,y)=0$ otherwise. The function $\uptau$ behaves much like the time-separation of a Lorentzian manifold, and in particular the reverse triangle inequality \cref{eq:reversetriangle} holds. A curve $\gamma = (x_0,\ldots,x_m)$ in $C$ with $x=x_0 \ll x_1 \ll \cdots \ll x_m=y$, commonly referred to as a \emph{chain}, is a \emph{geodesic} from $x$ to $y$ if it realises the time distance between $x$ and $y$, that is, if $m=\uptau(x,y)$. The timelike diameter of the causal set is then
\[
  \mathrm{diam}_{\uptau}(\mathcal{C})
  \coloneqq
  \sup\{\uptau(x,y)\mid x,y\in\mathcal{C},\ x\leq y\}.
\]
We note the following simple but useful fact: if $r$ is fixed and $\mathrm{diam}_{\uptau}(\mathcal{C})\geq 2r$, then
\begin{equation}
  \label{eq:diam2r}
  \mathrm{diam}_{\uptau}(\mathcal{C})
  =
  \sup\{\uptau(x_0,x_{n+r})\mid \gamma=(x_0,\dots,x_{n+r})\text{ is a geodesic, } n \geq r\}.
\end{equation}

\begin{remark}
  \label{remark:weighted}
  \cref{def:causalset} could be generalised to a \emph{weighted causal set}, i.e. $(C,\leq,w)$, where $(C,\leq)$ is a causal set and
  \[
    w \colon \mathcal{L}(C) \to (0,\infty)
  \]
  is a weight function defined on the set $\mathcal{L}(C)$ of pairs $(x,y) \in C \times C$ such that $x \ll y$ and there is no $z \in C$ with $x \ll z \ll y$. The definition of the time separation would have to be adapted accordingly. For the sake of simplicity, we restrict the discussion in this paper to unweighted causal sets, except for a few brief remarks.
\end{remark}

One of the most natural ways to produce a causal set is by performing a Poisson sprinkling into a Lorentzian manifold $(M,g)$. We denote by $(M,\mathcal B(M))$ the Borel measurable space underlying the spacetime and we consider $(\Omega,\mathcal F,\mathbb P)$ a probability space. We equip $\mathcal N(M)$, the space of counting measures on $M$, with the smallest $\sigma$-algebra making evaluation maps
\[
  \mathcal{N}(M) \to \interval{0}{+\infty} : \mu \mapsto \mu(A)
\]
measurable for all $A\in\mathcal B(M)$. A \emph{point process} on $M$ is a measurable map
\[
  N:(\Omega,\mathcal F,\mathbb P)\longrightarrow \mathcal N(M).
\]
Equivalently, for every measurable set $A\in\mathcal B(M)$, the map
\[
  N(A):\Omega\to \mathbb N\cup\{\infty\} : \omega \mapsto N(\omega)[A]
\]
is a random variable. A point process $N$ is called \emph{simple} if
\[
  \mathbb P\Big( \left\{ \omega \mid \forall x \in M, N(\omega)(\{x\})\leq 1 \right\}
  \Big)=1,
\]
which means that, with probability one, no point of $M$ occurs with multiplicity. In this case, $N$ can be identified with the random locally finite subset of $M$
\[
  C:\Omega\to \mathcal P(M)^{<\infty}_{\mathrm{loc}} : \omega \mapsto
  C(\omega)\coloneqq\{x\in M:N(\omega)(\{x\})=1\},
\]
where we have denoted $\mathcal P(M)^{<\infty}_{\mathrm{loc}}$ for the set of locally finite subsets of $M$. The space $\mathcal P(M)^{<\infty}_{\mathrm{loc}}$ is equipped with the counting $\sigma$-algebra, i.e. the smallest $\sigma$-algebra such that for all $A\in\mathcal B(M)$, the map $\mathcal P(M)^{<\infty}_{\mathrm{loc}} \to \mathbb{N} \cup \{\infty\} : C\mapsto \#(C\cap A)$ is measurable. Thus, for every
$A\in\mathcal B(M)$,
\begin{equation}
  \label{eq:counting}
  N(\omega)(A)=\#(A\cap C(\omega)).
\end{equation}
Conversely, any random locally finite subset $C$ defines a simple locally finite point process $N$ by taking \cref{eq:counting} as the definition of its counting measure. Thus simple locally finite point processes and random locally finite subsets are identified via \cref{eq:counting}, up to modification
on a $\mathbb P$-null set.

An \emph{sprinkling density} is a Borel-measurable function
$\rho:M\to[0,\infty)$ such that $\rho\in L^1_{\mathrm{loc}}(M,\mathrm{vol}_g)$. It defines the \emph{density measure}
\[
  \Lambda:\mathcal B(M)\to[0,\infty],
  \qquad
  \Lambda(A)\coloneqq\int_A \rho \, \mathrm{d}\mathrm{vol}_g .
\]
If $\rho$ is constant, then $\Lambda(A)=\rho\,\mathrm{vol}_g(A)$, and for pairwise disjoint measurable sets $A_1,\dots,A_k$, it holds that
\[
  \Lambda(\cup_{i=1}^k A_i) = \sum_{i=1}^k \Lambda(A_i).
\]
In causal set theory, $\rho$ is often called the sprinkling density (when constant), $\rho^{-1}$ the discreteness volume scale, and $\rho^{-1/\mathrm{dim}(M)}$ discreteness length scale

We say that the simple point process $N$, or equivalently the random subset $C$, is a \emph{Poisson sprinkling} with sprinkling density $\rho \in L^1_{\mathrm{loc}}(M,\mathrm{vol}_g)$ if its counting process satisfies:
\begin{enumerate}[label=(\roman*)]
  \item \emph{Complete independence:} whenever
        $A_1,\dots,A_k\in\mathcal B(M)$ are pairwise disjoint sets with
        $\Lambda(A_i)<\infty$, the random variables
        \[
          N(A_1),\dots,N(A_k)
        \]
        are independent.

  \item \emph{Mean measure:} for every $A\in\mathcal B(M)$ with
        $\Lambda(A)<\infty$,
        \[
          \mathbb E[N(A)]
          =
          \Lambda(A)
          =
          \int_A \rho\,\mathrm{d}\mathrm{vol}_g .
        \]
\end{enumerate}

Since $\rho$ is locally integrable, $\Lambda$ is locally finite and thus the sprinkling is locally
finite almost surely:
\[
  \mathbb P\Bigl(
  \bigl\{\omega\in\Omega:
  \#\bigl(\mathcal C(\omega)\cap K\bigr)<\infty
  \text{ for every compact }K\subset M
  \bigr\}
  \Bigr)=1.
\]

Because $\Lambda$ is $\sigma$-finite and non-atomic and $N$ is simple, the ``mean measure'' condition can equivalently be replaced by
\[
  N(A)\sim \mathrm{Poisson}(\Lambda(A)),
  \qquad A\in\mathcal B(M),\ \Lambda(A)<\infty,
\]
where $\mathrm{Poisson}(\Lambda(A))$ denotes the Poisson distribution with parameter $\Lambda(A)$.

Finally, assume that a Lorentzian manifold $(M,g)$ is globally hyperbolic, and let $C\subseteq M$ be a Poisson sprinkling. The causal relation on $M$ is a partial order, and since global hyperbolicity implies causality it induces a partial order on $C(\omega)$ by restriction. Causal diamonds in a globally hyperbolic spacetime are compact, and $C$ is locally finite almost surely, so the interval $J_{C(\omega)}(x,y)=C(\omega)\cap J_M(x,y)$ is almost surely finite for all $x,y\in C(\omega)$. Therefore, $(C(\omega),\leq_{C(\omega)})$ is a causal set for almost every $\omega\in\Omega$.

\subsection{Lorentzian optimal transport}

In this section, we survey several results in Lorentzian optimal transport. The general framework is that of \emph{Lorentzian (pre-)length spaces}, also called \emph{metric spacetimes}; see \cite{Kunzinger2018,minguzzi2019lorentzian}, while the general theory of Lorentzian optimal transport can be found in \cite{McCann4,MondinoSuhr2023,CavMondinoReview2022,CavMondino2024,BraunMcCann2023}. To keep the presentation simple and as self-contained as possible, we work with triples $(X,\leq,\uptau)$, where $X$ is either a Lorentzian manifold $M$ or a causal set $C$, equipped respectively with its causal structure as described in \cref{subsec:LorentzianGeometry,subsec:CausalSet}.

The $\sigma$-algebra $\mathcal B(X)$ is the Borel $\sigma$-algebra of the manifold topology when $X=M$, and the discrete $\sigma$-algebra when $X=C$. We write $\mathcal P(X)$ for the set of probability measures on $(X,\mathcal B(X))$, and denote the support of $\mu\in\mathcal P(X)$ by $\supp(\mu)$. For a causal set, $\supp(\mu)=\{x\in C:\mu(x)>0\}$ and integrals are finite sums,
\[
  \int_C f\,\diff\mu=\sum_{x\in C} f(x)\mu(x).
\]
In the smooth case, $\mathcal P_c(M)\subseteq\mathcal P(M)$ denotes the compactly supported Borel probability measures; in the causal set case, $\mathcal P_c(C)$ denotes the finitely supported probability measures. The pushforward of $\mu \in \mathcal{P}(X)$ by a measurable map $T : X \to Y$ is denoted as usual by $T_\sharp \mu \in \mathcal P(Y)$. We denote the coordinate projections by $\mathrm{pr}_1,\mathrm{pr}_2:X\times X\to X$. If $T:X\to X$ is measurable, then $(\mathrm{Id},T)_\sharp\mu\in\mathcal P(X\times X)$ is the graph measure of $T$, given by
\[
  (\mathrm{Id},T)_\sharp\mu(A\times B)
  =
  \mu\bigl(A\cap T^{-1}(B)\bigr),
  \qquad A,B\in\mathcal B(X).
\]
Its first and second marginals are $\mu$ and $T_\sharp\mu$, respectively. A function $u:E\to\mathbb R$, with $E\subseteq X$, is called \emph{$1$-steep} if $u(q)-u(p)\geq \uptau(p,q)$ for all $p,q\in E$ with $p\leq q$.

A \emph{coupling}, or \emph{transport plan}, between $\mu,\nu\in\mathcal P(X)$ is a probability measure $\pi\in\mathcal P(X\times X)$ with marginals $(\mathrm{pr}_1)_\sharp\pi=\mu$ and $(\mathrm{pr}_2)_\sharp\pi=\nu$. Let
\[
  X^2_{\leq}\coloneqq\{(p,q)\in X\times X:p\leq q\}, \qquad X^2_{\ll}\coloneqq\{(p,q)\in X\times X:p\ll q\}
\]
denote the graph of the causal (resp. chronological) order. A coupling $\pi\in\Pi(\mu,\nu)$ is called \emph{causal} (resp. \emph{chronological}), or \emph{admissible}, if it is concentrated on $X^2_{\leq}$ (resp. $X^2_{\ll}$)
In other words, a causal coupling transports mass only causally. The set of all causal couplings is denoted by $\Pi_{\leq}(\mu,\nu)$. Note that if $\supp(\mu)\times\supp(\nu)\subseteq X^2_{\leq}$, then every coupling between $\mu$ and $\nu$ is causal. A \emph{transport map} is a measurable map $T:X\to X$ such that $T_\sharp\mu=\nu$. The graph measure $(\mathrm{Id},T)_\sharp\mu$ is then a coupling, and is causal if $x\leq T(x)$ for $\mu$-a.e. $x\in X$.

For $\mu,\nu\in\mathcal P(X)$, the $1$-Lorentz-Wasserstein distance is defined by
\begin{equation}
  \label{eq:ell1OT}
  \ell_1(\mu,\nu)
  \coloneqq
  \sup_{\pi\in\Pi_{\leq}(\mu,\nu)}
  \int_{X\times X} \uptau(p,q)\,\diff\pi(p,q),
\end{equation}
if $\Pi_{\leq}(\mu,\nu)\neq\varnothing$, and by $\ell_1(\mu,\nu)\coloneqq-\infty$ otherwise. This is the $p=1$ case of the $p$-Lorentz-Wasserstein distance of \cite{CavMondino2024}, and is the only case used in this work. By \cite[Proposition~2.5]{CavMondino2024}, $\ell_1$ satisfies the reverse
triangle inequality:
\[
  \ell_1(\mu_0,\mu_2)
  \geq
  \ell_1(\mu_0,\mu_1)+\ell_1(\mu_1,\mu_2), \qquad \text{ for any } \mu_0,\mu_1,\mu_2\in\mathcal P(X).
\]

\begin{remark}[Dirac masses]
  \label{rem:dirac-mass-identities}
  If $p\leq q$, the only coupling between $\delta_p$ and $\delta_q$ is $\delta_{(p,q)}$, which is causal; hence $\ell_1(\delta_p,\delta_q)=\uptau(p,q)$. More generally, if $p\leq z$ for every $z\in\supp(\nu)$, then every coupling between $\delta_p$ and $\nu$ is causal, and $\ell_1(\delta_p,\nu)=\int_X \uptau(p,z)\,\diff\nu(z)$. Similarly, if $z\leq q$ for every $z\in\supp(\mu)$, then $\ell_1(\mu,\delta_q)=\int_X \uptau(z,q)\,\diff\mu(z)$.
\end{remark}

A causal coupling $\pi\in\Pi_{\leq}(\mu,\nu)$ attaining the supremum in \cref{eq:ell1OT} is called \emph{optimal}. If $X=M$ is a globally hyperbolic spacetime and $\mu,\nu\in\mathcal P_c(M)$ satisfy $\Pi_{\leq}(\mu,\nu)\neq\varnothing$, then that supremum is finite and attained; see \cite[Proposition~2.3]{CavMondino2024}. When $X = C$ is a causal set, the maximisation problem in \cref{eq:ell1OT} becomes
\begin{equation}
  \label{eq:discreteell1}
  \ell_1(\mu,\nu)
  =
  \max_{\pi}
  \sum_{a,b\in C} \pi(a,b)\,\uptau(a,b),
\end{equation}
where the maximum is taken over all non-negative matrices $\pi=(\pi(a,b))_{a,b\in C}$ satisfying, for all $a, b \in C$,
\[
  \pi(a,b)\geq 0,\qquad
  \sum_{b\in C}\pi(a,b)=\mu(a),\qquad
  \sum_{a\in C}\pi(a,b)=\nu(b),\qquad
  \pi(a,b)=0\ \text{if }a\nleq b.
\]
If $\mu,\nu\in\mathcal P_c(C)$ satisfy
$\Pi_{\leq}(\mu,\nu)\neq\varnothing$, then this is a
finite-dimensional linear program, and its
maximum is attained.

We now turn to Kantorovich duality, see \cite[Section~2.4]{CavMondino2024} and \cite[Appendix C]{BraunMcCann2023}. Let $X=M$ be a globally hyperbolic spacetime, and let $\mu,\nu\in\mathcal P_c(M)$ satisfy $\supp(\mu)\times\supp(\nu)\subseteq X^2_{\ll}$. Then, it holds that
\begin{equation}
  \label{theorem:KantorovichDuality}
  \ell_1(\mu,\nu)
  =
  \inf
  \Big(
  \int_M u\,\diff\nu
  -
  \int_M u\,\diff\mu
  \Big),
\end{equation}
where the infimum is taken over all Borel $1$-steep functions $u:E\to\mathbb R$, defined on a Borel set $E\subseteq M$ containing $\supp(\mu)\cup\supp(\nu)$, with $u\in L^1(\mu)\cap L^1(\nu)$. Moreover, the infimum is attained by a function $u:E\to\mathbb R$ defined on a compact, causally convex set $E\subseteq M$ containing $\supp(\mu)\cup\supp(\nu)$. Finally, if $u$ is a minimizer of \cref{theorem:KantorovichDuality} and $\pi$ is optimal, then
\begin{equation}
  \label{eq:equalityoptimal}
  u\circ\mathrm{pr}_2 - u\circ\mathrm{pr}_1 = \uptau
  \quad \pi\textnormal{-a.e.}
\end{equation}
Conversely, if $\pi\in\Pi_{\leq}(\mu,\nu)$ and $u$ is an admissible $1$-steep function satisfying \cref{eq:equalityoptimal}, then $\pi$ is optimal and $u$ is a dual minimizer.

In the causal set setting $X=C$, let $\mu,\nu\in\mathcal P_c(C)$ satisfy $\Pi_{\leq}(\mu,\nu)\neq\varnothing$. Then the primal problem defining $\ell_1(\mu,\nu)$ is a finite-dimensional linear program on $\supp(\mu)\times\supp(\nu)$, and finite-dimensional linear-programming duality gives
\begin{equation}
  \label{theorem:KantorovichDuality2}
  \ell_1(\mu,\nu)
  =
  \inf
  \Big(
  \sum_{b\in\supp(\nu)}u(b)\nu(b)
  -
  \sum_{a\in\supp(\mu)}u(a)\mu(a)
  \Big),
\end{equation}
where the infimum is taken over all $1$-steep functions $u:\supp(\mu)\cup\supp(\nu)\to\mathbb R$. The infimum is attained, and a feasible pair $(\pi,u)$ is primal-dual optimal if and only if \cref{eq:equalityoptimal} holds.

\begin{remark}[Representation of dual minimizers]
  \label{eq:utauconcave}
  If $X=M$ is a globally hyperbolic spacetime, $\mu,\nu\in\mathcal P_c(M)$ satisfy $\supp(\mu)\times\supp(\nu)\subseteq X^2_{\ll}$, and $u:E\to\mathbb R$ is a minimizer in \cref{theorem:KantorovichDuality}, then $u$ admits the following representation (see \cite[Definition~2.23,Corollary 2.29]{CavMondino2024}). There exist $A\subseteq\supp(\mu)\subseteq E$ and $B\subseteq\supp(\nu)\subseteq E$, with $\mu(A)=1$ and $\nu(B)=1$, and a $\mu$-integrable upper semicontinuous function $\phi:A\to\mathbb R$ such that, for every $z\in B$,
  \[
    u(z)
    =
    \sup_{x\in A}
    \bigl(\phi(x)+\uptau(x,z)\bigr).
  \]

\end{remark}

\section{Ollivier-Ricci curvature for causal sets}

\subsection{Definitions and basic properties}
\label{subsection:ORCST}

We postpone the proof of the smooth Lorentzian expansion \cref{eq:MainTheorem} to the following section, and use it here as motivation for a notion of curvature on causal sets. In analogy with Ollivier's definition for combinatorial graphs equipped with the graph distance, and motivated by the same observation in Riemannian geometry \cite{Ollivier2009}, the idea is to neglect the higher-order terms in \cref{eq:MainTheorem} and isolate the quantity
\[
    \frac{\epsilon^2}{2}\,\frac{n}{(n+1)(n+2)}\Ric(v,v),
\]
which is then taken as the definition of a directional curvature along the geodesic joining $x$ and $y$.

Motivated by the corresponding notions for graphs, we introduce here several versions of Ollivier--Ricci curvature for causal sets: the analogue of Ollivier's original definition from \cite{Ollivier2009}, together with the $\alpha$-idleness and Lin--Lu--Yau variants introduced in \cite{Lin2011}.

\begin{definition}[Ollivier-Ricci curvature for causal sets]
    \label{def:OllivierRicciCST}
    Let $(C,\leq)$ be a causal set, $\alpha\in[0,1]$, and let $n, r\in\mathbb N\setminus\{0\}$ with $n \geq r$. Let $\gamma=(x_0,\dots,x_{n+r})$ be a maximising geodesic, and set $x\coloneqq x_0$ and $y\coloneqq x_n$. Define the measures
    \[
        \mu_x
        \coloneqq
        \frac{1}{|J(x,x_r)|}
        \sum_{z\in J(x,x_r)}\delta_z, \qquad \mu_x^\alpha
        \coloneqq
        \alpha \delta_x
        +
        \frac{1-\alpha}{|J(x,x_r)|-1}
        \sum_{z\in J(x,x_r)\setminus\{x\}}\delta_z,
    \]
    and
    \[
        \mu_y
        \coloneqq
        \frac{1}{|J(y,x_{n+r})|}
        \sum_{z\in J(y,x_{n+r})}\delta_z, \qquad \mu_y^\alpha
        \coloneqq
        \alpha \delta_y
        +
        \frac{1-\alpha}{|J(y,x_{n+r})|-1}
        \sum_{z\in J(y,x_{n+r})\setminus\{y\}}\delta_z.
    \]
    The Ollivier--Ricci curvature along $\gamma$ at scale $r$, its $\alpha$-idleness version, and the Lin--Lu--Yau curvature are given by
    \[
        \kappa_r(\gamma)\coloneqq\frac{\ell_1(\mu_x,\mu_y)}{\uptau(x,y)}-1,\qquad
        \kappa_r^\alpha(\gamma)\coloneqq\frac{\ell_1(\mu_x^\alpha,\mu_y^\alpha)}{\uptau(x,y)}-1,\qquad
        \kappa_r^{\mathrm{LLY}}(\gamma)\coloneqq\lim_{\alpha\to 1}\frac{\kappa_r^\alpha(\gamma)}{1-\alpha}.
    \]
\end{definition}

For notational convenience, we will often write
\[
    N_x = N_x^{(r)}\coloneqq|J(x,x_r)|,
    \qquad
    N_y = N_y^{(r)}\coloneqq|J(y,x_{n+r})|,
\]
where $\gamma=(x_0,\dots,x_{n+r})$ is the maximising geodesic, which will be clear from the context.

\begin{remark}
    The assumption $n\ge r$ guarantees that every coupling between $\mu_x$ and $\mu_y$, and likewise between $\mu_x^\alpha$ and $\mu_y^\alpha$, is causal. Indeed, since $\gamma=(x_0,\dots,x_{n+r})$ is a maximising geodesic and $y=x_n$, the inequality $n\ge r$ implies $x_r\le y$. Hence for every $a\in J(x,x_r)$ and $b\in J(y,x_{n+r})$ one has
    \[
        x\le a\le x_r\le y\le b\le x_{n+r},
    \]
    so $a\le b$. In particular, the sets of causal couplings between $\mu_x$ and $\mu_y$, and between $\mu_x^\alpha$ and $\mu_y^\alpha$, are non-empty, since they contain $\mu_x\otimes\mu_y$ and $\mu_x^\alpha\otimes\mu_y^\alpha$, respectively. Hence $\ell_1(\mu_x, \mu_y) \in \ointerval{0}{+\infty}$ and $\ell_1(\mu_x^\alpha, \mu_y^\alpha) \in \ointerval{0}{+\infty}$.
\end{remark}

\begin{remark}
    Ollivier's original definition in \cite{Ollivier2009} is not intrinsically tied
    to uniform measures on graph balls. Its basic input is a metric space together
    with a Markov kernel $x \mapsto \mu_x$, that is, a probability measure $\mu_x$
    assigned to each point. Similarly, one could consider a generalisation of
    \cref{def:OllivierRicciCST} in which one is given a
    ``diamond-indexed probability kernel'': a collection $\mathcal D$ of causal
    diamonds together with a probability measure $\mu_J$ supported on $J$ for each
    $J\in\mathcal D$.
\end{remark}

We next study the dependence of $\kappa_r^\alpha(\gamma)$ on the idleness parameter $\alpha$. The main point is that, as in the graph setting considered by Lin, Lu, and Yau \cite[Lemma 2.1]{Lin2011}, this dependence is concave. This, together with the fact that $\kappa^1_r(0)$ (see \cref{prop:r1} below), also justifies taking the limit in the definition of $\kappa_r^{\mathrm{LLY}}(\gamma)$.

\begin{proposition}
    Let $n,r\in\mathbb N\setminus\{0\}$ with $n\ge r$, and let
    $\gamma=(x_0,\dots,x_{n+r})$ be a maximising geodesic. Then the map
    \[
        [0,1]\ni \alpha \mapsto \kappa_r^\alpha(\gamma)
    \]
    is concave.
\end{proposition}

\begin{proof}
    Write $x\coloneqq x_0$ and $y\coloneqq x_n$. Let $0\le \alpha<\beta<\gamma\le 1$, and set
    \begin{equation}
        \label{eq:identityalphabeta}
        \lambda\coloneqq\frac{\gamma-\beta}{\gamma-\alpha}\in(0,1), \quad \text{   so that } \quad \beta=\lambda\alpha+(1-\lambda)\gamma.
    \end{equation}

    Let $\pi^\alpha$ be an optimal causal coupling between $\mu_x^\alpha$ and $\mu_y^\alpha$, and let $\pi^\gamma$ be an optimal causal coupling between $\mu_x^\gamma$ and $\mu_y^\gamma$. Thus
    \[
        \ell_1(\mu_x^\alpha,\mu_y^\alpha)
        =
        \sum_{u,v\in C} \pi^\alpha(u,v)\,\uptau(u,v),
        \qquad
        \ell_1(\mu_x^\gamma,\mu_y^\gamma)
        =
        \sum_{u,v\in C} \pi^\gamma(u,v)\,\uptau(u,v).
    \]

    We claim that
    \begin{equation}
        \label{eq:interpolatepi}
        \pi^\beta(u,v)\coloneqq\lambda \pi^\alpha(u,v)+(1-\lambda)\pi^\gamma(u,v),
    \end{equation}
    is a causal coupling between $\mu_x^\beta$ and $\mu_y^\beta$. For all $u,v\in C$, we immediately have $\pi^\beta(u,v)\ge 0$ and
    \[
        \sum_{u\in C}\pi^\beta(u,v)
        =
        \lambda\sum_{u\in C}\pi^\alpha(u,v)+(1-\lambda)\sum_{u\in C}\pi^\gamma(u,v)
        =
        \lambda\,\mu_y^\alpha(v)+(1-\lambda)\,\mu_y^\gamma(v).
    \]
    Since $\mu_x^\alpha$ and $\mu_x^\gamma$ are supported on the same finite set
    $J(x,x_{r})$, one checks directly that
    \[
        \mu_y^\beta=\lambda\,\mu_y^\alpha+(1-\lambda)\,\mu_y^\gamma.
    \]
    Indeed, at the point $y$ this is exactly the identity \cref{eq:identityalphabeta} and, at every point of $J(y,x_{n+r})\setminus\{y\}$, it reduces to
    \[
        \frac{1-\beta}{N_y-1}
        =
        \lambda\frac{1-\alpha}{N_y-1}
        +(1-\lambda)\frac{1-\gamma}{N_y-1}.
    \]
    Hence, we have found that, for every $u, v \in C$,
    \[
        \sum_{v\in C}\pi^\beta(u,v)=\mu_x^\beta(u), \quad \text{ and, with a similar argument,} \quad \sum_{u\in C}\pi^\beta(u,v)=\mu_y^\beta(v).
    \]

    To prove that $\pi^\beta$ is causal, let $u,v\in C$ be such that $\pi^\beta(u,v)>0$. From \cref{eq:interpolatepi} and the fact that  $\pi^\alpha$ and $\pi^\gamma$ are causal couplings, we deduce that $\pi^\alpha(u,v)>0$ or $\pi^\gamma(u,v)>0$, and thus $u\le v$.

    Therefore, $\pi^\beta$ is indeed a causal coupling between $\mu_x^\beta$ and $\mu_y^\beta$. By the definition of $\ell_1$ as a supremum over causal couplings, we obtain
    \begin{align*}
        \ell_1(\mu_x^\beta,\mu_y^\beta)
         & \ge
        \sum_{u,v\in C}\pi^\beta(u,v)\,\uptau(u,v)              \\
         & =
        \lambda\sum_{u,v\in C}\pi^\alpha(u,v)\,\uptau(u,v)
        +(1-\lambda)\sum_{u,v\in C}\pi^\gamma(u,v)\,\uptau(u,v) \\
         & =
        \lambda\,\ell_1(\mu_x^\alpha,\mu_y^\alpha)
        +(1-\lambda)\,\ell_1(\mu_x^\gamma,\mu_y^\gamma).
    \end{align*}
    Dividing by the positive quantity $\uptau(x,y)$ and subtracting $1$, we get
    \begin{align*}
        \kappa_r^\beta(\gamma)
         & =
        \frac{\ell_1(\mu_x^\beta,\mu_y^\beta)}{\uptau(x,y)}-1 \ge
        \lambda\left(\frac{\ell_1(\mu_x^\alpha,\mu_y^\alpha)}{\uptau(x,y)}-1\right)
        +(1-\lambda)\left(\frac{\ell_1(\mu_x^\gamma,\mu_y^\gamma)}{\uptau(x,y)}-1\right) \\
         & =
        \lambda\,\kappa_r^\alpha(\gamma)+(1-\lambda)\,\kappa_r^\gamma(\gamma).
    \end{align*}
    This proves that $\alpha\mapsto \kappa_r^\alpha(\gamma)$ is concave.
\end{proof}

In the context of combinatorial graphs, Ollivier--Ricci curvature is typically evaluated at scale $r = 1$, that is, between adjacent vertices. This is because unit-radius neighbourhoods already encode non-trivial combinatorial information, such as branching, degree variation, and overlap between neighbours. As the next result shows, in our causal-set definition, the measures become degenerate when $r = 1$, being supported on only two points which renders the transport cost completely rigid.

\begin{proposition}
    \label{prop:r1}
    For every $\alpha\in[0,1]$ and all $r \geq 1$, we have
    \[
        \kappa_1(\gamma)=\kappa_1^\alpha(\gamma)=\kappa_1^{\mathrm{LLY}}(\gamma)=\kappa_r^1(\gamma)=0.
    \]
\end{proposition}

\begin{proof}
    Write $\gamma=(x_0,\dots,x_{n+1})$ and set $x\coloneqq x_0$, $y\coloneqq x_n$. If $r=1$, then
    $J(x,x_1)=\{x,x_1\}$ and $J(y,x_{n+1})=\{y,x_{n+1}\}$, so
    \[
        \mu_x=\tfrac12(\delta_x+\delta_{x_1}), \qquad
        \mu_y=\tfrac12(\delta_y+\delta_{x_{n+1}}),
    \]
    and
    \[
        \mu_x^\alpha=\alpha\delta_x+(1-\alpha)\delta_{x_1}, \qquad
        \mu_y^\alpha=\alpha\delta_y+(1-\alpha)\delta_{x_{n+1}}.
    \]

    For $a\in J(x,x_1)$ and $b\in J(y,x_{n+1})$, one has
    $\uptau(a,b)=n-\uptau(x,a)+\uptau(y,b)$, since the only possibilities are
    $\uptau(x,y)=n$, $\uptau(x,x_{n+1})=n+1$, $\uptau(x_1,y)=n-1$, and
    $\uptau(x_1,x_{n+1})=n$.

    Hence for any causal coupling $\pi$ between $\mu_x^\alpha$ and $\mu_y^\alpha$,
    \[
        \int \uptau(a,b)\diff \pi(a,b)
        = n-\int \uptau(x,a)\diff \mu_x^\alpha(a)+\int \uptau(y,b)\diff \mu_y^\alpha(b).
    \]
    Since both integrals on the right hand side are equal to $1-\alpha$, it follows that
    $\smash{\ell_1(\mu_x^\alpha,\mu_y^\alpha)}=n=\uptau(x,y)$. Therefore
    $\smash{\kappa_1^\alpha(\gamma)}=0$ for all $\alpha\in[0,1]$. In particular,
    taking $\alpha=1/2$ gives $\kappa_1(\gamma)=0$, and then also $\kappa_1^{\mathrm{LLY}}(\gamma) = 0$.

    Finally, for arbitrary $r\geq1$, if $\alpha=1$, then
    $\mu_x^1=\delta_x$ and $\mu_y^1=\delta_y$, so
    $\ell_1(\mu_x^1,\mu_y^1)=\uptau(x,y)$ and thus $\kappa_r^1(\gamma)=0$.
\end{proof}

This suggests that the notion of Ollivier--Ricci curvature introduced in the present work is inherently mesoscopic rather than strictly local. Curvature becomes visible only when the observation window is large enough for the diamonds $J(x,x_r)$ and $J(y,x_{n+r})$ to develop genuine internal structure. Having noted that the choice $r=1$ does not yield a meaningful notion of curvature, the previous proposition shows that $r=2$ is still insufficient. Indeed, we still have $\kappa_2(\gamma)=0$, while the $\alpha$-idleness and Lin--Lu--Yau versions are determined entirely by the numbers of points in the diamonds $J(x,x_r)$ and $J(y,x_{n+r})$. To detect genuine curvature effects, it is therefore necessary to consider at least the case $r\geq 3$.

\begin{proposition}
    \label{prop:r2}
    Let $n\geq 2$, let $\gamma=(x_0,\dots,x_{n+2})$ be a maximising geodesic,
    and set $x\coloneqq x_0$, $y\coloneqq x_n$.
    Then, for every $\alpha\in[0,1]$, one has
    \[
        \kappa_2(\gamma)=0, \quad \text{ and } \quad \kappa_2^\alpha(\gamma) = (1 - \alpha) \kappa_2^{\mathrm{LLY}}(\gamma)
        =
        \frac{1-\alpha}{n}
        \left(
        \frac{1}{N_y^{(2)}-1}-\frac{1}{N_x^{(2)}-1}
        \right).
    \]
\end{proposition}

\begin{proof}
    As in the case $r=1$, one checks that for all such $a \in J(x,x_2)$ and $b \in J(y,x_{n+2})$,
    \[
        \uptau(a,b)=n-\uptau(x,a)+\uptau(y,b).
    \]
    Hence for any causal coupling $\pi$ between measures $\mu$ and $\nu$ supported
    on $J(x,x_2)$ and $J(y,x_{n+2})$,
    \[
        \int \uptau(a,b)\diff\pi(a,b)
        =
        n-\int \uptau(x,a)\diff\mu(a)+\int \uptau(y,b)\diff\nu(b).
    \]

    For the non-idle measures, exactly one point in $J(x,x_2)$ has
    $\uptau(x,\cdot)=0$, exactly one has value $2$, and all others have value $1$.
    Therefore
    $\sum_{a\in J(x,x_2)}\uptau(x,a)=N_x$, so
    $\int \uptau(x,a)\diff\mu_x(a)=1$. Similarly,
    $\int \uptau(y,b)\diff\mu_y(b)=1$. It follows that
    $\ell_1(\mu_x,\mu_y)=n$, hence $\kappa_2(\gamma)=0$.

    For the $\alpha$-idle measures,
    \[
        \int \uptau(x,a)\diff\mu_x^\alpha(a)
        =(1-\alpha)\frac{N_x}{N_x-1},
        \qquad
        \int \uptau(y,b)\diff\mu_y^\alpha(b)
        =(1-\alpha)\frac{N_y}{N_y-1}.
    \]
    Therefore
    \[
        \ell_1(\mu_x^\alpha,\mu_y^\alpha)
        =
        n+(1-\alpha)\left(\frac{N_y}{N_y-1}-\frac{N_x}{N_x-1}\right),
    \]
    and dividing by $\uptau(x,y)=n$ gives the claimed formula for
    $\kappa_2^\alpha(\gamma)$. The formula for $\kappa_2^{\mathrm{LLY}}(\gamma)$
    follows by dividing by $1-\alpha$ and taking the limit $\alpha\to 1$.
\end{proof}

We now show that curvature along maximising geodesics satisfies a local-to-global propagation property, similar to \cite[Proposition 19]{Ollivier2009} and \cite[Lemma 2.3]{Lin2011}.

\begin{theorem}[Local-to-global principle]
    Let $\gamma$ be a maximising geodesic of a causal set $(C,\leq)$, and suppose that $\gamma=\gamma_1*\cdots *\gamma_k$ is a concatenation of subgeodesics with lengths $\geq r$. If $\kappa_r(\gamma_i)\geq K$ (resp. $\kappa^\alpha_r(\gamma_i)\geq K$, $\kappa^{\mathrm{LLY}}_r(\gamma_i)\geq K$) for all $1\leq i\leq k$, then $\kappa_r(\gamma)\geq K$ (resp. $\kappa^\alpha_r(\gamma)\geq K$, $\kappa^{\mathrm{LLY}}_r(\gamma)\geq K$).
\end{theorem}

\begin{proof}
    Write $\gamma=(x_0,\dots,x_{n+r})$ and let $0 = m_0 < m_1 < \cdots < m_k = n$
    be such that
    \[
        \gamma_i = (x_{m_{i-1}},x_{m_{i-1}+1},\dots,x_{m_i+r}).
    \]
    Since the measures $\mu_{x_{m_0}},\dots,\mu_{x_{m_k}}$ lie successively along the same maximising geodesic, the reverse triangle inequality for $\ell_1$ gives
    \begin{align*}
        \kappa_r(\gamma)
         & = \frac{\ell_1(\mu_{x_{0}},\mu_{x_{n}})}{\uptau(x_0, x_n)} - 1 \geq \frac{1}{n} \sum_{i=1}^{k}\ell_1(\mu_{x_{m_{i-1}}},\mu_{x_{m_i}})-1 \\
         & = \frac{1}{n}\sum_{i=1}^{k}(m_i-m_{i-1})\bigl(\kappa_r(\gamma_i)+1\bigr)-1  = \sum_{i=1}^{k}\frac{m_i-m_{i-1}}{n}\,\kappa_r(\gamma_i).
    \end{align*}
    The same inequality can be derived for  $\kappa_r^\alpha(\gamma)$. The claim then follows from the assumptions $\kappa_r(\gamma_i)\geq K$ (resp. $\kappa^\alpha_r(\gamma_i)\geq K$, $\kappa^{\mathrm{LLY}}_r(\gamma_i)\geq K$) for all $1\leq i\leq k$.
\end{proof}

This following results shows that the curvature along a maximising geodesic is universally bounded by the ratio between the observation scale and the separation of the two points being compared. In particular, when the observation scale is small compared with the time-separation of the two base points, all three curvature notions are automatically small.

\begin{proposition}
    Let $\gamma=(x_0,\dots,x_{n+r})$ be a maximising geodesic. Then, it holds that
    \[
        |\kappa_r(\gamma)|\leq \frac{r}{n}, \quad \text{ as well as } \quad |\kappa_r^\alpha(\gamma)|\leq \frac{r}{n} \quad \text{ and } \quad |\kappa_r^\mathrm{LLY}(\gamma)|\leq \frac{r}{n}.
    \]
\end{proposition}

\begin{proof}
    Since the arguments for $\kappa_r(\gamma)$, $\kappa_r^\alpha(\gamma)$, and $\kappa_r^{\mathrm{LLY}}(\gamma)$ are the same, we only prove the estimate for $\kappa_r(\gamma)$. As usual, set $x \coloneqq x_0$ and $y \coloneqq x_n$. For $a\in \supp(\mu_x)\subseteq J(x,x_r)$ and $b\in \supp(\mu_y)\subseteq J(y,x_{n+r})$, we have
    \[
        x \leq a \leq x_r \leq y \leq b \leq x_{n+r}.
    \]
    Since $\gamma$ is maximising, $\uptau(x_r,y)=n-r$ and $\uptau(x,x_{n+r})=n+r$. Therefore,
    \[
        n-r\leq \uptau(a,b)\leq n+r, \quad \text{ and thus } \quad n-r\leq \int \uptau(a,b)\diff\pi(a,b)\leq n+r,
    \]
    for every causal coupling $\pi$ of $\mu_x,\mu_y$. Hence $n-r\leq \ell_1(\mu_x,\mu_y)\leq n+r$, and dividing by $n$ and subtracting $1$ yields
    \[
        -\frac{r}{n}\leq \kappa_r(\gamma)\leq \frac{r}{n},
    \]
    as claimed.
\end{proof}

This suggests that, if one assumes a uniform positive lower bound on curvature, arbitrarily long maximising geodesics cannot exist. In turn, this forces the causal set to have bounded timelike diameter. This is precisely the content of the timelike Bonnet--Myers theorems that we now prove. The next three results rely on the same mechanism, in analogy with \cite{Ollivier2009}. We start with the curvature notion $\kappa_r(\gamma)$.

\begin{theorem}[Bonnet--Myers for non-idle Ollivier--Ricci curvature]
    \label{thm:BMnonidle}
    Fix $r > 2$, and suppose that $\kappa_r(\gamma)\geq K>0$ for all maximising geodesics $\gamma=(x_0,\dots,x_{n+r})$ and all $n \geq r$. Then
    \[
        \uptau(x_0,x_{n+r}) \leq r + \frac{J_{r}(\gamma)}{K},
    \]
    where
    \[
        J_{r}(\gamma) \coloneqq r - \ell_1(\delta_{x_0}, \mu_{x_0}) - \ell_1(\mu_{x_n}, \delta_{x_{n+r}}).
    \]
    Moreover, if $\mathrm{diam}_{\uptau}(\mathcal{C})\geq 2r$, then the $\uptau$-diameter of $\mathcal{C}$ is bounded by
    \[
        \mathrm{diam}_{\uptau}(\mathcal{C}) \leq r + \frac{\sup J_{r}(\gamma)}{K},
    \]
    where the supremum is taken over all maximising geodesics $\gamma=(x_0,\dots,x_{n+r})$ with $n\geq r$, and
    \begin{align*}
        J_{r}(\gamma)
         & \le
        r-2-\frac{(r-2)(r+1)}{2}
        \left(\frac{1}{N_x}+\frac{1}{N_y}\right) < r - 2.
    \end{align*}
\end{theorem}

\begin{proof}
    Let $\gamma=(x_0,\dots,x_{n+r})$ be maximising, and let $\delta_{x_0}$ and $\delta_{x_{n+r}}$ be the Dirac masses at $x_0$ and $x_{n+r}$, respectively. Then, we find that
    \begin{align*}
        \uptau(x_0,x_{n+r}) = n + r = \ell_1(\delta_{x_0},\delta_{x_{n+r}})
         & \geq \ell_1(\delta_{x_0}, \mu_{x_0}) + \ell_1(\mu_{x_0},\mu_{x_n}) + \ell_1(\mu_{x_n},\delta_{x_{n+r}}) \\
         & \geq \ell_1(\delta_{x_0}, \mu_{x_0}) + (K+1)n + \ell_1(\mu_{x_n},\delta_{x_{n+r}}).
    \end{align*}
    Rearranging yields
    \[
        \uptau(x_0,x_{n+r}) \leq r + \frac{J_{r}(\gamma)}{K}.
    \]
    Taking the supremum over all maximising geodesics $\gamma=(x_0,\dots,x_{n+r})$ gives
    \[
        \sup \uptau(x_0,x_{n+r})
        \leq
        r + \frac{\sup J_{r}(\gamma)}{K}.
    \]
    If $\mathrm{diam}_{\uptau}(\mathcal{C})\geq 2r$, \cref{eq:diam2r} yields the diameter bound.
    Since $\mu_{x_0}$ is uniform on $J(x_0,x_r)$,
    \[
        \ell_1(\delta_{x_0},\mu_{x_0}) = \sum_{z\in J(x_0,x_r)}\uptau(x_0,z)\mu_{x_0}(z)
        =
        \frac{1}{N_x}\sum_{z\in J(x_0,x_r)}\uptau(x_0,z).
    \]
    The chain points $x_0,\dots,x_r$ contribute
    \[
        \sum_{i=0}^{r}\uptau(x_0,x_i)=\sum_{i=0}^{r}i=\frac{r(r+1)}{2},
    \]
    and each additional point in $J(x_0,x_r)\setminus\{x_0,\dots,x_r\}$ contributes at least $1$. Hence
    \[
        \sum_{z\in J(x_0,x_r)}\uptau(x_0,z)
        \geq
        \frac{r(r+1)}{2} + \bigl(N_x-(r+1)\bigr)
        =
        N_x + \frac{(r-2)(r+1)}{2},
    \]
    so
    \[
        \ell_1(\delta_{x_0},\mu_{x_0})
        \geq
        1+\frac{(r-2)(r+1)}{2N_x}, \quad \text{ and } \quad \ell_1(\mu_{x_n},\delta_{x_{n+r}})
        \geq
        1+\frac{(r-2)(r+1)}{2N_y}
    \]
    by the same argument on the terminal interval. The proof is completed by substituting these bounds into the definition of $J_r(\gamma)$.
\end{proof}

We now turn to the corresponding result for the $\alpha$-idle curvature $\kappa_r^\alpha(\gamma)$.

\begin{theorem}
    Fix $r \geq 2$, and suppose that $\kappa^\alpha_r(\gamma)\geq K>0$ for all maximising geodesics $\gamma=(x_0,\dots,x_{n+r})$ and all integers $n \geq r$. Then
    \[
        \uptau(x_0,x_{n+r}) \leq r + \frac{J^\alpha_{r}(\gamma)}{K},
    \]
    where
    \[
        J_{r}^{\alpha}(\gamma) \coloneqq r - \ell_1(\delta_{x_0}, \mu_{x_0}^\alpha) - \ell_1(\mu_{x_n}^\alpha, \delta_{x_{n+r}}).
    \]
    Moreover, if $\mathrm{diam}_{\uptau}(\mathcal{C})\geq 2r$, then the $\uptau$-diameter of $\mathcal{C}$ is bounded by
    \[
        \mathrm{diam}_{\uptau}(\mathcal{C}) \leq r + \frac{\sup J_{r}^\alpha(\gamma)}{K},
    \]
    where the supremum is taken over all maximising geodesics $\gamma=(x_0,\dots,x_{n+r})$ with $n\geq r$, and
    \[
        J_{r}^\alpha(\gamma)
        \le
        \begin{cases}
            (1-\alpha)\left(
            \dfrac{1}{N_y-1}-\frac{1}{N_x-1}
            \right)\le \dfrac{1-\alpha}{2}, & r=2, \\[1mm]
            (1-\alpha)\left(
            r-2
            -\dfrac{r(r-1)}{2\left(N_x-1\right)}
            -\dfrac{r(r-3)}{2\left(N_y-1\right)}
            \right) < (1-\alpha)(r-2),      & r>2.
        \end{cases}
    \]
\end{theorem}

\begin{proof}
    Most of the proof of \cref{thm:BMnonidle} carries over here. It only remains to prove the bound for $J_r^\alpha(\gamma)$. We have
    \begin{align*}
        \ell_1(\delta_{x_0},\mu_{x_0}^\alpha) & = \sum_{z\in J(x_0,x_r)}\uptau(x_0,z)\mu_{x_0}^\alpha(z) = \alpha\,\uptau(x_0,x_0)
        +\frac{1-\alpha}{N_x-1}
        \sum_{z\in J(x_0,x_r)\setminus\{x_0\}}\uptau(x_0,z)                                                                        \\
                                              & = \frac{1-\alpha}{N_x-1}
        \sum_{z\in J(x_0,x_r)\setminus\{x_0\}}\uptau(x_0,z).
    \end{align*}
    and
    \begin{align*}
        \ell_1(\mu_{x_n}^\alpha,\delta_{x_{n+r}})
         & =
        \sum_{z\in J(x_n,x_{n+r})}
        \uptau(z,x_{n+r})\mu_{x_n}^\alpha(z) \\
         & =
        \alpha\,\uptau(x_n,x_{n+r})
        +
        \frac{1-\alpha}{N_y-1}
        \sum_{z\in J(x_n,x_{n+r})\setminus\{x_n\}}
        \uptau(z,x_{n+r})                    \\
         & = \alpha\,r
        +
        \frac{1-\alpha}{N_y-1}
        \sum_{z\in J(x_n,x_{n+r})\setminus\{x_n\}}
        \uptau(z,x_{n+r})
    \end{align*}

    As in the previous proof, the points $x_1,\dots,x_r$ contribute $r(r+1)/2$, and every additional point contributes at least $1$, hence
    \[
        \ell_1(\delta_{x_0},\mu_{x_0}^\alpha)
        \ge
        (1-\alpha)\left(
        1+\frac{r(r-1)}{2\left(N_x-1\right)}
        \right).
    \]
    On the terminal interval, the idle mass at $x_n$ contributes $\alpha r$.
    The points $x_{n+1},\dots,x_{n+r}$ contribute
    $\sum_{i=1}^{r}(r-i)=r(r-1)/2$, and every additional point contributes at least $1$, so that
    \[
        \ell_1(\mu_{x_n}^\alpha,\delta_{x_{n+r}})
        \ge
        \alpha r
        +(1-\alpha)\left(
        1+\frac{r(r-3)}{2\left(N_y-1\right)}
        \right).
    \]
\end{proof}

Finally, we consider the Bonnet--Myers estimate for the Lin--Lu--Yau curvature $\kappa_r^{\mathrm{LLY}}(\gamma)$.

\begin{theorem}
    Fix $r \geq 2$, and suppose that $\kappa_r^{\mathrm{LLY}}(\gamma)\geq K>0$ for all maximising geodesics $\gamma=(x_0,\dots,x_{n+r})$ and all integers $n \geq r$. Then
    \[
        \uptau(x_0,x_{n+r}) \leq r + \frac{J_r^{\mathrm{LLY}}(\gamma)}{K},
    \]
    where
    \[
        J_r^{\mathrm{LLY}}(\gamma)\coloneqq\lim_{\alpha\to 1}\frac{J_r^\alpha(\gamma)}{1-\alpha}.
    \]
    In particular, if $\mathrm{diam}_{\uptau}(\mathcal{C})\geq 2r$, then
    \[
        \mathrm{diam}_{\uptau}(\mathcal{C}) \leq r + \frac{\sup J_r^{\mathrm{LLY}}(\gamma)}{K},
    \]
    where the supremum is taken over all maximising geodesics $\gamma=(x_0,\dots,x_{n+r})$ with $n\geq r$, and
    \[
        J_r^{\mathrm{LLY}}(\gamma)
        \le
        \begin{cases}
            \dfrac{1}{N_y-1}-\dfrac{1}{N_x-1}\le \dfrac{1}{2}, & r=2, \\[1mm]
            r-2
            -\dfrac{r(r-1)}{2\left(N_x-1\right)}
            -\dfrac{r(r-3)}{2\left(N_y-1\right)} < r-2,        & r>2.
        \end{cases}
    \]
\end{theorem}

\begin{proof}
    For every $\alpha\in[0,1)$, the argument above yields
    \[
        n\,\kappa_r^\alpha(\gamma)\le J_r^\alpha(\gamma).
    \]
    Dividing by $1-\alpha$ and passing to the limit $\alpha\to 1$, we obtain
    \[
        n\,\kappa_r^{\mathrm{LLY}}(\gamma)\le J_r^{\mathrm{LLY}}(\gamma).
    \]
    Using $\kappa_r^{\mathrm{LLY}}(\gamma)\ge K$, we get
    \[
        n\le \frac{J_r^{\mathrm{LLY}}(\gamma)}{K},
    \]
    and therefore
    \[
        \uptau(x_0,x_{n+r})=n+r\le r+\frac{J_r^{\mathrm{LLY}}(\gamma)}{K}.
    \]
    The bounds on $J_r^{\mathrm{LLY}}(\gamma)$ follow from those on $J_r^\alpha(\gamma)$ and by taking a limit $\alpha\to1$.
\end{proof}

\subsection{Explicit examples of Ollivier--Ricci curvature on causal sets}
\label{subsection:explicitexamples}

We begin by recording a simple but useful formula for the Ollivier--Ricci curvature of a causal set. Assume that there exist functions
\[
    f:J(x,x_r)\to\mathbb R,
    \qquad
    g:J(y,x_{n+r})\to\mathbb R
\]
such that, for every $p\in J(x,x_r)$ and $q\in J(y,x_{n+r})$,
\begin{equation}
    \label{eq:affine}
    \uptau(p,q)=\uptau(x,y)+g(q)-f(p).
\end{equation}
Then, if $\pi$ is any causal coupling between $\mu_x$ and $\mu_y$, we have
\begin{align*}
    \ell_1(\mu_x,\mu_y)
     & =
    \uptau(x,y)
    +\sum_{q\in J(y,x_{n+r})}g(q)\mu_y(q)
    -\sum_{p\in J(x,x_r)}f(p)\mu_x(p) \\
     & =
    \uptau(x,y)
    +\frac{1}{N_y}\sum_{q\in J(y,x_{n+r})}g(q)
    -\frac{1}{N_x}\sum_{p\in J(x,x_r)}f(p),
\end{align*}
and therefore
\begin{equation}
    \label{eq:affinenonidle}
    \kappa_r(\gamma)
    =
    \frac{1}{\uptau(x,y)}
    \left(
    \frac{1}{N_y}\sum_{q\in J(y,x_{n+r})}g(q)
    -\frac{1}{N_x}\sum_{p\in J(x,x_r)}f(p)
    \right).
\end{equation}

The same computation applies to the $\alpha$-idle measures. Using the fact that
the affine identity \cref{eq:affine}, evaluated at $(p,q)=(x,y)$, gives
$g(y)-f(x)=0$, the two Dirac masses at $x$ and $y$ contribute only the
base term $\uptau(x,y)$. Hence one finds
\begin{equation}
    \label{eq:affinealpha}
    \kappa_r^\alpha(\gamma)
    =
    \frac{1-\alpha}{\uptau(x,y)}
    \left(
    \frac{1}{N_y-1}\sum_{q\in J(y,x_{n+r})\setminus\{y\}}g(q)
    -\frac{1}{N_x-1}\sum_{p\in J(x,x_r)\setminus\{x\}}f(p)
    \right),
\end{equation}
and thus
\begin{equation}
    \label{eq:affineLLY}
    \kappa_r^{\mathrm{LLY}}(\gamma)
    =
    \frac{1}{\uptau(x,y)}
    \left(
    \frac{1}{N_y-1}\sum_{q\in J(y,x_{n+r})\setminus\{y\}}g(q)
    -\frac{1}{N_x-1}\sum_{p\in J(x,x_r)\setminus\{x\}}f(p)
    \right).
\end{equation}

\begin{example}[Small Bonnet-Myers sharp positively curved causal set]
    \begin{figure}[htbp]
        \centering
        \begin{subfigure}{0.45\textwidth}
            \centering
            \begin{tikzpicture}[
                    >=Stealth,
                    line cap=round, line join=round,
                    dot/.style = {circle, fill=black, inner sep=1.4pt, outer sep=0.5pt},
                    every edge/.style = {->, draw=gray!90}
                ]
                \pgfmathsetmacro{\scaleFactor}{0.7}  
                \pgfmathsetmacro{\halfGap}{2.0}      
                \pgfmathsetmacro{\dy}{1.2}           

                \begin{scope}[scale=\scaleFactor]
                    \pgfmathsetmacro{\xL}{-\halfGap}
                    \pgfmathsetmacro{\xR}{\halfGap}

                    \node[dot,red,label=left:{$x^{1}_{1}$}] (A1) at (\xL, {0*\dy}) {};
                    \node[dot,red,label=left:{$x^{1}_{2}$}] (A2) at (\xL, {1*\dy}) {};
                    \node[dot,purple,label=left:{$x^{1}_{3}$}] (A3) at (\xL, {2*\dy}) {};
                    \node[dot,blue,label=left:{$x^{1}_{4}$}] (A4) at (\xL, {3*\dy}) {};
                    \node[dot,blue,label=left:{$x^{1}_{5}$}] (A5) at (\xL, {4*\dy}) {};

                    \node[dot,red,label=right:{$x^{2}_{1}$}] (B1) at (\xR, {0*\dy}) {};
                    \node[dot,label=right:{$x^{2}_{2}$}] (B2) at (\xR, {1*\dy}) {};
                    \node[dot,label=right:{$x^{2}_{3}$}] (B3) at (\xR, {2*\dy}) {};
                    \node[dot,label=right:{$x^{2}_{4}$}] (B4) at (\xR, {3*\dy}) {};
                    \node[dot,blue,label=right:{$x^{2}_{5}$}] (B5) at (\xR, {4*\dy}) {};

                    \node[dot,red,label=below:{$p$}] (P) at (0, {-0.75*\dy}) {};
                    \node[dot,blue,label=above:{$q$}] (Q) at (0, {5*\dy + 0.25*\dy}) {};

                    \draw (A1) edge (A2) (A2) edge (A3) (A3) edge (A4) (A4) edge (A5);
                    \draw (B1) edge (B2) (B2) edge (B3) (B3) edge (B4) (B4) edge (B5);

                    \draw (P) edge (A1);
                    \draw (P) edge (B1);

                    \draw (A1) edge (B3); 
                    \draw (B1) edge (A3); 
                    \draw (B3) edge (A5); 
                    \draw (A3) edge (B5); 

                    \draw (A5) edge (Q);
                    \draw (B5) edge (Q);
                \end{scope}
            \end{tikzpicture}
            \caption{}
            \label{fig:firstPosCurvExample}
        \end{subfigure}
        \hfill
        \begin{subfigure}{0.45\textwidth}
            \centering
            \begin{tikzpicture}[
                    >=Stealth,
                    line cap=round, line join=round,
                    dot/.style = {circle, fill=black, inner sep=1.4pt, outer sep=0.5pt},
                    every edge/.style = {->, very thin, draw=gray!90}
                ]
                \pgfmathsetmacro{\scaleFactor}{0.7}  
                \pgfmathsetmacro{\halfGap}{0.7}      
                \pgfmathsetmacro{\dy}{1.2}           

                \begin{scope}[scale=\scaleFactor]
                    \pgfmathsetmacro{\xL}{-\halfGap}
                    \pgfmathsetmacro{\xR}{\halfGap}
                    \pgfmathsetmacro{\xRR}{3*\halfGap}
                    \pgfmathsetmacro{\xRRR}{5*\halfGap}

                    \node[dot,label=below:{$x^{1}_{i}$}] (A1) at (\xL, {0*\dy}) {};
                    \node[dot] (A2) at (\xL, {1*\dy}) {};
                    \node[dot] (A3) at (\xL, {2*\dy}) {};
                    \node[dot] (A4) at (\xL, {3*\dy}) {};
                    \node[dot] (A5) at (\xL, {4*\dy}) {};

                    \node[dot,label=left:{$x^{2}_{i}$}] (B1) at (\xR, {0*\dy}) {};
                    \node[dot] (B2) at (\xR, {1*\dy}) {};
                    \node[dot] (B3) at (\xR, {2*\dy}) {};
                    \node[dot] (B4) at (\xR, {3*\dy}) {};
                    \node[dot] (B5) at (\xR, {4*\dy}) {};

                    \node[dot,label=right:{$x^{3}_{i}$}] (C1) at (\xRR, {0*\dy}) {};
                    \node[dot] (C2) at (\xRR, {1*\dy}) {};
                    \node[dot] (C3) at (\xRR, {2*\dy}) {};
                    \node[dot] (C4) at (\xRR, {3*\dy}) {};
                    \node[dot] (C5) at (\xRR, {4*\dy}) {};

                    \node[dot,label=below:{$x^{4}_{i}$}] (D1) at (\xRRR, {0*\dy}) {};
                    \node[dot] (D2) at (\xRRR, {1*\dy}) {};
                    \node[dot] (D3) at (\xRRR, {2*\dy}) {};
                    \node[dot] (D4) at (\xRRR, {3*\dy}) {};
                    \node[dot] (D5) at (\xRRR, {4*\dy}) {};

                    \node[dot,label=below:{$p$}] (P) at (2*\halfGap, {-0.75*\dy}) {};
                    \node[dot,label=above:{$q$}] (Q) at (2*\halfGap, {5*\dy + 0.25*\dy}) {};

                    \draw (A1) edge (A2) (A2) edge (A3) (A3) edge (A4) (A4) edge (A5);
                    \draw (B1) edge (B2) (B2) edge (B3) (B3) edge (B4) (B4) edge (B5);
                    \draw (C1) edge (C2) (C2) edge (C3) (C3) edge (C4) (C4) edge (C5);
                    \draw (D1) edge (D2) (D2) edge (D3) (D3) edge (D4) (D4) edge (D5);

                    \draw (P) edge (A1);
                    \draw (P) edge (B1);
                    \draw (P) edge (C1);
                    \draw (P) edge (D1);

                    \draw (A1) edge (B3);
                    \draw (B1) edge (A3);
                    \draw (B3) edge (A5);
                    \draw (A3) edge (B5);
                    \draw (A1) edge (C3);
                    \draw (C1) edge (A3);
                    \draw (C3) edge (A5);
                    \draw (A3) edge (C5);
                    \draw (A1) edge (D3);
                    \draw (D1) edge (A3);
                    \draw (D3) edge (A5);
                    \draw (A3) edge (D5);

                    \draw (B1) edge (C3);
                    \draw (C1) edge (B3);
                    \draw (C3) edge (B5);
                    \draw (B3) edge (C5);
                    \draw (B1) edge (D3);
                    \draw (D1) edge (B3);
                    \draw (D3) edge (B5);
                    \draw (B3) edge (D5);

                    \draw (C1) edge (D3);
                    \draw (D1) edge (C3);
                    \draw (D3) edge (C5);
                    \draw (C3) edge (D5);

                    \draw (A5) edge (Q);
                    \draw (B5) edge (Q);
                    \draw (C5) edge (Q);
                    \draw (D5) edge (Q);
                \end{scope}
            \end{tikzpicture}
            \caption{}
            \label{fig:secondPosCurvExample}
        \end{subfigure}
        \caption{Spaces of constant positive Ollivier Ricci curvature}
        \label{fig:bothPosCurvExample}
    \end{figure}
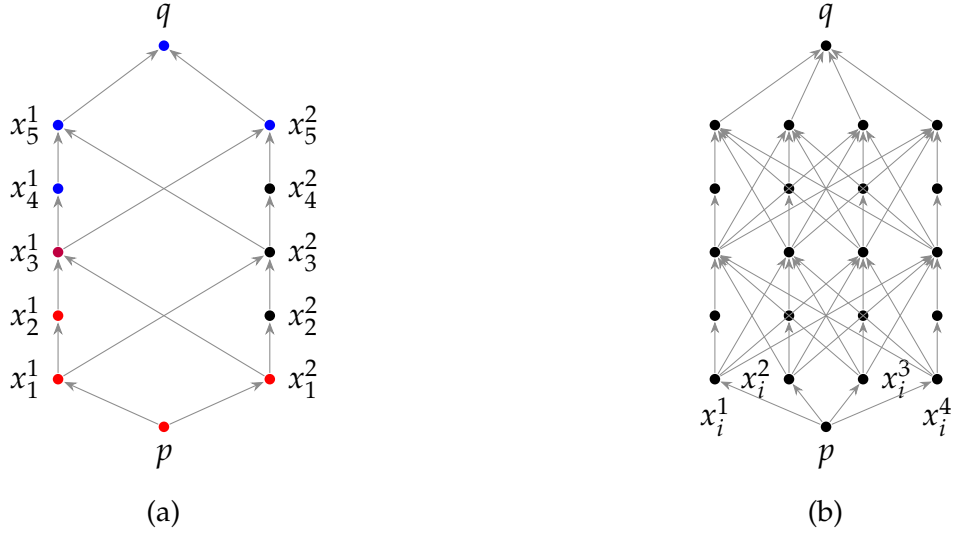
    We consider the causal set defined by \cref{fig:firstPosCurvExample}. For the non-idle curvature, the cases $r<3$ are trivial. Although $r=2$ can contribute to the $\alpha$-idle and Lin--Lu--Yau curvatures in general, one can easily see that $N_x=N_y=3$ for every admissible configuration, so \cref{prop:r1,prop:r2} give $\kappa_2^\alpha(\gamma)=\kappa_2^{\mathrm{LLY}}(\gamma)=0$. Thus the first relevant scale here is $r=3$. The causal set has timelike diameter $6$, while any admissible geodesic satisfies $n\ge r$ and has length $n+r$. Hence, for $r=3$, we must have $6\ge n+r\ge 6$, so $n=r=3$. For $r>3$, no admissible geodesic fits inside the diameter. It is therefore enough to focus on the case $n=r=3$: any admissible geodesic has length $n+r=6$, and hence must run from $p$ to $q$ along one of the two full vertical chains. These two choices are exchanged by the $\mathbb Z_2$-symmetry of the causal set, so it suffices to consider
    \[
        \gamma=(p,x_1^1,x_2^1,x_3^1,x_4^1,x_5^1,q).
    \]
    The two measures in \cref{def:OllivierRicciCST} are supported on $J(p,x_3^1)=\{p,x_1^1,x_2^1,x_3^1,x_1^2\}$ and $J(x_3^1,q)=\{x_3^1,x_4^1,x_5^1,x_5^2,q\}$. The transport-cost matrix is given by
    \[
        \begin{array}{c|ccccc}
            \uptau(\cdot,\cdot)
                  & x_3^1 & x_4^1 & x_5^1 & x_5^2 & q \\ \hline
            p     & 3     & 4     & 5     & 5     & 6 \\
            x_1^1 & 2     & 3     & 4     & 3     & 5 \\
            x_2^1 & 1     & 2     & 3     & 2     & 4 \\
            x_3^1 & 0     & 1     & 2     & 1     & 3 \\
            x_1^2 & 1     & 2     & 3     & 4     & 5
        \end{array}
    \]
    Optimal couplings, which follows easily from \cref{theorem:KantorovichDuality2}, between $\mu_x$ and $\mu_y$, and between $\mu_x^\alpha$ and $\mu_y^\alpha$, respectively, are given by
    \[
        p\xmapsto{\frac15} x_3^1,\quad
        x_1^1\xmapsto{\frac15} x_4^1,\quad
        x_2^1\xmapsto{\frac15} x_5^1,\quad
        x_3^1\xmapsto{\frac15} q,\quad
        x_1^2\xmapsto{\frac15} x_5^2,
    \]
    and
    \[
        p\xmapsto{\alpha} x_3^1,\quad
        x_1^1\xmapsto{\frac{1-\alpha}{4}} x_4^1,\quad
        x_2^1\xmapsto{\frac{1-\alpha}{4}} x_5^1,\quad
        x_3^1\xmapsto{\frac{1-\alpha}{4}} q,\quad
        x_1^2\xmapsto{\frac{1-\alpha}{4}} x_5^2.
    \]
    This immediately gives $\kappa_3(\gamma)=\frac{1}{15}$, $\kappa_3^\alpha(\gamma)=\frac{1-\alpha}{12}$, and hence $\kappa_3^{\mathrm{LLY}}(\gamma)=\frac{1}{12}$. Moreover,
    \[
        J_3(\gamma)
        =
        3-\ell_1(\delta_p,\mu_x)-\ell_1(\mu_y,\delta_q)
        =
        3-\frac75-\frac75
        =
        \frac15.
    \]
    Similarly, $J_3^\alpha(\gamma)=\frac{1-\alpha}{4}$ and $J_3^{\mathrm{LLY}}(\gamma)=\frac14$. Thus the causal set attains equality in the Bonnet--Myers bounds for all three notions of Ollivier--Ricci curvature.
\end{example}

\begin{example}[A generalisation of the previous example]
    We generalise the previous construction by adding more parallel chains; see \cref{fig:secondPosCurvExample}. For $N\geq 2$, define a causal set of timelike diameter $6$ with elements
    \[
        p,\quad q,\quad x_i^k \qquad (k=1,\ldots,N,\ i=1,\ldots,5).
    \]
    with order generated by
    \[
        p\ll x_1^k\ll x_2^k\ll x_3^k\ll x_4^k\ll x_5^k\ll q
        \qquad (k=1,\ldots,N),
    \]
    together with the relations
    \[
        x_1^i\ll x_3^j,\qquad x_3^i\ll x_5^j
        \qquad (i\neq j).
    \]
    As before, the first relevant scale is $r=3$, and by symmetry it suffices to consider the geodesic
    \[
        \gamma=(p,x_1^1,x_2^1,x_3^1,x_4^1,x_5^1,q).
    \]
    The relevant diamonds are
    \[
        J(p,x_3^1)=\{p,x_1^1,x_2^1,x_3^1,x_1^2,\ldots,x_1^N\}, \quad \text{ and } \quad J(x_3^1,q)=\{x_3^1,x_4^1,x_5^1,q,x_5^2,\ldots,x_5^N\}.
    \]
    Ordering rows and columns in this way, the transport-cost matrix is
    \[
        \begin{array}{c|cccc|ccc}
            \uptau(\cdot,\cdot)
                   & x_3^1  & x_4^1  & x_5^1  & q      & x_5^2  & \cdots & x_5^N  \\ \hline
            p      & 3      & 4      & 5      & 6      & 5      & \cdots & 5      \\
            x_1^1  & 2      & 3      & 4      & 5      & 3      & \cdots & 3      \\
            x_2^1  & 1      & 2      & 3      & 4      & 2      & \cdots & 2      \\
            x_3^1  & 0      & 1      & 2      & 3      & 1      & \cdots & 1      \\ \hline
            x_1^2  & 1      & 2      & 3      & 5      & 4      & \cdots & 3      \\
            \vdots & \vdots & \vdots & \vdots & \vdots & \vdots & \ddots & \vdots \\
            x_1^N  & 1      & 2      & 3      & 5      & 3      & \cdots & 4
        \end{array}.
    \]
    An optimal coupling is given by the matching
    \[
        p\xmapsto{\frac{1}{N+3}} x_3^1,\quad
        x_1^1\xmapsto{\frac{1}{N+3}} x_4^1,\quad
        x_2^1\xmapsto{\frac{1}{N+3}} x_5^1,\quad
        x_3^1\xmapsto{\frac{1}{N+3}} q,\quad
        x_1^j\xmapsto{\frac{1}{N+3}} x_5^j\quad (j=2,\ldots,N).
    \]
    For the $\alpha$-idle measures, the analogous coupling is
    \[
        p\xmapsto{\alpha} x_3^1,\quad
        x_1^1\xmapsto{\frac{1-\alpha}{N+2}} x_4^1,\quad
        x_2^1\xmapsto{\frac{1-\alpha}{N+2}} x_5^1,\quad
        x_3^1\xmapsto{\frac{1-\alpha}{N+2}} q,\quad
        x_1^j\xmapsto{\frac{1-\alpha}{N+2}} x_5^j\quad (j=2,\ldots,N).
    \]
    Optimality follows as in the previous example from \cref{theorem:KantorovichDuality2}. Hence we find that
    \[
        \kappa_3(\gamma)
        =
        \frac{N-1}{3(N+3)}, \quad \kappa_3^\alpha(\gamma)
        =
        \frac{1-\alpha}{3(N+2)}(N-1),
        \quad \text{ and } \quad
        \kappa_3^{\mathrm{LLY}}(\gamma)
        =
        \frac{N-1}{3(N+2)}.
    \]
    These curvatures are positive for $N\geq2$ (and $\alpha<1$ in the $\alpha$-idle case), are monotonically increasing in $N$.
    Similarly to the previous example, we can also compute
    \[
        \ell_1(\delta_p,\mu_x)=\ell_1(\mu_y,\delta_q)=\frac{N+5}{N+3},
    \]
    and
    \[
        \ell_1(\mu_y^\alpha,\delta_q)
        =
        3\alpha+\frac{1-\alpha}{N+2}(N+2), \quad \text{ and } \quad \ell_1(\mu_y^\alpha,\delta_q)
        =
        3\alpha+\frac{1-\alpha}{N+2}(N+2).
    \]
    Therefore, it holds that
    \[
        J_3(\gamma) =
        \frac{N-1}{N+3}, \quad J_3^\alpha(\gamma)
        =
        \frac{1-\alpha}{N+2}(N-1), \quad \text{ and } \quad
        J_3^{\mathrm{LLY}}(\gamma)
        =
        \frac{N-1}{N+2},
    \]
    and we conclude that this family attains equality in the three Bonnet--Myers bounds.
\end{example}

\begin{example}[Graded causal sets]
    A causal set $(C,\leq)$ is \emph{graded} if it admits a rank function $\rho:C\to\mathbb N$ such that, whenever $x \ll y$, one has $\rho(x)<\rho(y)$, and whenever $y$ is an immediate successor of $x$, one has $\rho(y)=\rho(x)+1$. Under that definition, for comparable $p\leq q$ one gets
    \[
        \uptau(p,q)=\rho(q)-\rho(p)=
        \uptau(x,y)+\uptau(y,q)-\uptau(x,p)
    \]
    for all $x, y \in C$ such that $x\leq p\leq q$ and $x\leq y\leq q$, so that \cref{eq:affine} holds.

    Let $\gamma=(x_0,\dots,x_{n+r})$ be a maximising geodesic, and set $x\coloneqq x_0$ and $y\coloneqq x_n$ as usual. The affine formulas \cref{eq:affinenonidle,eq:affinealpha,eq:affineLLY} imply that
    \begin{equation}
        \label{eq:gradednonidle}
        \kappa_r(\gamma)
        =
        \frac{1}{n}
        \left(
        \frac{1}{N_y}\sum_{i=1}^{r} i\,N_y(i)
        -
        \frac{1}{N_x}\sum_{i=1}^{r} i\,N_x(i)
        \right).
    \end{equation}
    and
    \begin{equation}
        \label{eq:gradedidle}
        \kappa_r^\alpha(\gamma)
        =
        (1-\alpha)\kappa_r^{\mathrm{LLY}}(\gamma)
        =
        \frac{1-\alpha}{n}
        \left(
        \frac{1}{N_y-1}\sum_{i=1}^{r} i\,N_y(i)
        -
        \frac{1}{N_x-1}\sum_{i=1}^{r} i\,N_x(i)
        \right),
    \end{equation}
    where we have set, for $0\leq i\leq r$, $N_x(i)\coloneqq\left|\{p\in J(x,x_r):\rho(p)-\rho(x)=i\}\right|$ and $N_y(i)\coloneqq\left|\{q\in J(y,x_{n+r}):\rho(q)-\rho(y)=i\}\right|$. All the remaining examples of this section are graded causal sets.
\end{example}

\begin{example}[Complete layered causal set]
    Let $C$ be a finite causal set of the form
    \[
        C=L_1\sqcup L_2\sqcup\cdots\sqcup L_K,
    \]
    where the chronological relation is defined by
    \[
        x\ll y
        \quad\Longleftrightarrow\quad
        x\in L_i,\ y\in L_j,\ \text{and } i<j.
    \]
    This is the causal-set analogue of a complete multipartite poset: points in the same layer are incomparable, while every point in a lower layer is causally below every point in a higher layer.

    This causal set is graded by the rank function $\rho(z)=i$ for $z\in L_i$. Note that if $u\in L_s$ and $v\in L_t$ with $s<t$, then a causal curve from $u$ to $v$ is maximising if and only if it contains exactly one point in each layer $L_s,\ldots,L_t$. Let $\gamma=(x_0,\ldots,x_{n+r})$ be such a maximising geodesic, with
    $x_k\in L_{s+k}$ for all $0\leq k\leq n+r$. Set $x\coloneqq x_0$ and $y\coloneqq x_n$.
    Since $J(x,x_r)=\{x,x_r\}\cup L_{s+1}\cup\cdots\cup L_{s+r-1}$ and $J(y,x_{n+r})=\{y,x_{n+r}\}\cup L_{s+n+1}\cup\cdots\cup L_{s+n+r-1}$, the graded formulas \cref{eq:gradednonidle,eq:gradedidle} give
    \[
        \kappa_r(\gamma)
        =
        \frac{1}{n}
        \left(
        \frac{r+\sum_{i=1}^{r-1} i\,|L_{s+n+i}|}{2+\sum_{i=1}^{r-1}|L_{s+n+i}|}
        -
        \frac{r+\sum_{i=1}^{r-1} i\,|L_{s+i}|}{2+\sum_{i=1}^{r-1}|L_{s+i}|}
        \right),
    \]
    while
    \[
        \kappa_r^\alpha(\gamma)
        = (1 - \alpha)\kappa_r^{\mathrm{LLY}}(\gamma) =
        \frac{1-\alpha}{n}
        \left(
        \frac{r+\sum_{i=1}^{r-1} i\,|L_{s+n+i}|}{1+\sum_{i=1}^{r-1}|L_{s+n+i}|}
        -
        \frac{\sum_{i=1}^{r-1} i\,|L_{s+i}|}{1+\sum_{i=1}^{r-1}|L_{s+i}|}
        \right).
    \]
\end{example}

\begin{example}[Positively-curved complete layered causal set]
    The complete-layered formula can be used recursively to search for examples of prescribed non-idle curvature. Writing
    \[
        n_i\coloneqq|L_i|,
        \qquad
        F_s\coloneqq
        \frac{r+\sum_{i=1}^{r-1} i\,n_{s+i}}{2+\sum_{i=1}^{r-1}n_{s+i}},
    \]
    the previous example gives
    \[
        \kappa_r(\gamma)=\frac{F_{s+n}-F_s}{n}.
    \]
    Thus, for fixed $r$ and $n$, prescribing $\kappa_r(\gamma)=K$ for all admissible geodesics $\gamma=(x_0,\ldots,x_{n+r})$ with $x_k\in L_{s+k}$ is equivalent to imposing $F_{s+n}=F_s+nK$ for all admissible values of $s$.  Starting at $s=1$ from suitable initial layer sizes $n_1,\ldots,n_{n+r-1}$, this relation can be used recursively to determine successive layer sizes, as long as no denominator vanishes and the resulting numbers remain positive integers.

    For instance, take $n=r=3$ and $K=\frac{1}{45}$. Then
    \[
        F_s=\frac{3+n_{s+1}+2n_{s+2}}{2+n_{s+1}+n_{s+2}},
        \qquad
        F_{s+3}=F_s+\frac{1}{15}.
    \]
    The following layer sizes give a positive-curvature complete layered causal set with $\kappa_3(\gamma)=\frac{1}{45}$ for each admissible geodesic in this range:
    \[
        \begin{multlined}
            (n_1,n_2,\ldots,n_{23})
            =
            (1,20,23,20,13,20,23,20,41,62,71,197,395,593,2375,\\
            6533,13067,84941,339767,934361,13081067,85026941,340107767).
        \end{multlined}
    \]
    Attempting to continue this particular recursion by solving for $n_{24}$ breaks down: the next target value is $F_{22}=2$, so the coefficient of $n_{24}$ in the equation for $F_{22}$ vanishes.
\end{example}

\begin{example}[Product causal set]
    Let $(X,d)$ be a locally finite geodesic metric space, and let $\lambda=(\lambda_m)_{m\in\mathbb Z}$ be a sequence of positive real numbers. Consider the causal set
    \[
        C_\lambda(X)\coloneqq(\mathbb Z\times X,\leq).
    \]
    The strict relation $\ll$ is generated by the one-step relations defined by the following property: $(t,u)\ll(t+1,v)$ if and only if $d(u,v)\leq \lambda_t$. Equivalently, for $t<s$,
    \[
        (t,u)\ll(s,v) \quad
        \text{ if and only if} \quad
        d(u,v)\leq \sum_{m=t}^{s-1}\lambda_m .
    \]
    Since $X$ is locally finite, all metric balls of finite radius are finite, and hence the causal intervals are finite. It is a graded causal set with rank function $\rho(t,u)=t$, since each one-step relation increases the time coordinate by one.

    The formulas \cref{eq:gradednonidle,eq:gradedidle} apply, where we can write $N_x(i)$ and $N_y(i)$ more explicitly as intersections of metric balls in $X$:
    for $0\le i\le r$,
    \[
        N_x(i)
        =
        \left|
        B_X\left(z_0,\sum_{m=t_0}^{t_0+i-1}\lambda_m\right)
        \cap
        B_X\left(z_r,\sum_{m=t_0+i}^{t_0+r-1}\lambda_m\right)
        \right|,
    \]
    and
    \[
        N_y(i)
        =
        \left|
        B_X\left(z_n,\sum_{m=t_0+n}^{t_0+n+i-1}\lambda_m\right)
        \cap
        B_X\left(z_{n+r},\sum_{m=t_0+n+i}^{t_0+n+r-1}\lambda_m\right)
        \right|.
    \]

\end{example}

\begin{example}[Locally finite Boolean algebra]
    Every locally finite Boolean algebra is finite, and hence isomorphic to a finite Boolean lattice $B_m$. Here $B_m$ is the set of all subsets of $\{1,\dots,m\}$, ordered by inclusion: $B_m\coloneqq\{A\mid A\subseteq\{1,\dots,m\}\}$, with $A\leq B$ if and only if $A\subseteq B$. This is a graded causal set with rank function $\rho(A)=|A|$. Therefore, using \cref{eq:gradednonidle,eq:gradedidle}, one obtains $\kappa_r(\gamma)=\kappa_r^\alpha(\gamma)=\kappa_r^{\mathrm{LLY}}(\gamma)=0$.
\end{example}

\begin{example}[Coxeter groups]
    A \emph{Coxeter system} $(W,S)$ consists of a group $W$ and a generating set
    $S$ such that $W=\langle S\mid (st)^{m_{st}}=1\rangle$, where $m_{ss}=1$ and
    $m_{st}=m_{ts} \geq 2$ is either an integer or $\infty$ for $t \neq s$, see \cite{Brenti2005}. The Coxeter length of an element $w\in W$ is
    \[
        \ell(w)\coloneqq\min\{k\geq 0 : w=s_1s_2\cdots s_k \text{ for some } s_i\in S\}.
    \]
    The \emph{(right) weak order} on $(W, S)$ is defined by
    \[
        u\leq_{\mathrm{weak}} v
        \quad
        \text{ if and only if} \quad
        v=uw
        \text{ and }
        \ell(v)=\ell(u)+\ell(w)
        \text{ for some } w\in W.
    \]
    The \emph{strong order}, usually called the \emph{Bruhat order}, is defined by $u\leq_{\mathrm{strong}} v$ if a reduced expression for $v$ contains, after deleting some letters, a reduced expression for $u$; that is to say, if there are generators $s_1,\ldots,s_k\in S$ and indices $1\leq i_1<\cdots<i_\ell\leq k$ such that $v=s_1\cdots s_k$, $k=\ell(v)$, $u=s_{i_1}\cdots s_{i_\ell}$, and $\ell=\ell(u)$.

    For any Coxeter system $(W,S)$, both the weak and the strong order are locally finite posets. We write $(C,\leq)$ for the corresponding causal set, where $\leq$ denotes either one of the weak orders or the Bruhat order. Both orders are graded by the Coxeter length $\ell$ and therefore induce the same formula for the time-separation: for comparable elements $u\leq v$ in both ordering systems, one has $\uptau(u,v)=\ell(v)-\ell(u)$. Note that the weak and Bruhat orders give the same formula for the
    time-separation on comparable pairs. However, Bruhat order has more comparable
    pairs, so the resulting time-separation functions differ precisely on pairs
    that are Bruhat-comparable but weak-incomparable: for such pairs,
    $\uptau_{\mathrm{Bruhat}}(u,v)>0$, whereas
    $\uptau_{\mathrm{weak}}(u,v)=0$.
    The formula for the Ollivier-Ricci curvature of graded causal sets \cref{eq:gradednonidle,eq:gradedidle} again apply here, with
    \[
        N_x(i)
        =
        \left|\{z\in J(x_0,x_r):\ell(z)=\ell(x_0)+i\}\right|,
    \]
    and
    \[
        N_y(i)
        =
        \left|\{z\in J(x_n,x_{n+r}):\ell(z)=\ell(x_n)+i\}\right|.
    \]

\end{example}

\begin{example}[Young's lattice] \emph{Young's lattice} $\mathbb Y$ is the poset whose elements are integer partitions, that is, finite non-increasing sequences of positive integers,
    \[
        \mathbb Y
        \coloneqq
        \left\{
        (\lambda_1,\ldots,\lambda_k)
        \;\middle|\;
        k\in\mathbb N,\
        \lambda_i\in\mathbb N,\
        \lambda_1\geq\cdots\geq\lambda_k>0
        \right\}.
    \]
    Given two partitions $\lambda$ and $\mu$ in $\mathbb Y$, we extend the shorter one by zeros and define $\lambda\leq\mu$ if and only if $\lambda_i\leq\mu_i$ for every $i$. With this order, $(\mathbb Y,\leq)$ is a graded causal set with rank $\rho(\lambda)=|\lambda|\coloneqq\sum_i\lambda_i$, so its Ollivier--Ricci curvatures are given by \cref{eq:gradednonidle,eq:gradedidle}.
\end{example}

\begin{example}[Rooted-tree posets]
    A \emph{rooted tree $T$ with root $o$} is a connected acyclic graph together with a distinguished vertex $o\in V(T)$. We define a partial order $\leq$ on $T$ by declaring $u\leq v$ if and only if $u$ lies on the unique simple path from the root $o$ to $v$. If every root-to-vertex path is finite, this gives a causal set. It is graded by the depth function, which counts the number of edges on the unique simple path from $o$ to $v$. However, every interval $J(u,v)$ is totally ordered, so for any maximising geodesic $\gamma$ we have $N_x(i)=N_y(i)=1$ for all $0\leq i\leq r$, and consequently $\kappa_r(\gamma)=\kappa_r^\alpha(\gamma)=\kappa_r^{\mathrm{LLY}}(\gamma)=0$.
\end{example}

\subsection{Numerical examples}
\label{subsection:numerical}

We now describe a numerical procedure for approximating the quantity
\[
    \frac{\ell_1(\mu_x,\mu_y)}{\delta}-1
\]
from a Poisson sprinkling of a Lorentzian manifold; see
\cref{subsec:CausalSet}. After rescaling by $\epsilon^{-2}$ and by the
dimensional constant appearing in \cref{eq:MainTheorem}, this quantity is intended to approximate $\Ric(v,v)$, where $v$ is the tangent vector to the geodesic joining $x$ to $y$. The algorithm goes as follows.

\begin{enumerate}[label=(\roman*), wide=0pt, itemsep=2pt, topsep=2pt]
    \item \textbf{Choose a coordinate chart.}
          The spacetime region is described in a coordinate chart
          $\phi:B\subseteq\mathbb R^d\to M$, where
          $B=\interval{a_1}{b_1}\times\cdots\times\interval{a_d}{b_d}$ and $\Omega\subseteq\phi(B)$ denotes the spacetime region to be sprinkled.

    \item \textbf{Sprinkle the spacetime region.}
          We want points distributed in $\widetilde\Omega=\phi^{-1}(\Omega)$ with density
          proportional to $w(u)\diff u_1\dots\diff u_d=\phi^*\mathrm{vol}_g$. Thus we
          oversample uniformly in $B$ by choosing an upper bound $M\geq \sup_B w(u)$.
          Computationally, we generate a random integer
          $N_{\mathrm{cand}}\sim \operatorname{Poisson}(\rho M |B|)$ and then generate
          $N_{\mathrm{cand}}$ independent uniformly distributed candidate points in $B$.
          We then accept a candidate point $u$ with probability $w(u)/M$, rejecting it
          automatically if $u\notin\widetilde\Omega$. In particular, the number $N$ of
          accepted points is Poisson distributed with mean
          $\rho\int_{\widetilde\Omega} w(u)\,\diff u_1\dots\diff u_d
              =\rho\,\mathrm{vol}_g(\Omega)$. The partial order is inherited from the spacetime causal order.
    \item \textbf{Choice of the time-separation function.}
          Depending on the numerical experiment, there are several possible choices.
          The most intrinsic one is the counting time-separation, defined as the
          length of a longest chain from $p$ to $q$, as in \cref{eq:diam2r}.
          Computationally, the counting time-separation is obtained by finding longest
          paths in the directed acyclic graph. In the
          two-dimensional product-order case, where the chart is such that
          $\phi^{-1}(p)=(u(p),v(p))$ and $p\leq_C q$ if and only if
          $u(p)\leq u(q)$ and $v(p)\leq v(q)$, we use the much
          faster longest-increasing-subsequence-type algorithm. When we want to use more of the background spacetime geometry, we instead
          take the restricted manifold time-separation between causally related
          sprinkled points.
    \item \textbf{Select the endpoints.}
          One option is to fix $x\in\Omega$ and a future-directed unit timelike
          vector $v\in\T_x(M)$, and then add the two points
          $x$ and $x'=\exp_x((\delta+2\epsilon)v)$ to the sprinkled causal set.
          Alternatively, the endpoints may instead be selected
          among the sprinkled points either by prescribing their counting
          time-separation $\uptau_C(x,x')=n+r$, or by first finding a long chain
          and taking $x$ and $x'$ its extremal elements. When working with the manifold time-separation, one
          may choose a pair for which
          $|\uptau_M(x,x')-(\delta+2\epsilon)|$ is minimal.

    \item \textbf{Find a maximising chain.}
          After the endpoints $x,x'$ have been chosen, we select a maximising chain $\gamma$
          between them. This step is omitted when
          the transport supports are prescribed directly at the continuum level,
          rather than obtained from a discrete maximising chain.
    \item \textbf{Split the maximising chain.}
          We choose the indices $n,r$ so that $\gamma=(x_0,\ldots,x_{n+r})$ and
          $n/(n+r)$ approximates $\delta/(\delta+2\epsilon)$.
          Thus $n$ represents the separation between the lower tips of the two
          diamonds, while $r$ represents their time diameter $2\epsilon$.
          In purely counting-time experiments one may instead prescribe the
          integers $n,r$ directly.
    \item \textbf{Build the transport supports and cost matrix.}
          The split chain determines the two
          causal diamonds $A\coloneqq J_C(x_0,x_r),
              B\coloneqq J_C(x_n,x_{n+r})$
          which support the two probability measures in
          \cref{def:OllivierRicciCST}. The transport cost matrix is
          \[
              C_{ab}=\uptau_{\mathrm{cost}}(a,b),
              \qquad a\in A,\ b\in B,
          \]
          where $\uptau_{\mathrm{cost}}$ is either the counting time-separation
          $\uptau_C$ or the
          restricted manifold time-separation $\uptau_M$, depending on the numerical experiment. Alternatively, one may prescribe the two continuum diamonds first and then
          set $A$ and $B$ to be their intersections with the sprinkled causal
          set $C$.

    \item \textbf{Solve the Lorentzian transport problem.}
          The optimal transport problem is solved with the Python Optimal Transport
          library. Since the library is written for minimization whereas
          $\ell_1$ is a supremum, see \cref{eq:discreteell1}, the implementation
          minimizes the shifted or negated cost matrix and then converts back to the
          corresponding maximal transport value. Depending on the experiment, we use
          either the exact linear-programming solver or a Sinkhorn-regularized
          approximation.

    \item \textbf{Form the curvature estimator.} Finally, we compute the renormalised quantity
          \[
              \widehat{\Ric}_{\rho,\epsilon}(v,v)
              \coloneqq
              \frac{2(d+1)(d+2)}{d\epsilon^2}
              \left(
              \frac{\ell_1(\mu_x,\mu_y)}{\uptau(x_0,x_n)}-1
              \right).
          \]
          Here $\uptau$ denotes the time-separation chosen for the transport
          cost, and the normalising factor is interpreted according to the chosen
          scale convention. If $\epsilon$ is a continuum proper-time scale, the
          prefactor is used as written. If the computation is carried out in
          counting-chain units, then $\epsilon$ is replaced by the continuum scale
          corresponding to the chain height $r$; in dimension two this means
          $\epsilon^2\simeq r^2/(8\rho)$. We then study the behaviour of this
          estimator as the sprinkling density $\rho$, the scale $\epsilon$,
          or both are varied.
\end{enumerate}\bigskip

\textbf{Curvature estimator on individual sprinkled causal sets.} We ran numerical experiments for the renormalised Ricci curvature of causal sets generated by sprinklings into two-dimensional Minkowski, de Sitter, and anti-de Sitter space, using the counting time-separation. In \Cref{fig:numerical1}, we show one representative trial for each spacetime, modelled in conformally flat coordinates. This reduces the causal order to a product order, allowing us to use a longest-increasing-subsequence algorithm and substantially improve performance. In these trials, the renormalised curvature is $-0.539$ in Minkowski space, $-1.385$ in de Sitter space, and $0.765$ in anti-de Sitter space.\medskip

\begin{figure}[ht]
    \centering
    \includegraphics[width=0.95\textwidth]{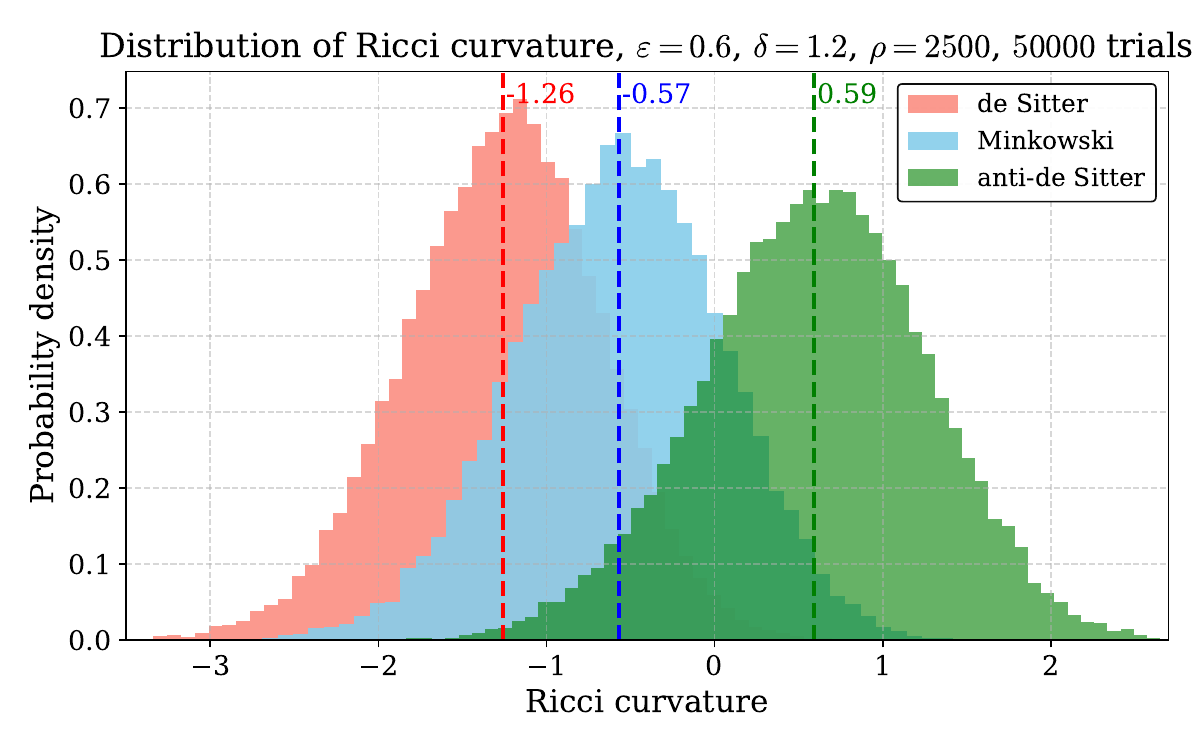}
    \caption{Empirical distributions of the renormalised curvature from 50{,}000 trials in two-dimensional de Sitter, Minkowski, and anti-de Sitter spacetimes at fixed $(\rho,\delta,\epsilon)$.}
    \label{fig:riccis_histogram}
\end{figure}

\textbf{Can the estimator distinguish different geometries?} In \Cref{fig:riccis_histogram} we run 50{,}000 independent trials in each spacetime and plot the resulting probability densities of the renormalised curvature. The three distributions are clearly separated, showing that the estimator already distinguishes the three background geometries at the tested density. The empirical means do not yet coincide with the smooth target values $-1$ (de Sitter), $0$ (Minkowski), and $+1$ (anti-de Sitter), which is expected at finite sprinkling density and fixed mesoscopic scales. Increasing $\rho$ while sending $\epsilon,\delta\to0$ should drive the estimator towards the corresponding smooth Ricci curvatures, but this quickly becomes computationally expensive. We also observe that the estimates appear biased from below, consistent with finite-density effects in the Myrheim--Meyer longest-chain approximation to proper time: after dimension-dependent rescaling, the counting time-separation converges to the continuum Lorentzian time-separation as the sprinkling density increases \cite{sumati2025}.\medskip

\textbf{How does the observation scale affect the curvature signal?} To quantify the dependence on the observation scale, we repeat the experiment for varying values of $\epsilon$ while keeping the sprinkling density $\rho$ fixed. The resulting empirical means and standard deviations of the renormalised curvature are shown in \Cref{fig:PltCurvatureEpsilonDep}.

\begin{figure}[ht]
    \centering
    \includegraphics[width=0.95\textwidth]{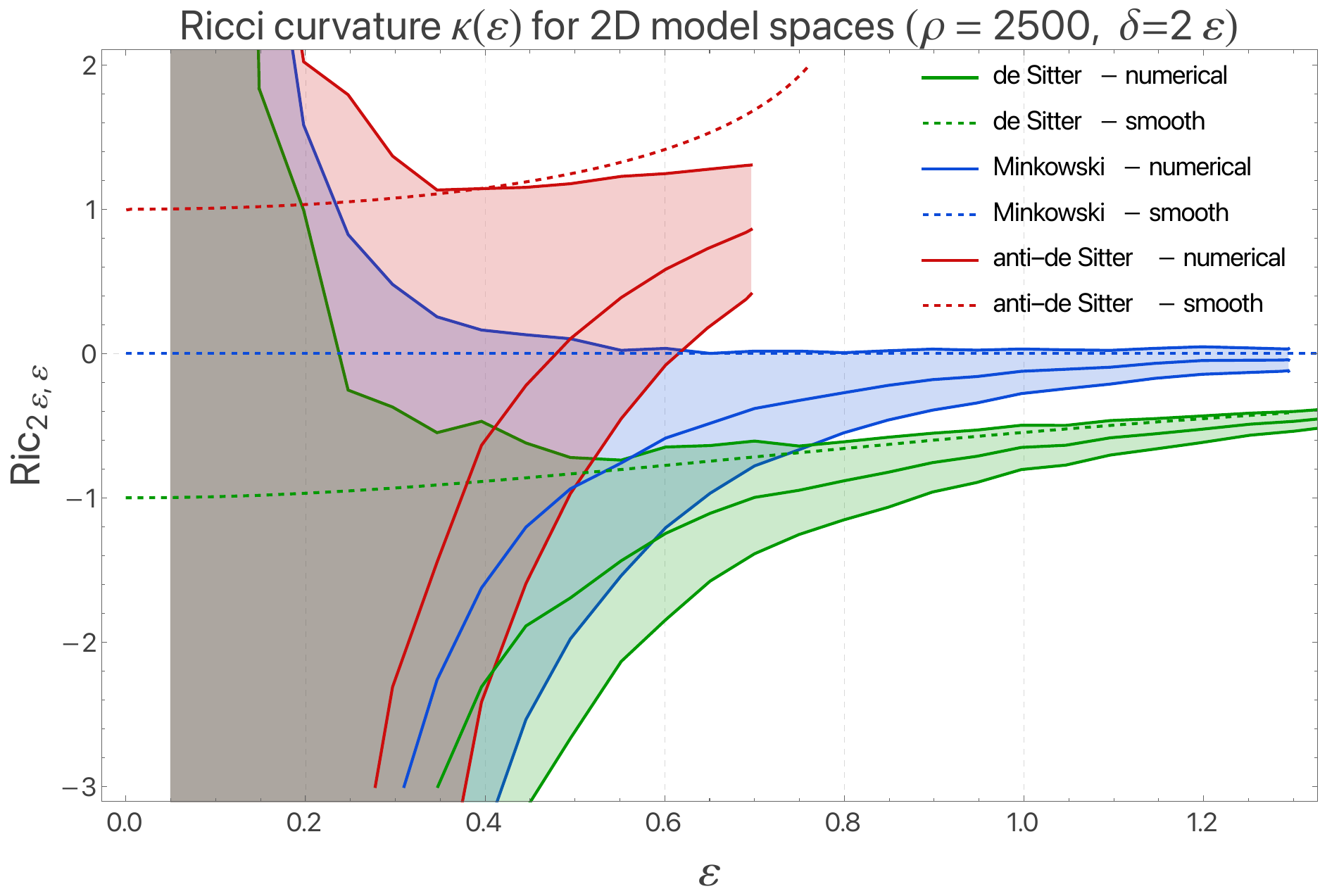}
    \caption{Mean and one standard deviation bands of the renormalised curvature as a function of $\epsilon$ at fixed $\rho$, for two dimensional Minkowski, de Sitter, and anti-de Sitter spacetimes, together with the corresponding smooth Ollivier--Ricci curvatures computed from the exact transport between uniform diamond measures.}
    \label{fig:PltCurvatureEpsilonDep}
\end{figure}

The dotted curves in \Cref{fig:PltCurvatureEpsilonDep} represent the smooth Ollivier--Ricci curvatures $\kappa(\mu_x,\mu_y)_{\epsilon,\delta}$, obtained from the corresponding continuum diamonds without discretisation. Exploiting the high symmetry of Minkowski, de Sitter, and anti-de Sitter space, we can write the optimal transport explicitly and hence evaluate this smooth curvature. We record the following observations:
\begin{itemize}
    \item As $\epsilon\to0$, the smooth Ollivier--Ricci curvatures converge to the expected values $-1$, $0$, and $+1$, in accordance with \Cref{thm:smoothTheorem}.
    \item For the causal-set data, increasing $\epsilon$ improves the separation between the three backgrounds and, in Minkowski and de Sitter space, the empirical means appear to approach the smooth values. Two effects contribute to this behaviour: larger $\epsilon$ yields more points per diamond, improving the empirical approximation of the smooth uniform measures; and it increases the typical time-separations between points in $A$ and $B$, improving the approximation of the manifold time-separation by the counting time-separation.
    \item The convergence is less clear in anti-de Sitter space, consistent with the singularity of the smooth Ollivier--Ricci curvature as $\epsilon\to\pi/4$, where the causal diamonds become unbounded.
    \item For large $\epsilon$, the smooth curves lie close to the empirical means shifted upward by one standard deviation. We cannot currently explain this bias.
    \item The plot illustrates the ``mesoscopic'' nature of $\kappa$: taking $\epsilon$ and $\delta$ too small amplifies discretisation error, while taking them too large destroys locality and leaves the small-scale regime in which the smooth expansion is valid.
\end{itemize}\medskip

\textbf{What part of the error comes from discreteness?} We isolate the discretisation error in flat space for small $\epsilon$ and $\delta$. In Minkowski space, the smooth Ollivier--Ricci curvature vanishes identically, independently of $\epsilon$ and $\delta$, and the background is scale invariant. In two dimensions, the renormalised estimator becomes
\[
    \widehat{\Ric}^2_{\rho,\epsilon}(v,v)
    \coloneqq
    \frac{12}{\epsilon^2}
    \underbrace{\left(
        \frac{\ell_1(\mu_x,\mu_y)}{\uptau(x_0,x_n)}-1
        \right)}_{\Delta_N}\, .
\]
Thus $\Delta_N$ is independent of the continuum scale and isolates the discretisation effect: it depends only on $N$, the number of sprinkled points per diamond. \Cref{fig:PltCurvatureNDep} shows the empirical mean of $-\Delta_N$ and the standard deviation of $\Delta_N$ as functions of $N$ on a log--log scale. For comparison, we also show the standard deviation obtained when the cost uses the manifold time-separation instead of the counting time-separation.

\begin{figure}[ht]
    \centering
    \includegraphics[width=0.9\textwidth]{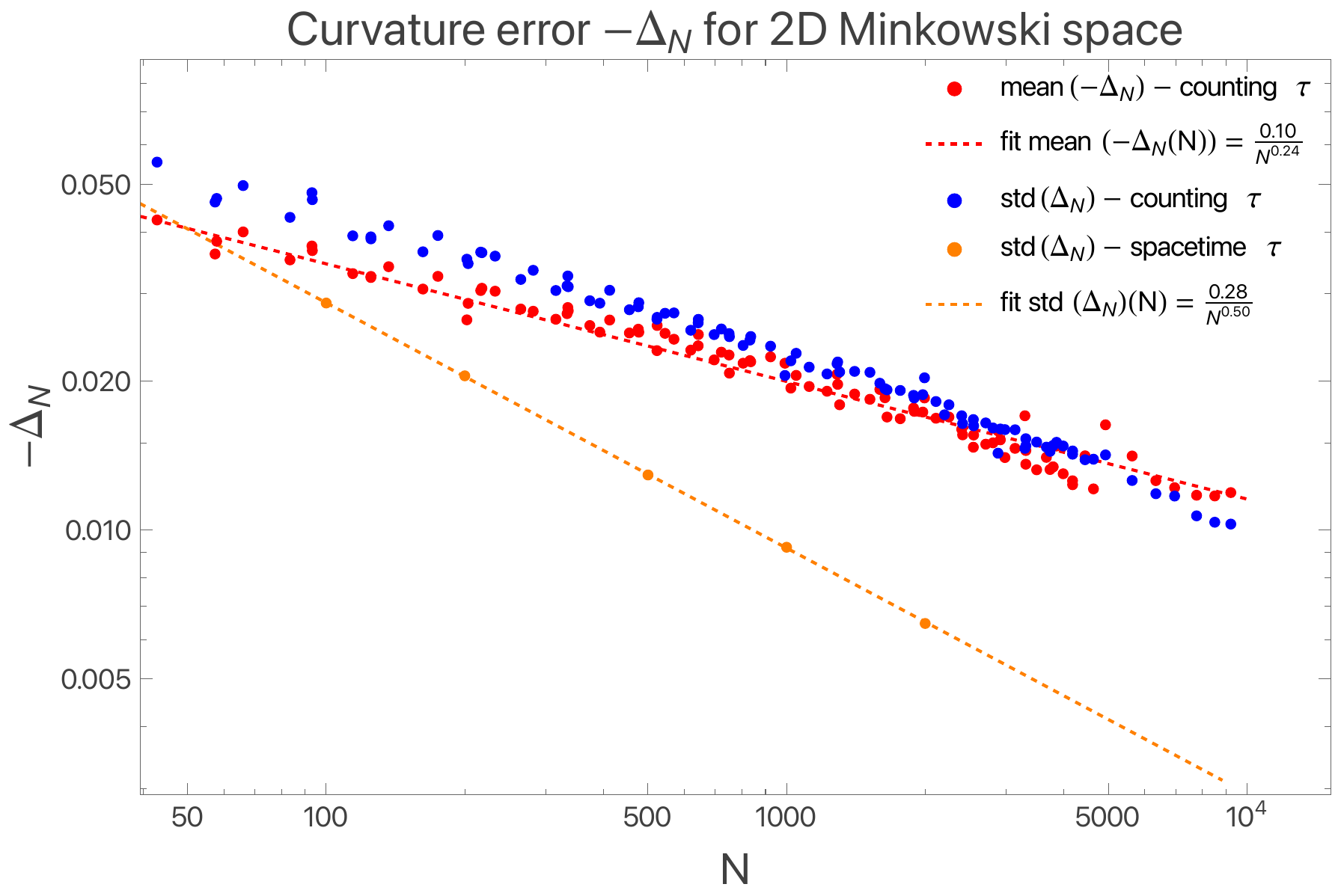}
    \caption{Discretisation error in 2D Minkowski space. Log--log plot of the empirical mean of $-\Delta_N$ and the standard deviation of $\Delta_N$ against the number $N$ of sprinkled points per diamond. Solid markers correspond to the counting time-separation. Hollow markers correspond to the manifold time-separation: only the standard deviation is shown because the empirical mean is statistically indistinguishable from zero.}
    \label{fig:PltCurvatureNDep}
\end{figure}

From \Cref{fig:PltCurvatureNDep} we observe:
\begin{itemize}
    \item With the counting time-separation, the empirical mean of $-\Delta_N$ remains close to its standard deviation throughout the tested range of $N$.
    \item With the manifold time-separation, the empirical mean of $\Delta_N$ is essentially zero within standard error, so only the standard deviation is shown. This suggests that the counting time-separation introduces a negative bias relative to the manifold value, whereas for the manifold time-separation the remaining error comes from approximating the uniform diamond measures by empirical measures.
    \item Both the counting-time mean and the standard deviations exhibit power-law decay. For the manifold time-separation, we observe $\mathrm{std}(\Delta_N)\sim N^{-1/2}$, consistent with sampling fluctuations. For the counting time-separation, the decay is slower, close to $\mathrm{mean}(-\Delta_N)\sim N^{-1/4}$.
\end{itemize}\medskip

\textbf{How quickly must the sprinkling density grow in the smooth limit?} To recover the smooth Ricci curvature from the Ollivier--Ricci curvature of a Poisson sprinkling, one must send the mesoscopic scales $\epsilon,\delta\to0$ while increasing the sprinkling density $\rho$. \Cref{fig:PltCurvatureNDep} indicates how fast $\rho$ must grow for the discretisation error to decay. In two dimensions, writing $N$ for the number of sprinkled points per diamond, the plot suggests
\[
    \Delta_N \sim N^{-\lambda}
    \qquad (\text{empirically } \lambda \approx 0.24).
\]
Using $N \asymp \rho\,\epsilon^2$ and
$\widehat{\Ric}_{\rho,\epsilon}(v,v)=\frac{12}{\epsilon^2}\Delta_N$, this gives
\[
    \widehat{\Ric}_{\rho,\epsilon}(v,v)
    \;\sim\;
    \frac{1}{\epsilon^2\,N^\lambda}
    \;\sim\;
    \frac{1}{\epsilon^2\,(\rho\,\epsilon^2)^\lambda}
    \;=\;
    \frac{1}{\epsilon^{2+2\lambda}\,\rho^{\lambda}}\,.
\]
Thus, for this error term to decay as $\epsilon\to0$, one needs
\[
    \rho(\epsilon) \gg \epsilon^{-(2+2\lambda)/\lambda} \,.
\]
With the empirical exponent $\lambda\approx0.24$, this requires
$\rho(\epsilon)\gtrsim\epsilon^{-10.3}$; for the idealised value $\lambda=\tfrac14$, one obtains $\rho(\epsilon)\gtrsim\epsilon^{-10}$. In practice, this growth rapidly exceeds our computational limits, so we did not include experiments in which $\epsilon$ is decreased while $\rho$ is increased simultaneously.

\begin{remark}
    \label{rem:asymptotic-bias}
    The exponent observed over the tested range in \Cref{fig:PltCurvatureNDep} may not describe the asymptotic decay. For endpoints $p,q$ of a causal diamond in two-dimensional Minkowski space, the relevant asymptotic for the expected counting time-separation is
    \[
        \mathbb E\!\left[
            \uptau_C(p,q)
            \,\middle|\,
            \left|J_C(p,q)\setminus\{p,q\}\right|=N
            \right]
        =2N^{1/2}+\mu_{\mathrm{TW}}N^{1/6}+o(N^{1/6}),
    \]
    where $\mu_{\mathrm{TW}}<0$ is the mean of the Tracy--Widom distribution.
    Relative to the leading term, the correction is of order $N^{1/6}/N^{1/2}=N^{-1/3}$. Since $\mu_{\mathrm{TW}}<0$, the expected counting time-separation approaches its leading-order value from below, with a larger relative deficit at scale $N$ than at scale $4N$. In Meyer's approximation this makes $\Delta_N$ negative and gives $\mathrm{mean}(-\Delta_N)\asymp N^{-1/3}$, rather than $N^{-1/4}$. Including the next correction gives $\mathrm{mean}(-\Delta_N)\approx A N^{-1/3}-B N^{-1/2}$ for constants $A,B>0$. The second term therefore partly cancels the first and can make the decay appear close to $N^{-0.24}$ over the values of $N$ considered in \Cref{fig:PltCurvatureNDep}, providing a possible explanation for the numerical result.
\end{remark}

\section{Ollivier-Ricci curvature in smooth Lorentzian geometry}

In this section, we prove \cref{eq:MainTheorem}. Our approach is inspired by Ollivier's original proof \cite{Ollivier2009}; see also \cite{Arnaudon2025,Liu2012Thesis,Sturm2005}. The key Lorentzian modification is to replace metric balls by causal diamonds.

\label{section:smooththeorem}

\subsection{Diffeomorphism between tangent and manifold diamonds}

On a sufficiently small neighbourhood of a point in a Riemannian manifold, the exponential map is a diffeomorphism from a ball in the tangent space onto a geodesic ball in the manifold.

This picture breaks down for causal diamonds in the Lorentzian setting. Even when the exponential map is a local diffeomorphism, it need not send the tangent-space (Minkowski) diamond onto a manifold diamond. Accordingly, rather than using the exponential map directly, we precompose it with a suitable radial rescaling so that the resulting map is indeed a diffeomorphism (up to a null set) from $\tilde D_\epsilon(x, v)$ onto $D_\epsilon(x, v)$, where we continue to adopt the notation of \cref{Introduction}.

\begin{definition}
    Let $\epsilon > 0$ sufficiently small, $x \in M$ and $u, v \in \T_x(M)$, the \emph{exit time} of the curve $s \mapsto \exp_x(su)$ from $D_\epsilon(x, v)$ is defined by
    \[
        t^\epsilon_{\mathrm{exit}}(x, v, u) \coloneqq \sup\{ t \mid \exp_{x} (s u) \in D_\epsilon(x, v) \text{ for all } s \in \ointerval{0}{t} \}.
    \]
    Likewise, the \emph{exit time} from $\tilde D_\varepsilon(x,v)$ is
    \[
        \tilde t^\epsilon_{\mathrm{exit}}(x, v, u) \coloneqq \sup\{ t \mid \exp_{x} (s u) \in \exp( \tilde D_\epsilon(x, v)) \text{ for all } s \in \ointerval{0}{t}  \}.
    \]
\end{definition}

\begin{remark}
    When no confusion can arise, we suppress ``$(x, v)$'' or even ``$u$'' from the notation. By construction, it holds that
    \[
        t^0_{\mathrm{exit}}(u) = 0 \qquad \text{(resp.\ $\tilde t^0_{\mathrm{exit}}(u) = 0$)}
    \]
    since $D_0(x, v) = \exp_x(\tilde D_0(x, v)) = \{x\}$.
\end{remark}

The map we will consider is
\begin{equation}
    \label{eq:fxmap}
    f_{(x,v)}^{\varepsilon}:\exp_x(\tilde D_{\varepsilon}(x,v))\to D_{\varepsilon}(x,v) :
    \exp_x(u) \mapsto \exp_x\left(\frac{t_{\mathrm{exit}}^{\varepsilon}(x,v,u)}{\tilde t_{\mathrm{exit}}^{\varepsilon}(x,v,u)}u\right),
\end{equation}
which will allow us to construct the commutative diagram between diamonds illustrated in \cref{fig:diamond-diagram}: there, the map
\begin{equation}
    \label{eq:mapT}
    T \coloneqq f_y \circ \exp_{y} \circ \mathord{\parallel_\delta^0}(c) \circ \exp_x^{-1} \circ f_x^{-1} : D_\epsilon(x, v) \to D_\epsilon(y, v_\delta)
\end{equation}
makes the diagram commute.

\begin{figure}
    \centering
    \[
        \begin{tikzcd}
            D_{\epsilon}(x,v) \arrow[r, "T"] & D_{\epsilon}(y,v_\delta) \\
            \exp_{x}(\tilde D_{\epsilon}(x,v)) \arrow[u, "f_x"] & \exp_{x}(\tilde D_{\epsilon}(y,v_\delta))\arrow[u,"f_y"] \\
            \tilde D_{\epsilon}(x,v) \arrow[r, "\mathord{\parallel_0^\delta}(c)"] \arrow[u, "\exp_x"] &
            \tilde D_{\epsilon}(y,v_\delta) \arrow[u, "\exp_y"]
        \end{tikzcd}
    \]
    \caption{A commutative diagram between diamonds.}
    \label{fig:diamond-diagram}
\end{figure}
The map \cref{eq:fxmap} is well defined for $\epsilon > 0$ sufficiently small because $\exp_x$ is injective on $\tilde D_\varepsilon(x,v)$, so each
$\tilde x\in \exp_x(\tilde D_\varepsilon(x,v))$ can be written uniquely as $\tilde x=\exp_x(u)$ with
$u\in \tilde D_\varepsilon(x,v)$, and for such $u$ we have $\tilde t^\varepsilon_{\rm exit}(x,v,u)>0$.
Moreover, $D_\varepsilon(x,v)$ and $\tilde D_\varepsilon(x,v)$ are star-shaped along future timelike rays from $x$,
so scaling $u$ by $t^\varepsilon_{\rm exit}(x,v,u)/\tilde t^\varepsilon_{\rm exit}(x,v,u)$ lands in $D_\varepsilon(x,v)$.
Reversing this scaling along the ray through any $y=\exp_x(w)\in D_\varepsilon(x,v)$ shows that $f^\epsilon$ is bijective.
By proving that the exit times depend smoothly on their arguments whenever $u$ and $v$ are not collinear, we will have shown that
\[
    f_{(x,v)}^{\varepsilon}:\exp_x(\tilde D_{\varepsilon}(x,v))\setminus\Sigma
    \to
    D_{\varepsilon}(x,v)\setminus f_{(x,v)}^{\varepsilon}(\Sigma)
\]
is a diffeomorphism, where both $\Sigma$ and $f_{(x,v)}^{\varepsilon}(\Sigma)$ are null sets, and
\[
    \Sigma\coloneqq\{\exp_x(u)\in \exp_x(\tilde D_\varepsilon(x,v)) \mid u \text{ is collinear with } v\}.
\]

\begin{proposition}
    Let $x\in M$ and let $u,v\in \T_x(M)$ be non-collinear future-directed timelike vectors. For $\epsilon>0$
    sufficiently small, we have
    \[
        \tilde t^\epsilon_{\mathrm{exit}}(x,v,u)
        =
        2\epsilon\frac{\langle u,v\rangle-\sqrt{\langle u,v\rangle^2-\langle u,u\rangle\langle v,v\rangle}}{\langle u,u\rangle}.
    \]
\end{proposition}

\begin{proof}
    Since $\epsilon>0$ sufficiently small, $\exp_x$ is a diffeomorphism on a neighbourhood containing $\tilde D_\epsilon(x,v)$, and thus $\exp_x(su)\in \exp_x(\tilde D_\epsilon(x,v))$ if and only if $su\in \tilde D_\epsilon(x,v)$. By definition of $\tilde D_\epsilon(x,v)$, the exit occurs when the ray $s\mapsto su$ meets the null boundary of
    $J^-(2\epsilon v)$, i.e. at the smallest positive $s$ such that
    \[
        \langle su-2\epsilon v,su-2\epsilon v\rangle=0.
    \]
    Expanding gives the quadratic equation
    \[
        \langle u,u\rangle s^2-4\epsilon\langle u,v\rangle s+4\epsilon^2\langle v,v\rangle=0,
    \]
    whose solutions are
    \[
        s
        =
        2\epsilon\frac{\langle u,v\rangle\pm\sqrt{\langle u,v\rangle^2-\langle u,u\rangle\langle v,v\rangle}}{\langle u,u\rangle}.
    \]
    Taking the smallest positive root (for future-directed timelike $u,v$ this corresponds to the minus sign) yields the
    claimed formula for $\tilde t^\epsilon_{\mathrm{exit}}(x,v,u)$.
\end{proof}

The following result relies on an expansion of the geodesic deviation, which we recall and adapt to the Lorentzian setting in \cref{Theorem:Energy_expansion}.

\begin{proposition}
    \label{prop:expansiontexit}
    Let $x\in M$ and let $u,v\in \T_x(M)$ be non-collinear future-directed timelike vectors. For $\epsilon>0$
    sufficiently small, we have
    \[
        t^\epsilon_{\mathrm{exit}}(x,v,u)=\epsilon s(\epsilon,x,u,v),
    \]
    for a smooth function $s$ defined in a neighbourhood of $(0,x,u,v)$. Moreover, letting $s_0\coloneqq s(0,x,u,v)$,
    the smallest positive root of
    \begin{equation}
        \label{eq:defs0}
        \langle su-2v,su-2v\rangle=0,
        \qquad\text{i.e.}\qquad
        s_0
        =
        2\frac{\langle u,v\rangle-\sqrt{\langle u,v\rangle^2-\langle u,u\rangle\langle v,v\rangle}}{\langle u,u\rangle},
    \end{equation}
    we have, as $\epsilon\to0$,
    \begin{equation}
        \label{eq:expansions(epsilon)}
        s(\epsilon, x, u, v)
        =
        s_0
        +
        \frac{1}{3}
        \frac{\langle R(u,v)u,v\rangle s_0^{2}}{\sqrt{\langle u,v\rangle^2-\langle u,u\rangle\langle v,v\rangle}}
        \epsilon^2
        +O(\epsilon^3),
    \end{equation}
    with the remainder locally uniform in $(x,u,v)$ away from $\Sigma$.
\end{proposition}

\begin{proof}
    The exit point $\exp_x(t^\epsilon_{\mathrm{exit}}(u)u)$ lies on the boundary of $D_\epsilon(x,v)$; equivalently, it is null-related to the top vertex $\gamma_v(2\epsilon)$. Being on $\partial J^{-}(\gamma_v(2\epsilon))$ means that the unique geodesic from $\exp_x(tu)$ to $\gamma_v(2\epsilon)$ is null, and hence its energy vanishes. In other words, the exit time $t^\epsilon_{\mathrm{exit}}(u)$ is characterised as the smallest positive solution $t$ of
    \[
        \Phi(t, \epsilon, x, v, u) \coloneqq \E(\exp_x(tu),\gamma_v(2\epsilon))=0.
    \]
    The map $\Phi$ is defined for all $(t,\epsilon,x,v,u)$ such that $\exp_x(tu)$ and $\gamma_v(2\epsilon)=\exp_x(2\epsilon v)$ lie in a common convex normal neighbourhood of $x$ (so that the connecting geodesic is unique and the energy $\E$ is well defined), which in particular holds when $|t|$ and $|\epsilon|$ are sufficiently small and $u$ and $v$ range over a bounded set.

    Applying \cref{Theorem:Energy_expansion} with $\gamma_u(t)=\exp_x(tu)$ and $\gamma_v(2\epsilon)=\exp_x(2\epsilon v)$ yields that the function defined by
    \[
        H(t,\epsilon,x,u,v) \coloneqq \frac{1}{\epsilon^2}\Phi(\epsilon t,\epsilon,x,v,u)
    \]
    satisfies, as $\epsilon\to0$ and locally uniformly for $t$ and $(x,u,v)$ in compact sets,
    \begin{equation}
        \label{eq:expansionH}
        H(t,\epsilon,x,u,v)=
        \frac12 \langle tu-2v, tu-2v \rangle
        +\frac{2}{3} \langle R(u,v)u, v \rangle t^2\epsilon^2
        +t^2O(|\epsilon|^3).
    \end{equation}
    Strictly speaking, $H$ is defined only for $\epsilon\neq 0$, but it extends smoothly to $\epsilon \in \mathbb{R}$ by setting
    \[
        H(t,0,x,u,v)=\frac12 \langle tu-2v, tu-2v \rangle.
    \]
    By construction, for each $\epsilon>0$ sufficiently small the exit time satisfies
    \[
        t^\epsilon_{\mathrm{exit}}(u)=\epsilon s(\epsilon,u),
    \]
    where $s(\epsilon,u)$ denotes the smallest positive root of the equation
    \begin{equation}
        \label{eq:H=0}
        H(s,\epsilon,x,u,v)=0.
    \end{equation}
    Note that $s_0 \coloneqq s(0, x, u, v)$ is the smallest positive root of
    $\langle su-2v, su-2v\rangle=0$, and we recover \cref{eq:defs0}.
    Using \cref{eq:expansionH}, we find that
    \begin{equation}
        \label{eq:partialHnot0}
        \partial_s H(s_0,0,x,u,v)
        = \langle s_0 u - 2v, u\rangle =
        -2\sqrt{\langle u,v\rangle^2-\langle u,u\rangle\langle v,v\rangle}\neq 0,
    \end{equation}
    where the final inequality holds since $u$ and $v$ are not collinear. The implicit function theorem implies that $(\epsilon, x, u, v) \mapsto s(\epsilon, x, u, v)$ is smooth in a neighbourhood of $(0, x, u, v)$; hence, for fixed $(x, u, v)$, we have the Taylor expansion
    \begin{equation}
        \label{eq:TaylorExpansions(epsilon)}
        s(\epsilon)=s_0+\left(\frac{\diff s}{\diff\epsilon}(0)\right)\epsilon+\frac12\left(\frac{\diff^2s}{\diff\epsilon^2}(0)\right)\epsilon^2+O(\epsilon^3),
        \qquad \text{ as }\epsilon\to0.
    \end{equation}
    To obtain the first two derivatives of $\epsilon \mapsto s(\epsilon)$, we differentiate \cref{eq:H=0}, which gives
    \[
        \partial_s H(s(\epsilon),\epsilon)\frac{\diff s}{\diff \epsilon}(\epsilon)+\partial_\epsilon H(s(\epsilon),\epsilon) = 0.
    \]
    Evaluating at $\epsilon=0$, and using the facts that $\partial_\epsilon H(s, 0) = 0$ from \cref{eq:expansionH} and $\partial_s H(s_0,0) \neq 0$ from \cref{eq:partialHnot0}, we obtain $\frac{ds}{d\epsilon}(0)=0$.

    Differentiating once more and evaluating at $\epsilon = 0$, we obtain
    \[
        \partial_s H(s_0,0)\frac{\diff ^2s}{\diff\epsilon^2}(0)
        +\partial_{ss}H(s_0,0)\left(\frac{\diff s}{\diff \epsilon}(0)\right)^2
        +2\partial_{s\epsilon}H(s_0,0)\frac{\diff s}{\diff \epsilon}(0)
        +\partial_{\epsilon\epsilon}H(s_0,0) = 0.
    \]
    The terms containing $\frac{\diff s}{\diff \epsilon}(0)$ vanish, and from \eqref{eq:expansionH} again, we obtain
    \[
        \partial_{\epsilon\epsilon}H(s_0,0)=\frac{4}{3}\langle R(u,v)u,v\rangle s_0^2.
    \]
    Therefore, we have that
    \[
        \frac{\diff^2s}{\diff\epsilon^2}(0)
        =-\frac{\frac{4}{3}\langle R(u,v)u,v\rangle s_0^{2}}{\partial_s H(s_0,0)}
        =\frac{2}{3}\frac{\langle R(u,v)u,v\rangle s_0^{2}}{\sqrt{\langle u,v\rangle^2-\langle u,u\rangle\langle v,v\rangle}}.
    \]
    Substituting into the Taylor expansion \cref{eq:TaylorExpansions(epsilon)} yields \cref{eq:expansions(epsilon)}.
\end{proof}

The ratio $t^\epsilon_{\mathrm{exit}}(x,v,u)/\tilde t^\epsilon_{\mathrm{exit}}(x,v,u)$ is thus smooth as long as $u$ is not collinear to $v$. Consequently, the map $f^\epsilon$ defined in \cref{eq:fxmap} is indeed smooth. By Hadamard's lemma, there exists a smooth function $a(\epsilon,x,v,u)$ such that
\[
    \frac{t^\epsilon_{\mathrm{exit}}(x,v,u)}
    {\tilde t^\epsilon_{\mathrm{exit}}(x,v,u)}=1+\epsilon^2 a(\epsilon,x,v,u).
\]
It follows that
\begin{equation}
    \label{eq:ratioexpansion}
    \frac{t^\epsilon_{\mathrm{exit}}(x,v,u)}
    {\tilde t^\epsilon_{\mathrm{exit}}(x,v,u)}
    =
    1 + O(\epsilon^2),
\end{equation}
and
\begin{equation*}
    \label{eq:derratioexpansion}
    \frac{d}{d\lambda}
    \left(
    \frac{t^\epsilon_{\mathrm{exit}}(x_\lambda,v_\lambda,u_\lambda)}
    {\tilde t^\epsilon_{\mathrm{exit}}(x_\lambda,v_\lambda,u_\lambda)}
    \right)
    =
    O(\epsilon^2),
\end{equation*}
for any smooth family of parameters $(x_\lambda,v_\lambda,u_\lambda)$ such that
$u_\lambda$ is not collinear to $v_\lambda$ for all $\lambda$. These estimates will be used repeatedly in the asymptotic analysis below.

\begin{proposition}\label{prop:diamond_volume_asymptotic}
    Let $x\in M$ and let $v\in \T_x(M)$ be a unit future-directed timelike vector.
    Then, as $\epsilon\to0$,
    \[
        \vol_g(D_\epsilon(x,v))
        =
        \Leb(\tilde D_\epsilon(x,v))
        \bigl(1+O(\epsilon^2)\bigr).
    \]
\end{proposition}

\begin{proof}
    Identifying $\T_x(M) \cong \mathbb R^n$ via a Lorentz-orthonormal basis
    \[
        e_0=v,\ e_1,\dots,e_{n-1}
    \]
    of $\T_x(M)$, the change-of-variables formula gives
    \begin{equation}
        \label{eq:changeofvariable}
        \vol_g(D_\epsilon(x,v))
        =
        \int_{\tilde D_\epsilon(x,v)}
        \sqrt{|\mathrm{det}\, g^x_{i j}(h^\epsilon_x(w))|}
        \,|\mathrm{det} \diff_w h^\epsilon_x| \diff \Leb(w),
    \end{equation}
    where $g^x_{ij}$ are the metric coefficients in the normal coordinates induced by $\exp_x^{-1}$, and
    \begin{equation}
        \label{eq:maph}
        h^\epsilon_x\coloneqq\exp_x^{-1}\circ f^\epsilon_x\circ \exp_x : \mathbb R^n \cong \T_x(M) \to \T_x(M) \cong \mathbb{R}^n
    \end{equation}
    is such that $D_\epsilon(x,v)=\exp_x\bigl(h^\epsilon_x(\tilde D_\epsilon(x,v))\bigr)$ by definition of $f^\epsilon_x$.

    If $w \cong (w_0, \dots, w_{n - 1}) \in \mathbb R^n$, then, by \cref{eq:ratioexpansion},
    \[
        h_x^\epsilon(w)
        =
        \frac{t^\epsilon_{\mathrm{exit}}(x,v,w)}
        {\tilde t^\epsilon_{\mathrm{exit}}(x,v,w)} w
        =
        (1 + O(\epsilon^2))w.
    \]
    Therefore,
    \[
        |\mathrm{det} \diff_w h^\epsilon_x| = 1 + O(\epsilon^2).
    \]
    The standard expansion of the metric density in normal coordinates yields
    \begin{align*}
        \sqrt{|\det g^x_{ij}(h_x^\epsilon(w))|}
         & =
        1
        -
        \frac{1}{6} \sum_{i,j=0}^{n-1}
        R_{ij}(x)\bigl(h_x^\epsilon(w)\bigr)^i\bigl(h_x^\epsilon(w)\bigr)^j
        +
        O\bigl(|h_x^\epsilon(w)|^3\bigr)                                                              \\
         & = 1 - (1 + O(\epsilon^2))\frac{1}{6} \sum_{i,j=0}^{n-1} R_{i j}(x) w^i w^j + O(\abs{w}^3),
    \end{align*}
    where $|\cdot|$ denotes the norm induced by the auxiliary Euclidean inner product
    for which $(e_0, \dots, e_{n-1})$ is an orthonormal basis. Since $w\in \tilde D_\epsilon(x,v)$, we have $|w|=O(\epsilon)$ uniformly on $\tilde D_\epsilon(x,v)$. It follows that the term $O(|w|^3)$ is in fact $O(\epsilon^3)$, and the proof is completed by substituting the above estimates into \cref{eq:changeofvariable}.
\end{proof}

\subsection{Asymptotics on the time-separation function}

Using the notation of \cref{Introduction}, we consider the orthogonal complement
\[
    v^\perp = \{ w \in \T_x(M) \mid \langle w, v \rangle = 0 \},
\]
which is an $(n - 1)$-dimensional spacelike subspace of $\T_x(M)$. We also choose a sufficiently small convex neighbourhood $\mathcal U$ of $0$ in $\T_x(M)$ on which $\exp_x$ is a diffeomorphism. Then, the region
\[
    \mathcal V \coloneqq \exp_x(\mathcal U \cap v^\perp)
\]
is an embedded $(n-1)$-dimensional spacelike submanifold of $M$, and we let $\mathcal N$ denote a sufficiently small timelike tubular neighbourhood of $\mathcal V$. From now on we take $\mathcal U$ and $\mathcal N$ as small as needed, shrinking them further whenever necessary, without changing notation.
Recall that, after fixing a unit normal field $\nu$ along $\mathcal V$, the second fundamental form at $x$ is defined for tangent vectors $U,V\in \T_x(\mathcal V)$ by
\[
    \mathrm{I\!I}_x(U,V)\coloneqq\langle \nabla_U V,\nu(x)\rangle.
\]
A standard computation in normal (exponential) coordinates centred at $x$ shows that the second fundamental form $\mathrm{I\!I}_x$ of $\mathcal V$ at $x$ vanishes identically (see for example \cite[Remark 2.10.]{Arnaudon2025}).

The \emph{(future-directed) time-separation to $V$} is the map given by
\begin{equation}
    \label{def:functionudef}
    u(x') \coloneqq \max_{y' \in \mathcal V} \uptau(y',x'), \qquad x' \in J^+(\mathcal V) \cap \mathcal N.
\end{equation}
For each $x' \in J^+(\mathcal V) \cap \mathcal N$, the maximiser in the definition of $u(x')$ is attained at a unique point $\pi(x')\in\mathcal V$, and the corresponding maximizing timelike geodesic segment between $\pi(x')$ and $x'$ meets $\mathcal V$ orthogonally at $\pi(x')$. The resulting projection map $\pi:J^+(\mathcal V) \cap \mathcal N \to \mathcal V$ is smooth. As shown in \cite[Lemma 4.1]{CavMondino2024}, the function $u$ is 1-steep on $J^+(\mathcal V) \cap \mathcal N$.

\begin{proposition}
    \label{prop:expansionu}
    As $(\epsilon, \delta) \to (0, 0)$, we have the expansion
    \begin{equation}
        \label{eq:expansionu}
        u(f^\epsilon_{(y, v_\delta)} \circ \exp_y(\epsilon w_\delta)) = \delta - \epsilon \langle v,w\rangle + \frac{\epsilon^2 \delta}{2} \langle R(w,v)v,w \rangle + O(\epsilon^2 \delta^2) + O(\epsilon^3),
    \end{equation}
    where the remainder term is locally uniform in $v$ and $w$.
\end{proposition}

\begin{proof}
    The strategy is to rewrite the square of the time-separation as the negative of the Lorentzian energy of the unique maximizing timelike geodesic segment realising the projection map $\pi$ onto $\mathcal V$, and then to obtain a second-order expansion by analysing an associated geodesic variation and its Jacobi field.

    Let $\gamma : \interval{0}{1} \to M : t \mapsto \exp_x(t \delta v)$, and note that $w_\delta = \mathord{\parallel_0^1}(\gamma)[w]$. Then
    \begin{equation}
        \label{eq:deftheta(s)}
        \theta : \interval{0}{\epsilon} \to M : s \mapsto f_{(y,v_\delta)}^\epsilon \circ \exp_y(s w_\delta)
        = \exp_y\left(
        s\frac{t^\epsilon_{\rm exit}(y,v_\delta,w_\delta)}{\tilde t^\epsilon_{\rm exit}(y,v_\delta,w_\delta)}w_\delta
        \right).
    \end{equation}
    This is the constant-speed geodesic from $y$ in the direction $w'$, adjusted so as to correct the mismatch between the tangent-space diamond image $\exp_y(\tilde D_\epsilon(y,v_\delta))$ and the manifold diamond $D_\epsilon(y,v_\delta)$. For $s \in \interval{0}{\epsilon}$, by construction of $u_+$ and $\pi$, the unique maximizing timelike geodesic from $\pi(\theta(s))$ to $\theta(s)$ is precisely the curve $c_s : \interval{0}{1} \to M$ given by
    \[
        c_s(t)=\exp_{\pi(\theta(s))}\left(t\,\exp_{\pi(\theta(s))}^{-1}(\theta(s))\right).
    \]
    Thus $(s,t)\mapsto c_s(t)$ is a geodesic variation: varying $s$ moves both the endpoint $\theta(s)$ and the footpoint $\pi(\theta(s))$. For each fixed $s$, the curve $t\mapsto c_s(t)$ is an affinely parametrised geodesic, and hence $\langle \dot c_s(t), \dot c_s(t) \rangle$ is constant in $t$. In particular,
    \[
        u_+^2(f^\epsilon_{(y, v_\delta)} \circ \exp_y(s w_\delta)) = u_+(\theta(s))^{2} = -\int_{0}^{1}\langle \dot c_s(t),\dot c_s(t)\rangle \diff t \eqqcolon \E(c_s).
    \]
    We denote by
    \[
        J_s(t)\coloneqq\frac{\partial}{\partial s}c_s(t)
    \]
    the Jacobi field along the curve $t\mapsto c_s(t)$. It satisfies the following endpoint identifications:
    \[
        J_s(1)=\partial_s c_s(1)=\partial_s \theta(s)=\dot\theta(s),\qquad
        J_s(0)=\partial_s c_s(0)=\partial_s \pi(\theta(s))\in \T_{\pi(\theta(s))}(\mathcal V).
    \]
    The geodesics $t \mapsto c_s(t)$ meet $\mathcal V$ orthogonally at $t=0$ for every $s$, i.e.
    \[
        \dot c_s(0)\perp T_{c_s(0)}\mathcal V
        \qquad\text{for all } s.
    \]
    In other words, if $\nu$ is a smooth unit normal field along $\mathcal V$, then
    \[
        \dot c_s(0) = - \langle \dot c_s(0), \nu \rangle \nu \eqqcolon \alpha(s) \nu.
    \]
    In particular, since $\pi(\theta(0))=x$ and $\T_x(\mathcal V)=v^\perp$, we have $J_0(0) \in v^\perp$, and thus
    \begin{equation}
        \label{eq:J00c00}
        \langle J_0(0),\dot c_0(0)\rangle = \delta \langle J_0(0), v \rangle = 0.
    \end{equation}
    Differentiating the identity $\langle J_s(0),\dot c_s(0)\rangle \equiv 0$ with respect to $s$ gives
    \[
        0=\langle \mathrm D_s J_s(0),\dot c_s(0)\rangle+\langle J_s(0),\mathrm D_s\dot c_s(0)\rangle = \langle \mathrm D_s J_s(0),\dot c_s(0)\rangle+\langle J_s(0),\mathrm D_t J_s(0)\rangle.
    \]
    More explicitly, we have
    \[
        \mathrm D_t J_s(0)
        = \mathrm D_s \dot c_s(0)
        = \alpha'(s)\nu(c_s(0))+\alpha(s)\mathrm D_s(\nu(c_s(0))) = \alpha'(s)\nu + \alpha(s)\nabla_{J_s(0)}\nu.
    \]
    Now $\nabla_{J_0(0)}\nu$ is tangent to $\mathcal V$, since $\langle \nu,\nu\rangle$ is constant and hence
    $0=J_0(0)\langle \nu,\nu\rangle=2\langle \nabla_{J_0(0)}\nu,\nu\rangle$. Since the second fundamental form of $\mathcal V$ vanishes at $x$, we also have, for all $W \in \T_x(\mathcal V)$, $\langle \nabla_{J_0(0)} \nu, W\rangle = -\mathrm{I\!I}_x(J_0(0),W) = 0$. Thus we obtain
    \begin{equation}
        \label{eq:usefulShape}
        \langle \mathrm D_s J_0(0),v\rangle = \langle \mathrm D_t J_0(0),J_0(0)\rangle = 0.
    \end{equation}
    Moreover, since $J_s(1)=\dot\theta(s)$, \cref{eq:deftheta(s)} gives
    \begin{equation}
        \label{eq:J01}
        J_0(1)=\dot\theta(0)=\frac{t^\epsilon_{\rm exit}(y,v_\delta,w_\delta)}{\tilde t^\epsilon_{\rm exit}(y,v_\delta,w_\delta)} w_\delta,
    \end{equation}
    so that, by \cref{eq:ratioexpansion},
    \begin{equation}
        \label{eq:J01c01}
        \begin{aligned}
            \langle J_0(1),\dot c_0(1) \rangle & = \langle \mathord{\parallel_1^0}(\gamma)[J_0(1)], \mathord{\parallel_1^0}(\gamma)[\dot c_0(1)] \rangle                                                                                 \\
                                               & = \frac{t^\epsilon_{\rm exit}(y,v_\delta,w_\delta)}{\tilde t^\epsilon_{\rm exit}(y,v_\delta,w_\delta)} \langle w,\delta v \rangle = \delta \langle w, v \rangle + O(\epsilon^2 \delta).
        \end{aligned}
    \end{equation}
    Since $t\mapsto c_s(t)$ is a geodesic, we have $\mathrm D_t \dot c_s=0$, and hence
    \[
        \frac{\diff}{\diff t}\langle J_s,\dot c_s\rangle=\langle \mathrm D_tJ_s,\dot c_s\rangle,
        \qquad
        \frac{\diff^2}{\diff t^2}\langle J_s,\dot c_s\rangle
        =\langle \mathrm D_t^2J_s,\dot c_s\rangle
        =-\langle R(J_s,\dot c_s)\dot c_s,\dot c_s\rangle
        =0.
    \]
    Therefore,
    \[
        \langle J_s(t),\dot c_s(t)\rangle
        =(1-t)\langle J_s(0),\dot c_s(0)\rangle
        +t\langle J_s(1),\dot c_s(1)\rangle.
    \]
    It follows, upon differentiating with respect to $t$, that
    \[
        \langle \mathrm D_t J_s(t),\dot c_s(t)\rangle
        =\frac{\diff}{\diff t}\langle J_s(t),\dot c_s(t)\rangle
        =\langle J_s(1),\dot c_s(1)\rangle-\langle J_s(0),\dot c_s(0)\rangle,
    \]
    so $\langle \mathrm D_t J_s,\dot c_s\rangle$ is constant in $t$. In particular, at $s=0$, using \cref{eq:J00c00,eq:J01c01}, we obtain
    \[
        \langle \mathrm D_t J_0(0),\dot c_0(0)\rangle
        =\langle J_0(1),\dot c_0(1)\rangle-\langle J_0(0),\dot c_0(0)\rangle
        =\delta\langle w,v\rangle+O(\epsilon^2\delta),
    \]
    as well as
    \begin{equation}
        \label{eq:usefulShape2}
        \mathrm D_t J_0(0) = \alpha'(0) v = -\frac{1}{\delta} \langle \mathrm D_t J_0(0),\dot c_0(0)\rangle v = (- \langle w,v\rangle + O(\epsilon^2)) v
    \end{equation}
    We are now in a position to compute the relevant terms in the expansion of $(\delta, \epsilon) \mapsto u_+^2(f^\epsilon_{(y, v_\delta)} \circ \exp_y(s w_\delta))$. First, since $\dot c_0(0)=\delta v$, we have
    \[
        \E(c_0)=-\int_{0}^{1}\langle \dot c_0(t),\dot c_0(t)\rangle \diff t
        =-\langle \delta v,\delta v\rangle
        =\delta^{2}.
    \]
    Secondly, we compute
    \begin{align*}
        \frac{\partial}{\partial s} \E(c_s) |_{s = 0} & = - 2 \int_0^1 \langle \mathrm D_s \dot c_s(t)|_{s = 0}, \dot c_0(t) \rangle  \diff t = - 2 \int_0^1 \langle \mathrm D_t J_0(t), \dot c_0(t) \rangle  \diff t        \\
                                                      & = - 2 \int_0^1  \frac{\diff}{\diff t} \langle J_0(t),\dot c_0(t) \rangle   \diff t = -2\big(\langle J_0(1),\dot c_0(1)\rangle-\langle J_0(0),\dot c_0(0)\rangle\big) \\
                                                      & = - 2 \delta\langle w,v\rangle+O(\epsilon^2\delta)
    \end{align*}
    Thirdly, we find that
    \begingroup
    \interdisplaylinepenalty=0
    \begin{align*}
        \frac{\partial^2}{\partial s^2} \E(c_s) |_{s = 0} & = -2 \frac{\diff}{\diff s} \left( \int _0^1 \langle \mathrm D_t J_s(t), \dot c_s(t) \rangle  \diff t \right)\Big|_{s = 0}                                                                                                   \\
                                                          & = -2 \left(\int_0^1 \langle \mathrm D_s \mathrm D_t J_s(t), \dot c_s(t) \rangle + \langle \mathrm D_t J_s(t), \mathrm D_s  \dot c_s(t) \rangle   \diff t \right) \Big|_{s=0}                                                \\
                                                          & =  -2 \int_0^1 \Big[ \langle \mathrm D_t \mathrm D_s J_s(t), \dot c_s(t) \rangle - \langle R(J_s(t), \dot c_s(t)) \dot c_s(t), J_s(t) \rangle \\
                                                                               & \qquad\qquad \qquad + \langle \mathrm D_t J_s(t), \mathrm D_t J_s(t) \rangle  \Big]\Big|_{s=0} \diff t                             &  & \text{by \cref{eq:commutatoridentity}}                         \\
                                                          & =  -2 \int_0^1 \Big[ \frac{\diff}{\diff t} \langle \mathrm D_s J_s(t), \dot c_s(t) \rangle - \langle R(J_s(t), \dot c_s(t)) \dot c_s(t), J_s(t) \rangle \\
                                                                               & \qquad\qquad \qquad + \frac{\diff}{\diff t}\langle \mathrm D_t J_s(t), J_s(t) \rangle - \langle \mathrm D^2_{tt} J_s(t), J_s(t) \rangle  \Big]\Big|_{s=0} \diff t                                      \\
                                                          & =  -2 \int_0^1 \frac{\diff}{\diff t} \Big[ \langle \mathrm D_s J_s(t), \dot c_s(t) \rangle + \langle \mathrm D_t J_s(t), J_s(t) \rangle \Big]\Big|_{s=0} \diff t             &  & \text{by \cref{eq:JacobiEquationGeneral}} \\
                                                          & = -2\Big[\langle D_s J_s(1)|_{s=0}, \dot c_0(1)\rangle+ \langle D_t J_0(1), J_0(1) \rangle \\
                                                                    & \qquad \qquad -\langle D_s J_s(0)\big|_{s=0}, \dot c_0(0)\rangle-\langle D_t J_0(0), J_0(0)\rangle\Big]                                                                                                           \\
                                                          & = -2 \langle D_t J_0(1), J_0(1) \rangle,
    \end{align*}
    \endgroup
    since $\mathrm D_s J_s(1) |_{s = 0} = \mathrm D_s \partial_s \theta(s) |_{s = 0} = 0$ because $s \mapsto \theta(s)$ is an affinely parametrised geodesic, and by \cref{eq:usefulShape}.

    It therefore remains to compute $\langle D_t J_0(1), J_0(1) \rangle$. Using parallel transport along the curve $t \mapsto c_0(t)$, define $\tilde J_0(t) \coloneqq \mathord{\parallel_t^0} J_0(t) \in \T_x(M)$. Then
    \[
        \dot{\tilde J}_0(t) = \mathord{\parallel_t^0} \mathrm{D}_t J_0(t), \quad \text{ and } \quad \ddot{\tilde{J}}_0(t) + \delta^2 \tilde R(t) \tilde J_0(t) = 0
    \]
    where, for $Z \in \T_x(M)$, we denote
    \[
        \tilde R(t)Z \coloneqq \frac{1}{\delta^2} \mathord{\parallel_t^0} \left[ R(\mathord{\parallel_0^t} Z, \dot c_{0}(t))\dot c_{0}(t)\right] = \mathord{\parallel_t^0} \left[ R(\mathord{\parallel_0^t} Z, v_t)v_t\right].
    \]
    Integrating
    the Jacobi equation gives
    \[
        \dot{\tilde J}_0(1)=\dot{\tilde J}_0(0)- \delta^2 \int_{0}^{1}\tilde R(u)\tilde J_0(u)\diff u.
    \]
    We therefore find
    \begin{align*}
        \langle D_t J_0(1), J_0(1) \rangle & = \langle \dot{\tilde J}_0(1), \tilde J_0(1) \rangle = \langle \dot{\tilde J}_0(0), \tilde J_0(1) \rangle - \delta^2 \int_{0}^{1} \langle \tilde R(u) \tilde J_0(u), \tilde J_0(1) \rangle \diff u
    \end{align*}
    Using \cref{eq:J01,eq:usefulShape,eq:usefulShape2}, we infer that
    \[
        \tilde{J}_0(1) = (1 + O(\epsilon^2)) w, \qquad \dot{\tilde{J}}_0(0) = - (\langle w, v\rangle + O(\epsilon^{2}))v,
    \]
    and
    \[
        \langle \dot{\tilde{J}}_0(0), \tilde{J}_0(0) \rangle = 0, \qquad \langle \dot{\tilde{J}}_0(0), \dot{\tilde{J}}_0(0) \rangle
        = -\langle w, v\rangle^{2} + O(\epsilon^{2}).
    \]
    Therefore,
    \[
        \langle \dot{\tilde J}_0(0),\tilde J_0(1)\rangle
        =-(1 + O(\epsilon^2))\langle w,v\rangle^2 = -\langle w,v\rangle^2 + O(\epsilon^2)
    \]
    and
    \begin{align*}
        \int_0^1 \langle \tilde R(u) \tilde J_0(u),\tilde J_0(1)\rangle \diff u
         & =
        \int_0^1
        \bigl\langle
        R_{c_0(u)}(\mathord{\parallel_0^u}\tilde J_0(u),v_u)v_u,
        \mathord{\parallel_0^u}\tilde J_0(1)
        \bigr\rangle
        \diff u                                  \\
         & =
        \int_0^1
        \bigl\langle R_x(\tilde J_0(u),v)v,\tilde J_0(1)\bigr\rangle
        \diff u
        +O(\delta)
         &   & \text{by \cref{eq:Riemannsmallo}} \\
         & =
        \int_0^1
        \bigl\langle R_x(w,v)v,w\bigr\rangle
        \diff u
        +O(\delta)+O(\epsilon^2)                 \\
         & =
        \langle R(w,v)v,w\rangle
        +O(\delta)+O(\epsilon^2).
    \end{align*}
    Therefore, we have
    \[
        \frac{\partial^2}{\partial s^2} \E(c_s) |_{s = 0} = 2 \langle w,v\rangle^2 + 2 \delta^2 \langle R(w,v)v,w\rangle
        \;+\;O(\delta^3) + O(\epsilon^2 \delta^2) + O(\epsilon^2)
    \]
    and we conclude that
    \[
        \begin{aligned}
            \E(c_\epsilon)
             & = \delta^2
            -2\delta\epsilon\langle w,v\rangle
            +\epsilon^2\langle w,v\rangle^2
            +\delta^2\epsilon^2\langle R(w,v)v,w\rangle                                                \\
             & \qquad + O(\epsilon^3\delta)+O(\epsilon^4)+O(\epsilon^2\delta^3)+O(\epsilon^4\delta^2),
        \end{aligned}
    \]
    and \cref{eq:expansionu} follows by taking the Taylor expansion of the square root.

\end{proof}

\begin{proposition}
    \label{prop:expansionu2}
    As $(\epsilon, \delta) \to (0, 0)$, we have the expansion
    \begin{multline*}
        \label{eq:expansionu2}
        u(f^\epsilon_{(y, v_\delta)} \circ \exp_y(\epsilon w_\delta)) - u(f^\epsilon_{(x, v)} \circ \exp_x(\epsilon w)) \\
        = \delta \left( 1 + \frac{\epsilon^2}{2} \langle R(v, w)w, v \rangle + O(\epsilon^3 + \epsilon^2 \delta) \right).
    \end{multline*}
    where the remainder term is locally uniform in $v$ and $w$.
\end{proposition}

\begin{proof}
    By \cref{prop:expansionu}, the quantity
    \[
        u(f^\epsilon_{(y, v_\delta)} \circ \exp_y(\epsilon w_\delta))
        - u(f^\epsilon_{(x, v)} \circ \exp_x(\epsilon w))
        - \delta\left( 1 + \frac{\epsilon^2}{2}\langle R(v,w)w, v\rangle + O(\epsilon^2\delta) \right)
    \]
    is $O(\epsilon^3)$. On the other hand, it is clearly $O(\delta)$. Hence, being $O(\epsilon^3)$ and $O(\delta)$ simultaneously, it is $O(\epsilon^3\delta)$.
\end{proof}

\begin{proposition} \label{prop:tau_f_lowerbound}
    As $(\epsilon, \delta) \to (0, 0)$, we have the expansion
    \[
        \uptau(f_x \circ \exp_x(\epsilon w), f_{y_\delta} \circ \exp_{y_\delta}(\epsilon w_\delta)) = \delta \left( 1 + \frac{\epsilon^2}{2} \langle R(v, w)w, v \rangle + O(\epsilon^3 + \epsilon^2 \delta) \right).
    \]
\end{proposition}

\begin{proof}
    Note that given $v_s \coloneqq \dot c(s) \in \T_{c(s)}(M)$, we have $\langle v_s, w_s \rangle = \langle v, w \rangle = 0$ for all $s \in [0, \delta]$ because $\mathord{\parallel_0^s} : \T_x(M) \to \T_{c(s)}(M)$ is an isometry. Defining the geodesic variation $F : [0, \delta] \times [0, \epsilon] \to M$ by
    \[
        F(s, t) \coloneqq f^{\epsilon}_{c(s)} \circ \exp_{c(s)}(t w_s) = \exp_{c(s)}\left( t \frac{\tilde t^\epsilon_{\mathrm{exit}}(w_s)}{t^\epsilon_{\mathrm{exit}}(w_s)} w_s \right),
    \]
    we have that the vector field
    \[
        J_s(t) \coloneqq \frac{\partial}{\partial s} F(s, t)
    \]
    is a Jacobi field along $t \mapsto c_s(t) \coloneqq F(s, t)$. In particular, we have the Jacobi equation
    \[
        \frac{\mathrm{D}^2}{\diff t^2} J_s(t) + R(J_s(t), \dot c_s(t))\dot c_s(t) = 0,
    \]
    Note that the curve $\gamma_\epsilon : \interval{0}{\delta} : s \mapsto F(s, \epsilon)$ joins $f_x \circ \exp_x(\epsilon w)$ to $f_{y_\delta} \circ \exp_{y_\delta}(\epsilon w_\delta)$ and that
    \begin{equation}
        \label{eq:toexpand}
        \langle \dot \gamma_\epsilon (s), \dot \gamma_{\epsilon}(s) \rangle = \langle J_s(\epsilon), J_s(\epsilon) \rangle.
    \end{equation}
    We wish to expand \cref{eq:toexpand} as $\epsilon \to 0$. We compute
    \[
        \langle J_s(0), J_s(0) \rangle = \langle \dot c(s), \dot c(s) \rangle = \langle v_s, v_s \rangle = \langle \mathord{\parallel_s^0}v_s, \mathord{\parallel_s^0}v_s \rangle = \langle v, v \rangle = -1,
    \]
    since $v$ is assumed to be timelike and unit, and
    \begin{align*}
        \frac{\mathrm{D}}{\diff t} J_s(0) ={} & \frac{\mathrm{D}}{\diff s} \dot c_s(0) = \frac{\mathrm{D}}{\diff s} \left(\frac{\tilde t^\epsilon_{\mathrm{exit}}(w_s)}{t^\epsilon_{\mathrm{exit}}(w_s)} w_s\right)                                                                    \\
        ={}                                   & \frac{\diff}{\diff s} \left(\frac{\tilde t^\epsilon_{\mathrm{exit}}(w_s)}{t^\epsilon_{\mathrm{exit}}(w_s)} \right) w_s + \frac{\tilde t^\epsilon_{\mathrm{exit}}(w_s)}{t^\epsilon_{\mathrm{exit}}(w_s)} \frac{\mathrm{D} w_s}{\diff s}
        = \frac{\diff}{\diff s} \left(\frac{\tilde t^\epsilon_{\mathrm{exit}}(w_s)}{t^\epsilon_{\mathrm{exit}}(w_s)} \right) w_s
    \end{align*}
    because $s \mapsto w_s$ is parallel by construction. We thus find that
    \begin{align*}
        \frac{\diff}{\diff t}\Big|_{t = 0} \langle J_s(t), J_s(t) \rangle ={} & 2 \left\langle \frac{\mathrm{D}}{\diff t} J_s(0), J_s(0) \right\rangle = 2 \left\langle \frac{\mathrm{D}}{\diff s} \dot c_s(0), \dot c(s) \right\rangle  \\
        ={}                                                                   & 2 \frac{\diff}{\diff s} \left(\frac{\tilde t^\epsilon_{\mathrm{exit}}(w_s)}{t^\epsilon_{\mathrm{exit}}(w_s)} \right) \langle w_s, \dot c(s) \rangle = 0,
    \end{align*}
    where the last equality is justified since $\langle w_s, \dot c(s) \rangle = \langle \mathord{\parallel_s^0} w_s, \mathord{\parallel_s^0} \dot c(s) \rangle = \langle w, v \rangle = 0$. Finally, we have that
    \begin{align*}
        \frac{1}{2} \frac{\diff^2}{\diff t^2}\Big|_{t = 0} \langle J_s(t), J_s(t) \rangle & = \langle J'_s(0), J_s'(0) \rangle + \langle J''_s(0), J_s(0) \rangle = \langle J''_s(0), J_s(0) \rangle + O(\epsilon^4)                                      \\
                                                                                          & = - \langle R(J_s(0), \dot c_s(0))\dot c_s(0), J_s(0) \rangle + O(\epsilon^4)                                                                                 \\
                                                                                          & = - \langle R(v_s, w_s)w_s, v_s \rangle = - \langle R(v, w)w, v \rangle + O(\delta) + O(\epsilon^4),
                                                                                          &                                                                                                                          & \text{by \cref{eq:Riemannsmallo}}.
    \end{align*}
    So far, we have shown that
    \[
        \langle \dot \gamma_\epsilon(s), \dot \gamma_\epsilon(s) \rangle = -1 - \epsilon^2 \langle R(v, w)w, v \rangle + O(\epsilon^3 + \epsilon^2 \delta), \ \text{ as } (\epsilon, \delta) \to (0, 0).
    \]
    Therefore, for $\epsilon$ and $\delta$ small enough, the curve $\gamma_\epsilon$ is timelike and its Lorentzian length gives the estimate
    \begin{align*}
        \uptau(f_x \circ \exp_x(\epsilon w), f_{y_\delta} \circ \exp_{y_\delta}(\epsilon w_\delta)) \geq & \int_0^\delta \sqrt{- \langle \dot \gamma_\epsilon(t), \dot \gamma_\epsilon(t) \rangle} \diff t                              \\
        ={}                                                                                              & \int_0^\delta \sqrt{1 + \epsilon^2 \langle R(v, w)w, v \rangle + O(\epsilon^3 + \epsilon^2 \delta)} \diff t                  \\
        ={}                                                                                              & \int_0^\delta \left(1 + \frac{\epsilon^2}{2} \langle R(v, w)w, v \rangle + O(\epsilon^3 + \epsilon^2 \delta)\right) \diff t, \\
                                                                                                         & = \delta \left(1 + \frac{\epsilon^2}{2} \langle R(v, w)w, v \rangle \right) + O(\epsilon^3 \delta + \epsilon^2 \delta^2).
    \end{align*}
    For the reverse inequality, it suffices to use that the function $u$ defined in \cref{def:functionudef} is 1-steep. Applying this to
    $x=f_x \circ \exp_x(\epsilon w)$ and
    $y=f_{y_\delta} \circ \exp_{y_\delta}(\epsilon w_\delta)$, and using
    \cref{prop:expansionu2}, yields the desired upper bound.
\end{proof}

\subsection{Radon-Nikodym density estimates}

\begin{proposition}
    \label{eq:averageRiemannTensor}
    Let $x \in M$ and let $v \in \T_x(M)$ be a unit future-directed timelike vector. The average of $w \mapsto \langle R(v, w)w, v \rangle$ over $\tilde D_\epsilon(x,v)$ is given by
    \[
        \fint_{\tilde D_\epsilon(x, v)} \langle R(v, w)w, v \rangle \diff w = \epsilon^2 \frac{n}{(n + 1)(n + 2)} \mathrm{Ric}(v, v).
    \]
\end{proposition}

\begin{proof}
    For $w \in \T_x(M)$, we define the bilinear form $B(w) \coloneqq \langle R(v, w)w, v \rangle$. We choose an orthonormal basis $e_0, \dots, e_{n - 1}$ of $\T_x(M)$ such that, for $1\le i,j\le n-1$,
    \[
        e_0=v,\qquad \langle e_0,e_0\rangle=-1,\qquad \langle e_i,e_j \rangle=\delta_{ij}.
    \]
    Every vector $w \in \T_x(M)$ can be written uniquely as $w = t v + z$, for unique $t \in \mathbb{R}$ and $z \in v^\perp$. We write $|z|$ for the Euclidean norm induced on $v^\perp$, so that $\langle w, w \rangle = -t^2 + |z|^2$. In these coordinates, we may write
    \[
        \tilde J^+(0)=\{(t,z): t\ge |z|\},
        \quad \text{ and } \quad
        \tilde J^-(2\epsilon v)=\{(t,z): 2\epsilon-t\ge |z|\},
    \]
    and therefore
    \[
        \tilde D_\epsilon(x,v)
        =
        \{(t,z): 0\le t\le 2\epsilon,\ |z|\le \rho(t)\},
        \qquad \text{ where }
        \rho(t)\coloneqq\min\{t,2\epsilon-t\}.
    \]
    By bilinearity and the symmetry properties of the curvature tensor, we have
    \[
        B(t v + z) = \langle R(v, z)z, v \rangle.
    \]
    The bilinear form $B^\perp$ on $v^\perp$ defined by $B^\perp(z_1, z_2) \coloneqq \langle R(v, z_1)z_2, v \rangle$ is symmetric by the pair symmetry of the curvature tensor, and
    \[
        \mathrm{tr}(B^\perp) = \sum_{i = 1}^{n-1} B^\perp(e_i,e_i) = \sum_{i = 1}^{n-1} \langle R(v, e_i)e_i,v \rangle = \mathrm{Ric}(v, v).
    \]
    We now use the following standard Euclidean averaging formula: if $A$ is a symmetric bilinear form on $\mathbb R^m$, then
    \[
        \int_{B_r^m} A(z,z) \diff z
        =
        \frac{\omega_m r^{m+2}}{m+2} \mathrm{tr}(A),
    \]
    where $\omega_m$ denotes the Lebesgue measure of the unit ball in $\mathbb R^m$. Applying this to the $(n-1)$-dimensional Euclidean space $v^\perp$ equipped with the bilinear form $B^\perp$ gives
    \[
        \int_{|z|\le r} B^\perp(z,z) \diff z
        =
        \frac{\omega_{n-1}r^{n+1}}{n+1}\Ric_x(v,v).
    \]
    Therefore,
    \begin{align*}
        \int_{\tilde D_\epsilon(x,v)} B(w)\diff w
         & =
        \int_0^{2\epsilon}\left(\int_{|z|\le \rho(t)} B^\perp(z,z)\diff z\right)\diff t =
        \frac{\omega_{n-1}\Ric_x(v,v)}{n+1}
        \int_0^{2\epsilon}\rho(t)^{n+1}\diff t                          \\
         & = \frac{2\omega_{n-1}}{(n+1)(n+2)}\epsilon^{n+2}\Ric_x(v,v).
    \end{align*}
    Next, the volume of the diamond is
    \begin{align*}
        \Leb(\tilde D_\epsilon(x,v))
         & =
        \int_0^{2\epsilon}\Leb(B_{\rho(t)}^{n-1})\diff t =
        \omega_{n-1}\int_0^{2\epsilon}\rho(t)^{n-1}\diff t \\
         & =
        2\omega_{n-1}\int_0^\epsilon t^{n-1} \diff t  =
        \frac{2\omega_{n-1}}{n}\epsilon^n.
    \end{align*}
    Dividing the two identities yields
    \[
        \fint_{\tilde D_\epsilon(x,v)} \langle R_x(v,w)w,v\rangle_xdw
        =
        \frac{
            \frac{2\omega_{n-1}}{(n+1)(n+2)}\epsilon^{n+2}\Ric_x(v,v)
        }{
            \frac{2\omega_{n-1}}{n}\epsilon^n
        }
        =
        \epsilon^2\frac{n}{(n+1)(n+2)}\Ric_x(v,v),
    \]
    as claimed.
\end{proof}

The final goal consists in estimating $\ell_1(\mu_x,\mu_y)$ as $(\epsilon,\delta)\to 0$.
For this purpose, we first estimate the Wasserstein distances between several auxiliary natural measures, namely
\[
    \bar\mu_x \coloneqq (f_x\circ\exp_x)_\sharp\tilde\mu_x,
    \qquad
    \bar\mu_y \coloneqq (f_y\circ\exp_y)_\sharp\tilde\mu_y,
    \qquad
    T_\sharp\mu_x,
\]
where
\[
    \tilde\mu_x \coloneqq \frac{\Leb|_{\tilde D_\epsilon(x,v)}}{\Leb(\tilde D_\epsilon(x,v))},
    \qquad
    \tilde\mu_y \coloneqq \frac{\Leb|_{\tilde D_\epsilon(y,v')}}{\Leb(\tilde D_\epsilon(y,v'))},
\]
and $T$ is the transport map introduced in \eqref{eq:mapT}. More precisely, $\tilde\mu_x$ and $\tilde\mu_y$ are the normalized Lebesgue measures supported on $\tilde D_\epsilon(x,v)$ and $\tilde D_\epsilon(y,v_\delta)$, respectively. Their pushforwards under the maps $f_x\circ\exp_x$ and $f_y\circ\exp_y$ are denoted by $\bar\mu_x$ and $\bar\mu_y$, and are supported on $D_\epsilon(x,v)$ and $D_\epsilon(y,v_\delta)$, respectively, while $T_\sharp\mu_x$ is also supported on $D_\epsilon(y,v_\delta)$.

The next proposition shows that the Radon--Nikodym derivative of $\mu_x$ with respect to $\bar\mu_x$ is of the form $1+O(\epsilon^2)$, with the error term having zero mean.

\begin{proposition} \label{lem:h_bar_estimate}
    Let $\mu_x$ and $\bar\mu_x$ be defined as previously. Then
    \begin{equation}
        \label{eq:RadonmuyTsharpbarmux}
        \frac{\diff\mu_x}{\diff\bar\mu_x}(z)=1+\bar h(z),
    \end{equation}
    where $\bar h(z)=O(\epsilon^2)$ for all $z\in D_\epsilon(x,v)$, and
    \begin{equation}
        \label{eq:barhintegrates0}
        \int_{D_\epsilon(x,v)} \bar h(z)\,d\bar\mu_x(z)=0.
    \end{equation}
\end{proposition}

\begin{proof}
    We again choose an orthonormal basis $e_0, \dots, e_{n - 1}$ of $\T_x(M)$ such that, for $1\le i,j\le n-1$,
    \[
        e_0=v,\qquad \langle e_0,e_0\rangle=-1,\qquad \langle e_i,e_j \rangle=\delta_{ij},
    \]
    that we use to identify $\T_x(M)$ with $\mathbb{R}^n$. Since $\bar\mu_x=(f_x \circ \exp_x)_\sharp\tilde\mu_x$, for every integrable function
    $u:D_\epsilon(x,v)\to\mathbb R$ one has
    \[
        \int_{D_\epsilon(x,v)} u(z)\diff\bar\mu_x(z)
        =
        \frac{1}{\Leb(\tilde D_\epsilon(x,v))}
        \int_{\tilde D_\epsilon(x,v)} u(f^\epsilon_x \circ \exp_x(w))\diff \Leb(w).
    \]
    As in the proof of \cref{prop:diamond_volume_asymptotic}, the change-of-variable coordinates gives that the integral $\int_{D_\epsilon(x,v)} u(z)\diff\mu_x(z)$ equals
    \[
        \frac{1}{\vol_g(D_\epsilon(x,v))}
        \int_{\tilde D_\epsilon(x,v)}
        u(\exp_x \circ h^\epsilon_x(w)) \sqrt{|\mathrm{det}\, g_{i j}(h^\epsilon_x(w))|}
        \,|\mathrm{det} \diff_w h^\epsilon_x| \diff \Leb(w),
    \]
    where $h^\epsilon_x$ was given in \cref{eq:maph}. Therefore, for $z = \exp_x\circ h_x^\epsilon(w)$, we have
    \[
        \frac{\diff\mu_x}{\diff\bar\mu_x}(z)
        =
        \frac{\Leb(\tilde D_\epsilon(x,v))}{\vol_g(D_\epsilon(x,v))}
        \sqrt{|\mathord{\det}\, g_{ij}(h_x^\epsilon(w))|}\,
        |\mathord{\det} \diff_w h_x^\epsilon| = 1 + O(\epsilon^2),
    \]
    where the last equality follows from \cref{prop:diamond_volume_asymptotic} and its proof.

    Defining
    \[
        \bar h(z)\coloneqq\frac{d\mu_x}{d\bar\mu_x}(z)-1,
    \]
    we get
    \[
        \int_{D_\epsilon(x,v)} \bar h(z)\,d\bar\mu_x(z)
        =
        \int_{D_\epsilon(x,v)} d\mu_x(z)-\int_{D_\epsilon(x,v)} d\bar\mu_x(z)
        =0,
    \]
    since both
    $\mu_x$ and $\bar\mu_x$ are probability measures.
\end{proof}

We now carry out the analogous estimate for the Radon-Nikodym derivative of $\mu_y$ with respect to $T_\sharp\mu_x$.

\begin{proposition}
    \label{lem:h_estimate}
    Let $\mu_y,T_\sharp\mu_x$ be defined as previously. Then
    \begin{equation}
        \label{eq:RadonmuyTsharpmux}
        \frac{\diff \mu_y}{\diff T_\sharp\mu_x}(z') = 1 + h(z'),
    \end{equation}
    where $h(z') = O(\epsilon^2 \delta)$ for all $z' \in D_\epsilon(y,v_\delta)$ and
    \begin{equation*}
        \label{eq:hintegrates0}
        \int_{D_\epsilon(y,v_\delta)} h(z') \diff T_\sharp\mu_x(z') = 0
    \end{equation*}
\end{proposition}

\begin{proof}
    We choose again a Lorentz-orthonormal basis $e_0 = v,\dots,e_{n-1}$ that identifies $\T_x(M)$ with $\mathbb{R}^n$. As in the previous proof, for every integrable function
    $u:D_\epsilon(y,v_\delta)\to\mathbb R$, the integral $\int_{D_\epsilon(y,v_\delta)} u(z')\,\diff\mu_y(z')$ is equal to
    \[
        \frac{1}{\vol_g(D_\epsilon(y,v_\delta))}
        \int_{\tilde D_\epsilon(y,v_\delta)}
        u(\exp_y \circ h^\epsilon_y(w)) \sqrt{|\det g^x_{ij}(h^\epsilon_y(w))|}
        \,|\det \diff_w h^\epsilon_y| \diff \Leb(w),
    \]
    while
    \[
        \int_{D_\epsilon(y,v_\delta)} u(z') \diff T_\sharp\mu_x(z')
        =
        \int_{D_\epsilon(x,v)} u(T(z)) \diff \mu_x(z)
    \]
    can be written as
    \[
        \frac{1}{\vol_g(D_\epsilon(x,v))}
        \int_{\tilde D_\epsilon(x,v)}
        u(T\circ\exp_x \circ h^\epsilon_x(w)) \sqrt{|\mathord{\det}\, g^x_{ij}(h^\epsilon_x(w))|}
        \,|\mathord{\det} \diff_w h^\epsilon_x| \diff \Leb(w).
    \]
    Since $T \coloneqq f_y \circ \exp_{y} \circ \mathord{\parallel_\delta^0}(c) \circ \exp_x^{-1} \circ f_x^{-1}$ and $h^\epsilon_x\coloneqq\exp_x^{-1}\circ f^\epsilon_x\circ \exp_x$, the integral term becomes
    \[
        \int_{\tilde D_\epsilon(y,v_\delta)}
        u(f_y(\exp_{y}(w'))) \sqrt{|\mathord{\det}\, g^x_{ij}(h^\epsilon_x(\mathord{\parallel_\delta^0}(w)))|}
        \, |\mathord{\det} \, \diff_{w} \widehat h_{x}^{\epsilon, \delta}| \diff \Leb(w),
    \]
    where
    \[
        \widehat h_{x}^{\epsilon, \delta}
        \coloneqq
        \mathord{\parallel_0^\delta}\circ h_x^{\epsilon}\circ \mathord{\parallel_\delta^0}
        :
        \T_y(M)\to\T_y(M).
    \]
    In particular, for $z' = \exp_y\circ h_y^\epsilon(w)$, we have
    \[
        \frac{\diff \mu_y}{\diff T_\sharp\mu_x}(z') = \frac{\vol_g(D_\epsilon(y,v_\delta))}{\vol_g(D_\epsilon(x,v))} \frac{\sqrt{|\det g^x_{ij}(h^\epsilon_y(w))|}
        }{\sqrt{|\mathord{\det}\, g^x_{ij}(h^\epsilon_x(\mathord{\parallel_\delta^0}(w)))|}} \frac{|\det \diff_w h^\epsilon_y|}{|\mathord{\det} \, \diff_{w} \widehat h_{x}^{\epsilon, \delta}|}.
    \]
    Each of these three factors is of the form $1 + O(\epsilon^2 \delta)$. We explain this only for the first one, since the other two are handled in the same way. By \cref{prop:diamond_volume_asymptotic}, both $\vol_g(D_\epsilon(y,v_\delta))$ and $\vol_g(D_\epsilon(x,v))$ are of the form $\Leb(\tilde D_\epsilon(x,v)) (1 + O(\epsilon^2)) = \Leb(\tilde D_\epsilon(y,v_\delta)) (1 + O(\epsilon^2))$. Hence their ratio is $1 + O(\epsilon^2)$. On the other hand,
    \[
        \frac{\vol_g(D_\epsilon(y,v_\delta))}{\vol_g(D_\epsilon(x,v))} \to 1, \qquad \text{ as $\delta \to 0$}.
    \]
    Therefore,
    \[
        \frac{\vol_g(D_\epsilon(y,v_\delta))}{\vol_g(D_\epsilon(x,v))} - 1
    \]
    is both $O(\epsilon^2)$ and $O(\delta)$, and hence $O(\epsilon^2 \delta)$. This establishes the estimate $h(z') = O(\epsilon^2 \delta)$. The argument showing that $h$ integrates to 0 is identical to that in the previous proposition.
\end{proof}

\subsection{Asymptotics of the Lorentz-Wasserstein distance}

In this section, we prove \cref{eq:MainTheorem}. The argument proceeds by showing, using the density estimates from \cref{lem:h_bar_estimate,lem:h_estimate}, that the quantity $\ell_1(\mu_x,\mu_y)$ coincides with $\ell_1(\bar\mu_x, T_\sharp \bar\mu_x)$ up to the relevant order. We bound $\ell_1(\bar\mu_x, T_\sharp \bar\mu_x)$ from below via \cref{prop:tau_f_lowerbound}, and from above by applying Kantorovich duality to the function $u$, for which \cref{prop:expansionu2} was established.

\begin{proposition}
    \label{prop:wasserstein_estimate_tilde}
    As $(\epsilon,\delta)\to(0,0)$, one has
    \[
        \int \uptau(x',T(x'))\diff\bar\mu_x(x')
        =
        \delta\left(
        1+\frac{\epsilon^2}{2}\frac{n}{(n+1)(n+2)}\Ric(v,v)
        + O(\epsilon^3+\epsilon^2\delta)
        \right),
    \]
    and
    \[
        \int u \diff T_\sharp\bar\mu_x-\int u \diff\bar\mu_x
        =
        \delta\left(
        1+\frac{\epsilon^2}{2}\frac{n}{(n+1)(n+2)}\Ric(v,v)
        + O(\epsilon^3+\epsilon^2\delta)
        \right),
    \]
    where $T$ and $u$ are the maps introduced in \cref{eq:mapT} and \cref{def:functionudef}, respectively.

    Therefore, it holds that
    \[
        \ell_1(\bar \mu_x, T_\sharp \bar \mu_x) = \delta\left(
        1+\frac{\epsilon^2}{2}\frac{n}{(n+1)(n+2)}\Ric(v,v)
        + O(\epsilon^3+\epsilon^2\delta)
        \right).
    \]
\end{proposition}

\begin{proof}
    We have $T(f_x\circ\exp_x(\epsilon w)) = f_y\circ\exp_y(\epsilon w')$, and, by the definition of $\bar\mu_x$, it follows that
    \begin{align*}
        \int \uptau(x',T(x'))\,\diff\bar\mu_x(x') & = \fint_{\tilde D_x(1,v)}
        \uptau(f_x\circ\exp_x(\epsilon w),f_y\circ\exp_y(\epsilon w_\delta))\,\diff w \\
                                                  & = \fint_{\tilde D_x(1,v)}
        \delta \left( 1 + \frac{\epsilon^2}{2}\langle R(v,w)w,v\rangle + O(\epsilon^3+\epsilon^2\delta)\right) \diff w,
    \end{align*}
    where the second equality follows from \cref{prop:tau_f_lowerbound}. Likewise,
    \begin{align*}
        \int u \diff T_\sharp\bar\mu_x-\int u \diff\bar\mu_x & = \fint_{\tilde D_x(1,v)} \left[u(f_y\circ\exp_y(\epsilon w_\delta)) - u(f_x\circ\exp_x(\epsilon w)) \right] \diff w \\
                                                             & = \fint_{\tilde D_x(1,v)}
        \delta \left( 1 + \frac{\epsilon^2}{2}\langle R(v,w)w,v\rangle + O(\epsilon^3+\epsilon^2\delta)\right)\,\diff w
    \end{align*}
    where the second equality follows from \cref{prop:expansionu2}. The first part of the statement is then proven by \cref{eq:averageRiemannTensor}.

    To conclude, we recall that, by definition, the Lorentz-Wasserstein distance $\ell_1$ is a supremum over all admissible transport plans between $\bar \mu_x$ and $T_\sharp \bar \mu_x$. Since the probability measure $(\mathrm{Id}, T)_\sharp \bar \mu_x$ is clearly such a plan, it follows that $\ell_1(\bar \mu_x, T_\sharp \bar \mu_x) \geq \int\uptau(x',T(x'))\,\diff\bar\mu_x(x')$. Conversely, for $\epsilon$ and $\delta$ sufficiently small, the union of supports $\supp(\bar\mu_x)\cup\supp(T_\sharp\bar\mu_x)$ is contained in $J^+(\mathcal V)\cap\mathcal N$, where $u$ is smooth and 1-steep. Kantorovich Duality \cref{theorem:KantorovichDuality} yields $\ell_1(\bar \mu_x, T_\sharp \bar \mu_x) \leq \int u \diff T_\sharp\bar\mu_x-\int u \diff\bar\mu_x$.
\end{proof}

We now compare $\ell_1(\mu_x, T_\sharp \mu_x)$ and $\ell_1(\bar \mu_x, T_\sharp \bar \mu_x)$.

\begin{proposition}
    \label{prop:Wasserstein_T_to_tilde}
    As $(\epsilon, \delta) \to (0, 0)$, it holds that
    \[
        \ell_1(\mu_x, T_\sharp \mu_x) = \ell_1(\bar \mu_x, T_\sharp \bar \mu_x) + O(\epsilon^3 \delta + \epsilon^2 \delta^2).
    \]
\end{proposition}

\begin{proof}
    The map $T$ defined in \cref{eq:mapT} is trivially a transport map from $\mu_x$ to $T_\sharp \mu_x$. Therefore, by the definition of $\ell_1$,
    \begin{align*}
        \ell_1(\mu_x, T_\sharp \mu_x) & \geq \int \uptau(z, T(z)) \diff \mu_x(z) = \int \uptau(z, T(z)) \frac{\diff \mu_x}{\diff \bar \mu_x}(z) \diff \bar \mu_x(z)                  \\
                                      & = \int \uptau(z, T(z)) \diff \mu_x(z) + \int \uptau(z, T(z)) \bar h(z) \diff \bar \mu_x(z)                                                   \\
                                      & = \ell_1(\bar \mu_x, T_\sharp \bar \mu_x) + \int \uptau(z, T(z)) \bar h(z) \diff \bar \mu_x(z) + O(\epsilon^3 \delta + \epsilon^2 \delta^2),
    \end{align*}
    where we used \cref{eq:RadonmuyTsharpbarmux,prop:wasserstein_estimate_tilde}. We now deduce that $\ell_1(\mu_x, T_\sharp \mu_x) \geq \ell_1(\bar \mu_x, T_\sharp \bar \mu_x) + O(\epsilon^3 \delta + \epsilon^2 \delta^2)$ by observing that
    \begin{align*}
        \int \uptau(z, T(z)) \bar h(z) \diff \bar \mu_x(z) & = \int \big(\uptau(z, T(z)) - \delta\big) \bar h(z) \diff \bar \mu_x(z) &  & \text{by \cref{eq:barhintegrates0}}    \\
                                                           & = \int O(\delta \epsilon^2) \bar h(z) \diff T_\sharp \bar \mu_x(z)      &  & \text{by \cref{prop:tau_f_lowerbound}} \\
                                                           & = \int O(\delta \epsilon^2) O(\epsilon^2) \diff \bar \mu_x(z)           &  & \text{by \cref{lem:h_bar_estimate}}.
    \end{align*}
    For the reverse inequality, we use the 1-steep function $u$ defined in \cref{def:functionudef} together with Kantorovich's duality formula:
    \begin{align*}
        \ell_1(\mu_x, T_\sharp \mu_x) & \leq \int u \diff T_\sharp \mu_x(z) - \int u \diff \mu_x(z) = \int \big(u(T(z)) - u(z)\big) \diff \mu_x(z)                                            \\
                                      & = \int \big(u(T(z)) - u(z)\big) \frac{\diff \mu_x}{\diff \bar \mu_x}(z) \diff \bar \mu_x(z)                                                           \\
                                      & = \int \big(u(T(z)) - u(z)\big) \diff \bar \mu_x(z) + \int \big(u(T(z)) - u(z)\big) \bar{h}(z) \diff \bar \mu_x(z)                                    \\
                                      & = \ell_1(\bar \mu_x, T_\sharp \bar \mu_x) + \int \big(u(T(z)) - u(z)\big) \bar{h}(z) \diff \bar \mu_x(z) + O(\epsilon^3 \delta + \epsilon^2 \delta^2)
    \end{align*}
    again by \cref{eq:RadonmuyTsharpbarmux,prop:wasserstein_estimate_tilde}. Arguing exactly as above, \cref{eq:barhintegrates0,prop:expansionu2,lem:h_bar_estimate} implies that $u(T(z)) - u(z) - \delta = O(\delta \epsilon^2)$, $\bar \mu_x(z) = O(\epsilon^2)$, and therefore $\ell_1(\mu_x, T_\sharp \mu_x) \leq \ell_1(\bar \mu_x, T_\sharp \bar \mu_x) + O(\epsilon^3 \delta + \epsilon^2 \delta^2)$, which completes the proof.
\end{proof}

We then compare $\ell_1(\mu_x, \mu_y)$ and $\ell_1(\mu_x, T_\sharp \mu_x)$.

\begin{proposition}
    \label{prop:Wasserstein_real_to_T}
    As $(\epsilon, \delta) \to (0, 0)$, it holds that
    \[
        \ell_1(\mu_x, \mu_y) = \ell_1(\mu_x, T_\sharp \mu_x) + O(\epsilon^3 \delta + \epsilon^2 \delta^2).
    \]
\end{proposition}

\begin{proof}
    By the Kantorovich duality formula, see \cref{theorem:KantorovichDuality}(i), we have
    \begin{equation}
        \label{eq:KDformuxmuy}
        \ell_1(\mu_x, \mu_y) = \inf \left\{\int u\diff\mu_y - \int u\diff\mu_x\right\},
    \end{equation}
    where the infimum is taken over all 1-steep functions $u\colon E\to \mathbb{R}$ belonging to $L^1(M,\mu_x)\cap L^1(M,\mu_y)$ and defined on some Borel set $E \subseteq M$ containing $\supp(\mu_x)\cup\supp(\mu_y)$.

    Moreover, by \cref{theorem:KantorovichDuality}(ii), the corresponding infimum for $\ell_1(\mu_x, T_\sharp \mu_x)$ is attained by some 1-steep function $u : E \to \R$ in
    \[
        L^1(M,\mu_x)\cap L^1(M,T_\sharp\mu_x) = L^1(M,\mu_x)\cap L^1(M,\mu_y),
    \]
    defined on some Borel set $E \subseteq M$ containing $\supp(\mu_x)\cup\supp(\mu_y)$. Thus
    \[
        \ell_1(\mu_x, T_\sharp \mu_x ) = \int u \diff T_\sharp \mu_x - \int u\diff\mu_x
    \]
    for this specific function $u$. Using \cref{eq:KDformuxmuy}, we then obtain
    \begin{align*}
        \ell_1(\mu_x, \mu_y) & \leq \int u\diff\mu_y - \int u\diff\mu_x = \int u \frac{\diff \mu_y}{\diff T_\sharp \mu_x} \diff T_\sharp \mu_x - \int u\diff\mu_x                     \\
                             & = \int u \diff T_\sharp \mu_x - \int u\diff\mu_x + \int u \, h \diff T_\sharp \mu_x = \ell_1(\mu_x, T_\sharp \mu_x ) + \int u\,h \diff T_\sharp \mu_x,
    \end{align*}
    where $h$ is the function introduced in \cref{eq:RadonmuyTsharpmux}. By \cref{lem:h_estimate}, we know that $h(z') = O(\epsilon^2\delta)$ for all $z' \in D_\epsilon(y, v_\delta)$, while $\int h(z') \diff T_\sharp\mu_x(z') = 0$. We choose the fixed point $\bar z=y\in D_\epsilon(y,v_\delta)$. Then
    \begin{equation}
        \label{eq:addingzero}
        \int u(\bar z) h(z') \diff T_\sharp\mu_x(z') = u(\bar z) \int h(z') \diff T_\sharp\mu_x(z') = 0.
    \end{equation}

    As recalled in \cref{eq:utauconcave}, the fact that $u$ is a minimizer for the dual transport cost means that it can be represented as
    \[
        u(z) = \sup_{x' \in D_\epsilon(x,v)} \bigl(\phi(x') + \uptau(x',z)\bigr),
        \qquad \text{for all } z \in D_\epsilon(x,v) \cup D_\epsilon(y,v_\delta),
    \]
    for some $\uptau$-concave function $\phi \colon D_\epsilon(x,v) \to \mathbb{R}$. Consequently, for every $z,\bar z \in D_\epsilon(x,v) \cup D_\epsilon(y,v_\delta)$,
    \begin{align*}
        |u(z)-u(\bar z)|
         & = \abs{\sup_{x' \in D_\epsilon(x,v)} \bigl(\phi(x')+\uptau(x',z)\bigr)
                 - \sup_{x' \in D_\epsilon(x,v)} \bigl(\phi(x')+\uptau(x',\bar z)\bigr)}        \\
         & \le \sup_{x' \in D_\epsilon(x,v)}
        \left| \bigl(\phi(x')+\uptau(x',z)\bigr)-\bigl(\phi(x')+\uptau(x',\bar z)\bigr) \right| \\
         & = \sup_{x' \in D_\epsilon(x,v)} |\uptau(x',z)-\uptau(x',\bar z)|.
    \end{align*}
    Now let $x' \in D_\epsilon(x,v)$ and $z' \in D_\epsilon(y,v_\delta)$. Since $\delta>2\epsilon$, we have $x' \le c(2\epsilon)\le y \le z'$, and therefore
    \[
        \uptau(x',z')
        \ge \uptau(x',c(2\epsilon))+\uptau(c(2\epsilon),y)+\uptau(y,z')
        \ge \delta-2\epsilon.
    \]
    On the other hand, as $x \le x' \le z' \le c(\delta+2\epsilon)$, we also have
    \[
        \delta+2\epsilon
        = \uptau(x,c(\delta+2\epsilon))
        \ge \uptau(x,x')+\uptau(x',z')+\uptau(z',c(\delta+2\epsilon))
        \ge \uptau(x',z'),
    \]
    and therefore $\uptau(x',z') \in [\delta - 2 \epsilon, \delta + 2\epsilon]$. The same estimate applies to $\uptau(x',y)$, since $y\in D_\epsilon(y,v_\delta)$. In particular, we have shown that
    \begin{equation}
        \label{eq:estimatingu}
        |u(z') - u(\bar z) | \leq 4 \epsilon, \quad \text{for all $z' \in D_\epsilon(y, v_\delta)$}.
    \end{equation}
    Combining \cref{eq:addingzero,eq:estimatingu} with \cref{lem:h_estimate}, we infer that
    \begin{align*}
        \int u\,h \diff T_\sharp \mu_x & = \abs{\int (u(z')-u(\bar z)) h(z') \diff T_\sharp \mu_x(z')} \leq 4 \epsilon \int |h(z')| \diff T_\sharp \mu_x(z') \leq O(\epsilon^3 \delta).
    \end{align*}
    This proves that
    \[
        \ell_1(\mu_x, \mu_y) \leq \ell_1(\mu_x, T_\sharp \mu_x ) + O(\epsilon^3 \delta).
    \]
    The other inequality is obtained in a similar fashion, starting instead from $\ell_1(\mu_x, T_\sharp \mu_x )$ and using a 1-steep function that minimizes the dual cost for $\ell_1(\mu_x, \mu_y)$.
\end{proof}

At this point, we have finally established \cref{eq:MainTheorem}.

\begin{theorem}
    \label{thm:smoothTheorem}
    Let $(M,g)$ be a $n$-dimensional globally hyperbolic Lorentzian manifold with time-separation function $\uptau$. Consider $x \in M$ and $v \in \T_x(M)$ a future-directed timelike tangent vector. If $0 < 2 \epsilon < \delta$,
    \[
        y \coloneqq \exp_x(\delta v), \qquad x_{2 \epsilon} \coloneqq \exp_x(2\epsilon v), \quad \text{and} \quad y_{2 \epsilon} \coloneqq \exp_x\left((2\epsilon+\delta) v\right).
    \]
    then the following asymptotic formula holds as $\epsilon, \delta \rightarrow 0$,
    \[
        \ell_1(\mu_x, \mu_y) = \delta \left( 1+ {\frac{\epsilon^2}{2}} \frac{n}{(n+1)(n+2)}\mathrm{Ric}(v,v) + O(\epsilon^3 + \epsilon^2 \delta)\right),
    \]
    where $\mu_x$ is the uniform measure on $J(x,x_{2\epsilon})$ and $\mu_y$ is the uniform measure on $J(y,y_{2\epsilon})$.
\end{theorem}

\begin{proof}
    We have found that
    \begin{align*}
        \ell_1(\mu_x, \mu_y) & = \ell_1(\mu_x, T_\sharp \mu_x) + O(\epsilon^3 \delta + \epsilon^2 \delta^2)           &                                                   & \text{by \cref{prop:Wasserstein_real_to_T}}  \\
                             & = \ell_1(\bar \mu_x, T_\sharp \bar \mu_x) + O(\epsilon^3 \delta + \epsilon^2 \delta^2) &                                                   & \text{by \cref{prop:Wasserstein_T_to_tilde}} \\
                             & = \delta\left(
        1+\frac{\epsilon^2}{2}\frac{n}{(n+1)(n+2)}\Ric(v,v)
        + O(\epsilon^3+\epsilon^2\delta)
        \right)              &                                                                                        & \text{by \cref{prop:wasserstein_estimate_tilde}},
    \end{align*}
    which is the desired asymptotic.
\end{proof}

\startappendix
\section{Geodesic deviation in Lorentzian geometry}

The energy between two points $x, y$ in a convex neighbourhood $\mathcal{U}$ is defined as
\[
    \E(x, y) = \frac{1}{2} \int_0^1 \langle \dot \gamma(t), \dot \gamma(t)\rangle \diff t,
\]
where $\gamma : \interval{0}{1} \to M$ is the unique geodesic contained in $U$ joining $x$ to $y$. A sketch of the fourth-order part of the following expansion is given in \cite[Lemma 4.12]{Agrachev2018} in the Riemannian setting, together with references indicating the remaining arguments. For the reader's convenience, we present here a full proof, including the terms of order five.

\begin{theorem} \label{Theorem:Energy_expansion}
    Let $x \in M$, $v, w \in \T_x(M)$, and consider the geodesics $\gamma_v(t) \coloneqq \exp_x(t v)$ and $\gamma_w(s) \coloneqq \exp_x(s w)$. Then, it holds, as $(t, s) \to (0, 0)$, that
    \begin{equation}
        \label{eq:ExpansionEnergy}
        \begin{aligned}
            \E(\gamma_v(t), \gamma_w(s)) ={} & \frac{1}{2} \langle tv - sw, tv - sw \rangle
            + \frac{1}{6} \langle R(v, w)v, w \rangle \, t^2 s^2                                                                                                         \\
                                             & + \frac{1}{24} \left( t \langle (\nabla_v R)(v, w)v, w \rangle + s \langle (\nabla_w R)(v, w)v, w \rangle \right) t^2 s^2 \\
                                             & + t^2 s^2 O((|t| + |s|)^2).
        \end{aligned}
    \end{equation}
    where the $O$-term is uniform for $v, w$ ranging in compact subsets of $\T_x(M)$.
\end{theorem}

\begin{proof}
    Let $t, s > 0$ be small enough, and fix a convex normal neighbourhood of $x$ such that the unique geodesic joining $\gamma_v(t)$ to $\gamma_w(s)$ is given by
    \[
        c_{s, t}(u) \coloneqq \exp_{\gamma_w(s)}[u \ \exp_{\gamma_w(s)}^{-1}(\gamma_v(t))] = \exp_{\gamma_v(t)}[(1 - u) \ \exp_{\gamma_v(t)}^{-1}(\gamma_w(s))], \quad u \in \interval{0}{1},
    \]
    which means that $\E(\gamma_v(t), \gamma_w(s)) = \tfrac{1}{2} \int_0^1 \langle \dot c_{s, t}(u), \dot c_{s, t}(u) \rangle \diff u$.
    Note that $\langle \dot c_{s, t}(u), \dot c_{s, t}(u)\rangle$ is constant in $u$, since
    \[
        \frac{\diff}{\diff u} \langle\dot c_{s, t}(u), \dot c_{s, t}(u)\rangle = 2 \Big\langle\frac{\mathrm D}{\diff u} \dot c_{s, t}(u), \dot c_{s, t}(u)\Big\rangle = 0
    \]
    and $\frac{\mathrm D}{\diff u} \dot c_{s, t}(u) = 0$, as $u \mapsto c_{s, t}(u)$ is a geodesic. Moreover, since $\exp_{\gamma_w(s)}^{-1}(x) = - s \Gamma(\gamma_w)_0^s [w]$, we obtain
    \begin{equation}
        \label{eq:ExpansionEnergy10}
        \langle\dot c_{s, 0}(u), \dot c_{s, 0}(u)\rangle = \langle\dot c_{s, 0}(0), \dot c_{s, 0}(0)\rangle = s^2 \langle \Gamma(\gamma_w)_0^s [w], \Gamma(\gamma_w)_0^s [w] \rangle = s^2 \langle w, w \rangle.
    \end{equation}
    Similarly, we have
    \begin{equation}
        \label{eq:ExpansionEnergy01}
        \langle\dot c_{0, t}(u), \dot c_{0, t}(u)\rangle = t^2 \langle v, v \rangle
    \end{equation}
    Next, we compute
    \[
        \frac{\partial}{\partial s} \E(\gamma_v(t), \gamma_w(s)) = \int_0^1 \Big\langle  \frac{\mathrm D}{\diff s} \dot c_{s, t}(u), \dot c_{s, t}(u)\Big\rangle \diff u,
    \]
    evaluated at $(s, t) = (0, t)$. We have
    \[
        \Big\langle  \frac{\mathrm D}{\diff s} \dot c_{s, t}(u), \dot c_{s, t}(u)\Big\rangle = \Big\langle  \frac{\mathrm D}{\diff u} \frac{\partial}{\partial s} c_{s, t}(u), \dot c_{s, t}(u)\Big\rangle = \frac{\partial}{\partial u} \Big\langle \frac{\partial}{\partial s} c_{s, t}(u), \dot c_{s, t}(u)\Big\rangle,
    \]
    where in the last equality we have used again that $u \mapsto c_{s, t}(u)$ is a geodesic. In particular, we can write that
    \[
        \frac{\partial}{\partial s} \E(\gamma_v(t), \gamma_w(s)) = \Big\langle \frac{\partial}{\partial s} c_{s, t}(1), \dot c_{s, t}(1)\Big\rangle - \Big\langle \frac{\partial}{\partial s} c_{s, t}(0), \dot c_{s, t}(0)\Big\rangle.
    \]
    Clearly, $c_{s, t}(1) = \gamma_v(t)$ does not depend on $s$, while $\dot c_{0, t}(0) = t v$ and $\tfrac{\partial}{\partial s} c_{s, t}(0) |_{s = 0} = \dot \gamma_w(0) = w$. Thus, we obtain
    \begin{equation}
        \label{eq:ExpansionEnergy1n}
        \frac{\partial}{\partial s} \E(\gamma_v(t), \gamma_w(s)) \big|_{(s, t) = (0, t)} = -t \langle v, w \rangle.
    \end{equation}
    Similarly, we also have
    \begin{equation}
        \label{eq:ExpansionEnergyn1}
        \frac{\partial}{\partial t} \E(\gamma_v(t), \gamma_w(s)) \big|_{t = 0} = -s \langle v, w \rangle.
    \end{equation}
    Let
    \[
        F(t, s) \coloneqq \E(\gamma_v(t), \gamma_w(s)) - \frac{1}{2} \langle tv - sw, tv - sw \rangle.
    \]
    With \cref{eq:ExpansionEnergy10,eq:ExpansionEnergy01,eq:ExpansionEnergy1n,eq:ExpansionEnergyn1}, we have
    \[
        F(t, 0) = F(0, s) = \frac{\partial F}{\partial s}(t, 0) = \frac{\partial F}{\partial t}(0, s) = 0.
    \]
    Since $F$ is smooth, Hadamard's lemma applied twice yields a smooth function $G$ such that $F(t, s) = t^2 s^2 G(t, s)$.
    We finally consider
    \begin{align*}
        \frac{\partial^2}{\partial s^2} \E(\gamma_v(t), \gamma_w(s)) ={} & \Big\langle \frac{\mathrm D}{\diff s} \frac{\partial}{\partial s} c_{s, t}(1), \dot c_{s, t}(1)\Big\rangle + \Big\langle \frac{\partial}{\partial s} c_{s, t}(1), \frac{\mathrm D}{\diff s} \dot c_{s, t}(1)\Big\rangle   \\
                                                                         & - \Big\langle \frac{\mathrm D}{\diff s} \frac{\partial}{\partial s} c_{s, t}(0), \dot c_{s, t}(0)\Big\rangle - \Big\langle \frac{\partial}{\partial s} c_{s, t}(0), \frac{\mathrm D}{\diff s} \dot c_{s, t}(0)\Big\rangle
    \end{align*}
    evaluated at $(s, t) = (0, t)$.
    We know that $\dot c_{0, t}(0) = t v$, $c_{s, t}(1) = \gamma_v(t)$, and $c_{s, t}(0) = \gamma_w(s)$, so that we have
    \begin{equation}
        \label{eq:dds2}
        \frac{\partial^2}{\partial s^2} \E(\gamma_v(t), \gamma_w(s))\big|_{(s, t) = (0, t)} = - \Big\langle w, \frac{\mathrm D}{\diff s} \dot c_{s, t}(0) \big|_{s = 0} \Big\rangle = - \Big\langle w, \frac{\mathrm D}{\diff u} V_t(0) \Big\rangle,
    \end{equation}
    where $V_t(u) \coloneqq (\partial/\partial s) c_{s, t} (u) |_{s = 0}$ is a Jacobi field satisfying
    \begin{equation}
        \label{eq:JacobiEquation}
        \frac{\mathrm D^2}{\diff u^2} V_t(u) + R(V_t(u), \dot c_{0, t}(u))\dot c_{0, t}(u) = 0
    \end{equation}
    with $V_t(0) = w$ and $V_t(1) = 0$. Denote by $\mathord{\parallel_0^t}(u)$ the parallel transport along $t \mapsto c_{0, t}(u)$, and consider $\tilde V_t(u) \coloneqq \mathord{\parallel_t^0}(u) V_t(u) \in \T_x(M)$. We can write, as $t \to 0$,
    \[
        \tilde V_t(u) = \tilde V_0(u) + t \tilde V'_0(u) + \frac{t^2}{2} \tilde V''_0(u) + \frac{t^3}{6} \tilde V'''_0(u) + O(t^4).
    \]
    We denote by $(\dot{\ })$ the derivative with respect to $u$ and by $()'$ the derivative with respect to $t$. Applying the same parallel transport to the Jacobi equation \cref{eq:JacobiEquation} yields
    \begin{equation}
        \label{eq:JacobiEquation2}
        \ddot{\tilde{V}}_t(u) + \tilde R_t(u) \tilde V_t(u) = 0
    \end{equation}
    where, for $Z \in \T_x(M)$, we denote
    \[
        \tilde R_t(u)Z \coloneqq \mathord{\parallel_t^0}(u) \left[ R(\mathord{\parallel_0^t}(u) Z, \dot c_{0, t}(u))\dot c_{0, t}(u)\right].
    \]
    Translating the initial conditions gives
    \[
        \tilde V_t(0) = \mathord{\parallel_t^0}(0) w = \mathrm{Id}_{\T_x(M)} (w) = w, \quad \tilde V_t(1) = \mathord{\parallel_t^0}(1) 0 = 0.
    \]
    In particular, we have
    \begin{equation}
        \label{eq:IVJacobiFields}
        \tilde V_t'(0) = \tilde V_t''(0) = \tilde V_t'''(0) = \tilde V_t'(1) = \tilde V_t''(1) = \tilde V_t'''(1) = 0, \quad \text{ for all } t \in \interval{0}{1}.
    \end{equation}
    Since it acts on a fixed vector space, we can also expand $\tilde R_t(u)$ in $t$:
    \[
        \tilde R_t(u) = \tilde R_0(u) + t \tilde R'_0(u) + \frac{t^2}{2} \tilde R''_0(u) + \frac{t^3}{6} \tilde R'''_0(u) + O(t^4).
    \]
    We also have $c_{0, t}(u) = \gamma_v(u t)$, and thus we can write, for $Z \in \T_x(M)$,
    \begin{align*}
        \tilde R_t(u)Z ={} & t^2 \mathord{\parallel_t^0}(u) \big[ R(\mathord{\parallel_0^t}(u) Z, \dot \gamma_v(t u))\dot \gamma_v(t u)\big].
    \end{align*}
    In particular, we have $\tilde R_0(u) = \tilde R'_0(u) = 0$, and $\tilde R''_0(u)Z = 2  R_x(Z, v)v$. Along the curve $t \mapsto c_{0, t}(u) = \gamma_v(t u)$, the vector fields $\mathord{\parallel_0^t}(u) Z$ and $\dot \gamma_v(t u)$ are both parallel. Using the product rule for the Riemann curvature tensor, we obtain
    \begin{equation*}
        \tilde R'''_0(u)Z
        = 6 \frac{\mathrm D}{\diff t} \Big[ R(\mathord{\parallel_0^t}(u) Z, \dot \gamma_v(t u))\dot \gamma_v(t u) \Big]_{t = 0} =6u (\nabla_v R)_x(Z, v)v,
    \end{equation*}
    where the factor $u$ comes from the velocity of the curve $t \mapsto \gamma_v(tu)$.
    Substituting this into the Jacobi equation \cref{eq:JacobiEquation2} gives
    \[
        \ddot{\tilde{V}}_0(u) = 0, \quad \ddot{\tilde{V}}'_0(u) = 0,\qquad \ddot{\tilde{V}}''_0(u) = -2 R(\tilde V_0(u), v)v,
    \]
    and
    \[
        \ddot{\tilde V}'''_0(u) = -6 \left[ R(\tilde V'_0(u), v)v + u(\nabla_v R)_x(\tilde V_0(u), v)v \right].
    \]
    These differential equations can be solved explicitly using \cref{eq:IVJacobiFields}, and we obtain
    \[
        \tilde V_0(u) = w (1 - u), \quad \tilde V'_0(u) = 0, \quad \tilde V''_0(u) = \frac{1}{3} u (u - 1) (u - 2) R_x(w, v)v,
    \]
    and
    \[
        \tilde V'''_0(u) = \left(\frac{u}{2} - u^3 + \frac{u^4}{2}\right) (\nabla_v R)_x(w, v)v.
    \]
    In particular, this shows that
    \[
        \dot{\tilde V}_t(0) = - w + \frac{t^2}{3} R_x(w, v)v + \frac{t^3}{12} (\nabla_v R)_x(w, v)v + O(t^4).
    \]
    We deduce from \cref{eq:dds2} that
    \begin{align*}
        \frac{\partial^2}{\partial s^2} \E(\gamma_v(t), \gamma_w(s)) & \big|_{(s, t) = (0, t)} = - \langle \mathord{\parallel_0^t}(0) w, \dot{\tilde V}_t(0) \rangle = - \langle w, \dot{\tilde V}_t(0) \rangle \\
        ={}                                                          & \langle w, w \rangle - \frac{t^2}{3} \langle R_x(w, v)v, w \rangle - \frac{t^3}{12} \langle(\nabla_v R)_x(w, v)v, w \rangle + O(t^4).
    \end{align*}
    Since $\partial_s^2 F(t, 0) = 2 t^2 G(t, 0)$, we deduce that
    \[
        G(t, 0)
        = \frac{1}{6} \langle R_x(v, w)v, w \rangle
        + \frac{t}{24} \langle(\nabla_v R)_x(v, w)v, w \rangle
        + O(t^2).
    \]
    By the symmetry of the energy under interchange of its endpoints, followed by the symmetries of $R$ and $\nabla R$, we likewise have
    \[
        G(0, s)
        = \frac{1}{6} \langle R_x(v, w)v, w \rangle
        + \frac{s}{24} \langle(\nabla_w R)_x(v, w)v, w \rangle
        + O(s^2).
    \]
    Since $G$ is smooth, its first-order Taylor expansion at $(0,0)$ therefore yields
    \begin{align*}
        G(t, s) ={} & \frac{1}{6} \langle R_x(v, w)v, w \rangle
        + \frac{t}{24} \langle(\nabla_v R)_x(v, w)v, w \rangle               \\
                    & + \frac{s}{24} \langle(\nabla_w R)_x(v, w)v, w \rangle
        + O((|t| + |s|)^2).
    \end{align*}
    Recalling that $F(t, s) = t^2 s^2 G(t, s)$ gives \cref{eq:ExpansionEnergy}. The smooth dependence of all the quantities above on $(v,w)$ makes the remainder uniform when $v$ and $w$ range in compact subsets of $\T_x(M)$.
\end{proof}

\printbibliography[heading=bibintoc]

\end{document}